\documentclass[preprint,12pt]{elsarticle}

\usepackage{fullpage}

\usepackage{amssymb}
\usepackage{amsfonts}
\usepackage{amsmath}
\usepackage{amsthm}
\usepackage{bm}
\usepackage{color}      
\usepackage{enumerate}
\usepackage{graphicx}
\usepackage{subfig}
\usepackage[notref,notcite]{showkeys}
\usepackage{url}
\usepackage{cleveref}

\newtheorem{thm}{Theorem}
\newtheorem{lem}[thm]{Lemma}

\newtheorem{cor}[thm]{Corollary}
\newtheorem{remark}[thm]{Remark}

\newcommand{\tr}{\mathrm{tr}}        
\newcommand{\dist}{\mathrm{dist}}       
\newcommand{\tp}{^{T}} 
\newcommand{\dij}{\delta_{ij}}

\newcommand{\va}{\mathbf{a}}
\newcommand{\vb}{\mathbf{b}}

\newcommand{\vu}{\mathbf{u}}
\newcommand{\vv}{\mathbf{v}}
\newcommand{\vw}{\mathbf{w}}

\newcommand{\vq}{\mathbf{q}}
\newcommand{\ve}{\mathbf{e}}
\newcommand{\vQ}{\mathbf{Q}}
\newcommand{\vr}{\mathbf{r}}

\newcommand{\vm}{\mathbf{m}}
\newcommand{\vn}{\mathbf{n}}
\newcommand{\vt}{\mathbf{t}}

\newcommand{\vx}{\mathbf{x}}

\newcommand{\vnu}{\bm{\nu}}

\newcommand{\vA}{\mathbf{A}}

\newcommand{\vD}{\mathbf{D}}
\newcommand{\vE}{\mathbf{E}}

\newcommand{\vI}{\mathbf{I}}
\newcommand{\vR}{\mathbf{R}}

\newcommand{\vF}{\mathbf{F}}

\newcommand{\vzero}{\mathbf{0}}

\newcommand{\Om}{\Omega}
\newcommand{\dOm}{\partial \Omega}
\newcommand{\Gm}{\Gamma}
\newcommand{\Oc}{\Om_{\mathrm{c}}} 
\newcommand{\Ochat}{\hat{\Om}_{\mathrm{c}}} 
\newcommand{\bdys}{\Gamma_{s}}
\newcommand{\bdyvn}{\Gamma_{\vn}}
\newcommand{\bdyvu}{\Gamma_{\vu}}

\newcommand{\iO}{\int_{\Omega}}

\newcommand{\X}{\mathbb{X}}
\newcommand{\Hbdy}[1]{H^1_{#1}(\Omega)}   
\newcommand{\Sh}{\mathbb{S}_h}   
\newcommand{\Nh}{\mathbb{N}_h}   
\newcommand{\Uh}{\mathbb{U}_h}   
\newcommand{\Y}{\mathbb{U}^{\perp}}   
\newcommand{\Yh}{\Y_h}   
\newcommand{\R}{\mathbb{R}}   

\newcommand{\Admis}{\mathbb{A}}   
\newcommand{\Sing}{\mathbb{S}}   

\newcommand{\dt}{\delta t}

\newcommand{\cdist}{d}
\newcommand{\cdisthat}{\hat{\cdist}}
\newcommand{\collmap}{\vF}
\newcommand{\collrot}{\vR}

\newcommand{\colltrans}{\vb}
\newcommand{\phase}{\phi}
\newcommand{\phaseref}{\phase_{\mathrm{ref}}}
\newcommand{\anchorcoef}{K_a}
\newcommand{\Ea}{E_{\mathrm{a}}}

\newcommand{\mmcoef}{\rho}

\newcommand{\ipanchor}{a_h}
\newcommand{\linanchor}{\ell_h}
\newcommand{\ipelec}{e_h}

\newcommand{\Eext}{E_{\mathrm{ext}}}
\newcommand{\Ecoef}{K_{\mathrm{ext}}}
\newcommand{\ebar}{\bar{\varepsilon}}
\newcommand{\ea}{\varepsilon_{\mathrm{a}}}
\newcommand{\ga}{\gamma_{\mathrm{a}}}

\newcommand{\Tk}{\mathcal{T}}

\newcommand{\Nk}{\mathcal{N}}

\newcommand{\coefB}{\mathrm{B}}
\newcommand{\coefC}{\mathrm{C}}
\newcommand{\coefF}{\mathrm{F}}
\newcommand{\coefG}{\mathrm{G}}
\newcommand{\coefN}{\mathrm{N}}
\newcommand{\coefQ}{\mathrm{Q}}
\newcommand{\coefS}{\mathrm{S}}
\newcommand{\coefT}{\mathrm{T}}
\newcommand{\coefV}{\mathrm{V}}
\newcommand{\coefW}{\mathrm{W}}
\newcommand{\coefY}{\mathrm{Y}}
\newcommand{\coefZ}{\mathrm{Z}}

\journal{Journal of Computational Physics}

\begin{document}

\begin{frontmatter}



\title{The Ericksen Model of Liquid Crystals \\ with Colloidal and Electric Effects}


\author[addRHN]{Ricardo H. Nochetto}
\ead{rhn@math.umd.edu}
\author[addSWW]{Shawn W. Walker}
\ead{walker@math.lsu.edu}
\author[addWZ]{Wujun Zhang}
\ead{wujun@math.rutgers.edu}

\address[addRHN]{Department of Mathematics and Institute for Physical Science and Technology, University of Maryland, College Park, MD 20742}
\address[addSWW]{Department of Mathematics and Center for Computation and Technology (CCT) Louisiana State University, Baton Rouge, LA 70803}
\address[addWZ]{Department of Mathematics, Rutgers University, Piscataway, NJ 08854}


\begin{abstract}
We present a robust discretization of the Ericksen model of liquid crystals with variable degree of orientation coupled with colloidal effects and electric fields.  The total energy consists of the Ericksen energy, a weak anchoring (or penalized Dirichlet) energy to model colloids, and an electrical energy for a given electric field.  We describe our special discretization of the total energy along with a method to compute minimizers via a discrete quasi-gradient flow algorithm which has a strictly monotone energy decreasing property.  Numerical experiments are given in two and three dimensions to illustrate that the method is able to capture non-trivial defect patterns, such as the Saturn ring defect.  We conclude with a rigorous proof of the $\Gamma$-convergence of our discrete energy to the continuous energy.

\end{abstract}

\begin{keyword}
liquid crystals \sep finite element method \sep gamma-convergence \sep gradient flow \sep line defect \sep plane defect \sep Saturn ring defect



\MSC 65N30 \sep 49M25 \sep 35J70

\end{keyword}

\end{frontmatter}


\section{Introduction}\label{sec:intro}

This paper presents a method for solving the Ericksen model of liquid crystals \cite{Ericksen_ARMA1991, deGennes_book1995}, with additional effects due to colloidal domains and electric fields.  Liquid crystals are a work-horse technology enabling electronic displays \cite{Goodby_inbook2012, Perkins_website2009, Senyuk_website2010}, for instance.  Moreover, they have a host of potential applications in material science \cite{Ackerman_NC2015, Araki_PRL2006, Bisoyi_CSR2011, Blanc_Sci2016, Blinov_book1983, Coles_NP2010, Conradi_SM2009, Hain_OC2001, Humar_OE2010, Moreno-Razo_Nat2012, Musevic_Sci2006, Musevic2011, Rahimi_PNAS2015, Shah_Small2012, Sun_SMS2014, Wang_NL2014}.  One avenue is to use external fields (e.g. electric fields) and colloidal dispersions to build new materials through directed self-assembly \cite{Copar_Mat2014, Araki_PRL2006, Conradi_SM2009, Eskandari_Lang2013, Furst_PNAS2011, Hamley_ACIE2003, Hiemenz_Book1997, Jeong_SM2015, Kuksenok_PRE1996, Liu_PNAS2013, Musevic_Sci2006, Shah_Small2012, Tasinkevych_NJP2012, Wang_Nat2012}.

A significant amount of mathematical analysis has been done on liquid
crystals \cite{Virga_book1994, Calderer_SJMA2002, Ambrosio_MM1990b,
  Ambrosio_MM1990a, Bauman_ARMA2002, Ball_ARMA2011, Lin_CPAM1989,
  Lin_CPAM1991, Ball_PAMM2007, Golovaty_JMAA2001, Hardt_ANL1988,
  Hardt_CVPDE1988, LinLiu_JPDE2001}.  Moreover, a host of numerical
methods have been developed for statics and dynamics
\cite{Badia_ACME2011, Barrett_M2AN2006, Cruz_JCP2013, Ramage_SJSC2013,
  Adler_SJSC2015, Adler_SJNA2015, Adler_SJSC2016}.  In particular, the
methods in \cite{Bartels_MC2010b,  Hardt_CVPDE1988, Cohen_CPC1989,
  Lin_SJNA1989, Alouges_SJNA1997} are for harmonic mappings and liquid
crystals with fixed degree of orientation, i.e. a unit vector field
$\vn$ (called the director field) represents the orientation of liquid
crystal molecules.  See \cite{Guillen-Gonzalez_M2AN2013, Liu_SJNA2000,
  Walkington_M2AN2011, Yang_JNNFM2011, Yang_JCP2013} for methods that
couple liquid crystals to Stokes flow.  We also refer to the survey paper \cite{Badia_ACME2011} for more numerical methods.

The method we present \cite{Nochetto_SJNA2017,
  NochettoWalker_MRSproc2015} is for the one-constant model of liquid
crystals with variable degree of orientation \cite{Ericksen_ARMA1991,
  deGennes_book1995, Virga_book1994} (Ericksen's model). The state of
the liquid crystal is described by a director field $\vn$ and a
scalar function $s$, the so-called degree-of-orientation, which
minimize the energy
\begin{equation}\label{energy0}
E[ s ,\vn]  := \iO \Big(\kappa | \nabla s |^2 + s^2 | \nabla  \vn |^2\Big) dx +
\iO \psi (s) dx.
\end{equation}
Hereafter, $\kappa > 0$ is a material constant, $\Om$ is a
bounded Lipschitz domain in $\R^d$ with $d\ge2$, and $\psi$ is a
double well potential (defined below).

Minimizers of the Ericksen model may exhibit non-trivial defects (depending on boundary conditions) \cite{Bethuel_book1994, Blinov_book1983, Brezis_CMP1986, Lin_CPAM1991, Lin_CPAM1989, Schoen_JDG1982}.  The presence of $s$ in \eqref{energy0} leads to an Euler-Lagrange equation for $\vn$ that is \emph{degenerate}.  This allows for line and plane defects (singularities of $\vn$) in dimension $d=3$ when $s$ vanishes; these types of defects are important for applications, especially defects that lie on three dimensional space curves \cite{Araki_PRL2006, Tojo_EPJE2009}.  Regularity properties of minimizers, and the size of defects, were studied in \cite{Lin_CPAM1991}.  This leads to the study of dynamics \cite{Calderer_SJMA2002} and corresponding numerics \cite{Barrett_M2AN2006}, which are relevant to our paper.  But in both cases they regularize the model to avoid the degeneracy associated with the order parameter $s$ vanishing.

Our finite element method (FEM) does \emph{not} require any regularization.  We discussed the mathematical foundation of our method in \cite{Nochetto_SJNA2017, NochettoWalker_MRSproc2015}: we proved stability and convergence properties via $\Gamma$-convergence \cite{Braides_book2002} (as the mesh size $h$ goes to zero) and developed a quasi-gradient flow method to solve the discrete problem.  Our discretization of the energy \eqref{energy0}, defined in \eqref{discrete_energy}, requires that the mesh $\Tk_h$ be \emph{weakly acute} (or the stronger condition of having \emph{non-obtuse} angles).  This discretization preserves the underlying structure and robustly handles the unit length constraint on $\vn$ and the degeneracy present when $s$ vanishes.  Our previous paper \cite{Nochetto_SJNA2017} showed a variety of simulations of minimizers with interesting defect structures.

The present paper demonstrates the ability of the Ericksen model, and of our method, to capture defect structures induced by colloidal inclusions (i.e. holes in the domain) and effects due to electric fields.  We are able to recreate the famous Saturn ring defect \cite{Gu_PRL2000}, which occurs around colloidal particles in different situations, by using both a conforming (non-obtuse mesh) and a non-conforming cube mesh with an immersed boundary approach to model the colloid.  In addition, we include electric field effects by incorporating an electric energy term into the total energy \eqref{energy0}, and demonstrate the classic \emph{Freedericksz transition} \cite{Virga_book1994, deGennes_book1995, Golovaty_JMAA2001, Biscari_CMT2007, Hoogboom_RSA2007}.  We also investigate the coupling of colloidal and electric effects.

The paper is organized as follows.  In Section \ref{sec:model}, we
recall the Ericksen model for liquid crystals with variable degree of
orientation, and describe our discretization of the continuous energy.
Section \ref{sec:method} recalls properties of the discretization, and
our initial minimization scheme.  Section \ref{sec:implement}
describes the details for properly implementing our method.  Section
\ref{sec:conforming_mesh} illustrates our method in the presence of a
colloidal inclusion with a \emph{conforming non-obtuse} mesh.  Section
\ref{sec:immersed_boundary} shows an alternative way to model colloids
by an immersed boundary approach (along with supporting
simulations). In Section \ref{sec:elec_field}, we show how to include
electric field effects in the model and describe a modified
minimization procedure to compute minimizers.
Section \ref{sec:general_scheme} presents the monotone energy decreasing
property of the quasi-gradient flow algorithm to compute discrete minimizers.
Section \ref{sec:Gamma_conv_results} provides a summary of the $\Gamma$-convergence for our discrete energy.  We close in Section \ref{sec:conclusion} with some discussion.

\section{Ericksen's model}\label{sec:model}

Let the \emph{director} field $\vn : \Om \subset \mathbb{R}^d \rightarrow \mathbb{S}^{d-1}$ be a vector-valued function with unit length (see Figure \ref{fig:director} for a description of the meaning of $\vn$).
The \emph{degree-of-orientation} $s: \Om \subset \mathbb{R}^d \rightarrow ( - \frac{1}{2} , 1) $ is a real valued function (see Figure \ref{fig:degree_of_orientation} for a description of the meaning of $s$).  The variable $\vn$, by itself, cannot properly describe a ``loss of order'' in the liquid crystal material because it has unit length.  The $s$ variable provides a way to characterize the local order (see Figure \ref{fig:degree_of_orientation}).

\begin{figure}[h!]
\begin{center}

\includegraphics[width=1.5in]{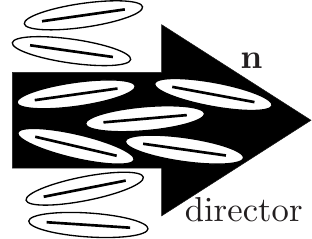}
\caption{Macroscopic order parameter: the director variable $\vn$.  Nematic liquid crystal molecules have an elongated rod-like shape (see elongated ellipsoids), which gives the material its \emph{anisotropic} nature.  The value of $\vn(x)$ (a unit vector), at the point $x$, represents a probabilistic average over a local ensemble of liquid crystal molecules ``near'' $x$ \cite{Virga_book1994}.}
\label{fig:director}
\end{center}
\end{figure}

\begin{figure}
\begin{center}

\includegraphics[width=5.0in]{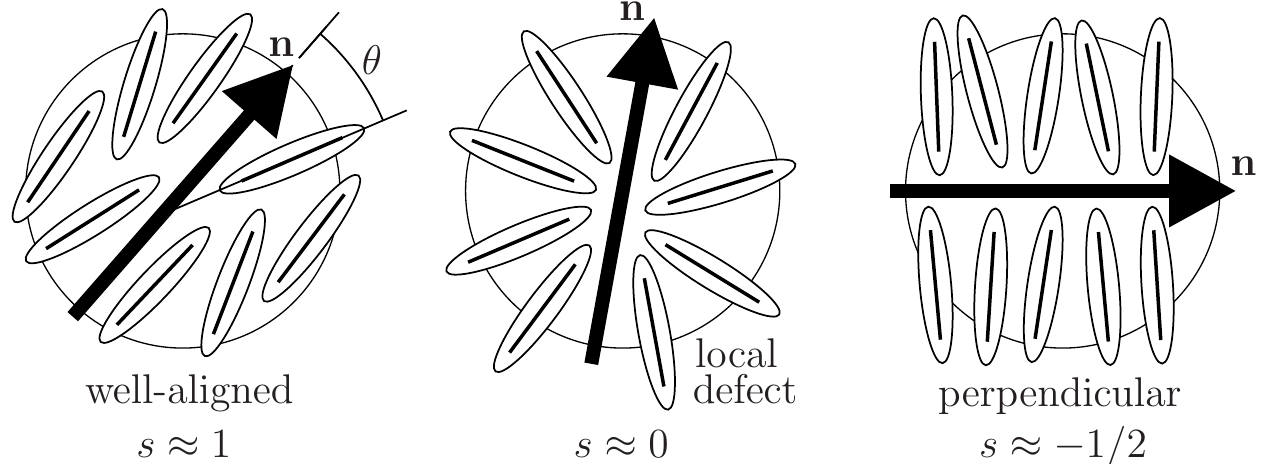}
\caption{Macroscopic order parameter: the degree-of-orientation variable $s$.  It is a probabilistic average over the angle $\theta$ between $\vn(x)$ and an individual liquid crystal molecule; the average is taken over a local ensemble \cite{Virga_book1994}.   The case $s = 1$ represents the state of perfect alignment in which all molecules (in the ensemble) are parallel to $\vn$. Likewise, $s = -1/2$ represents the state of microscopic order in which all molecules (in the ensemble) are perpendicular to $\vn$.  When $s = 0$, the molecules do not lie along any preferred direction which represents the state of an isotropic, uniformly random, distribution of molecules (in the ensemble).  The state $s = 0$ is called a \emph{defect} in the liquid crystal material.}
\label{fig:degree_of_orientation}
\end{center}
\end{figure}

\subsection{Ericksen's one constant energy}\label{sec:energy}

The equilibrium state of the liquid crystal material is described by
the pair $(s, \vn)$ minimizing a bulk-energy functional
\eqref{energy0} which we split as
\begin{equation}\label{energy}
  E_1[ s ,\vn]  := \iO \Big(\kappa | \nabla s |^2 + s^2 | \nabla  \vn |^2\Big) dx,
  \qquad
  E_2[s] := \iO \psi (s) dx ,
\end{equation}
where $\kappa > 0$.  The double well potential $\psi$ is a $C^2$ function defined on $-1/2 < s < 1$ that satisfies \cite{Ericksen_ARMA1991, Ambrosio_MM1990a, Lin_CPAM1991}
\begin{enumerate}
  \item $\lim_{s \rightarrow 1} \psi (s) = \lim_{s \rightarrow -1/2} \psi (s) = \infty$,
  \item $\psi(0) > \psi(s^*) = \min_{s \in [-1/2, 1]} \psi(s) = 0$ for some $s^* \in (0,1)$,
    \item $\psi'(0) =0 $.
\end{enumerate}

It was shown in \cite{Ambrosio_MM1990a, Lin_CPAM1991} that introducing
an auxiliary variable $\vu = s \vn$ allows the energy $E_1[s,\vn]$
to be rewritten as
\begin{equation}\label{auxiliary_energy_identity}
  E_1[ s , \vn] = \widetilde{E}_1[ s , \vu] := \iO
  \Big((\kappa - 1) | \nabla s |^2 + | \nabla  \vu |^2 \Big) dx,
\end{equation}
which follows from the orthogonal decomposition $\nabla \vu = \vn \otimes \nabla s + s \nabla \vn$ (and is due to the constraint $| \vn | = 1$). Hence, \cite{Ambrosio_MM1990a, Lin_CPAM1991} define the admissible class of solutions (minimizers) as
\begin{equation}\label{admissibleclass}
\begin{split}
    \Admis :=& \{ (s, \vu) : \Omega \rightarrow (-1/2 , 1) \times\mathbb{R}^d:
    ~ (s,\vu) \in [H^1(\Omega)]^{d+1},\; \vu = s \vn, \vn\in \mathbb{S}^{d-1} \}.
\end{split}
\end{equation}
We may also enforce boundary conditions on $(s,\vu)$,
possibly on different parts of the boundary.  Let $(\bdys,\bdyvu)$
be open subsets of $\dOm$ where we set Dirichlet boundary
conditions for $(s,\vu)$.
Then we have the following restricted admissible class
\begin{align}\label{admissibleclass_BC}
  \Admis(g,\vr) := \left\{ (s, \vu) \in \Admis :
  ~ s|_{\bdys} = g, \quad \vu|_{\bdyvu} = \vr \right\},
\end{align}
for some given functions $(g,\vr)\in [W^1_\infty(\R^d)]^{d+1}$ that
satisfy the following in a neighborhood of $\dOm$:
$-1/2 < g < 1$ and $\vr = g \vq$, for some $\vq \in \mathbb{S}^{d-1}$.
Note that if we further assume
\begin{equation}\label{gne0}
g \ge \delta_0 \quad\text{ on } \dOm, ~ \text{ for some } \delta_0 > 0,
\end{equation}
then the function $\vn$ is $H^1$ in a neighborhood of $\dOm$ and satisfies
$\vn=g^{-1}\vr = \vq \in \mathbb{S}^{d-1}$ on $\dOm$.

When the degree of orientation $s$ is a non-zero constant, the energy
$E_1[s,\vn]$ in \eqref{energy} effectively reduces to the Oseen-Frank energy $\int_{\Om} |\nabla \vn|^2$.
The purpose of the degree of orientation is to relax the energy of defects.
In fact, discontinuities in $\vn$ (i.e. defects) may still occur in the singular set
\begin{align}\label{singularset}
\Sing := \{ x \in \Omega :\; s(x) = 0 \},
\end{align}
with finite energy: $E[ s ,\vn] < \infty$.  The existence of
minimizers in the admissible class, subject to Dirichlet boundary
conditions, was shown in \cite{Ambrosio_MM1990a, Lin_CPAM1991}.
Minimizers with defects are constructed explicitly in \cite{Virga_book1994}
or discovered numerically in \cite{Nochetto_SJNA2017}.

The parameter $\kappa$ in \eqref{energy} plays a major role in the occurrence of defects.  Assuming the boundary condition for $s$ is a positive constant well away from zero, if $\kappa$ is large, then $\iO \kappa |\nabla s|^2 dx$ dominates the energy and $s$ stays close to a positive constant within the domain $\Om$.  Thus, defects are less likely to occur. If $\kappa$ is small (say $\kappa < 1$), then $\iO s^2 |\nabla \vn|^2 dx$ dominates the energy, and $s$ may vanish in regions of $\Om$ and induce a defect.  This was confirmed by our numerical experiments in \cite{Nochetto_SJNA2017, NochettoWalker_MRSproc2015}.  The physically relevant case $0 < \kappa < 1$ is the more difficult case with regard to proving $\Gamma$-convergence (see \cite{Nochetto_SJNA2017}) because the energy is no longer convex.

\subsection{Discretization of the energy}\label{sec:discretization}

Let $\Tk_h =\{ T \}$ be a conforming simplicial triangulation of
$\Om$. The set of nodes (vertices) of $\Tk_h$ is denoted $\Nk_h$ and has
cardinality $n$. We demand that $\Tk_h$ be {\it weakly acute}, namely
\begin{equation}\label{weakly-acute}
  k_{ij} := -\int_{\Om} \nabla \phi_i \cdot \nabla \phi_j \, dx \geq 0 \quad\text{for all } i\ne j,
\end{equation}
where $\phi_i$ is the standard ``hat'' basis function associated with node $x_i \in \Nk_h$. We indicate with $\omega_i = \text{supp} \;\phi_i$ the patch of a node $x_i$ (i.e. the ``star'' of elements in $\Tk_h$ that contain the vertex $x_i$).  Of course, \eqref{weakly-acute} imposes a severe geometric restriction on $\Tk_h$ \cite{Ciarlet_CMAME1973, Strang_FEMbook2008} (especially in three dimensions).  We recall the
following characterization of \eqref{weakly-acute} for $d=2$.
\begin{lem}[weak acuteness in two dimensions]
\label{weak_acuteness_2D}
For any pair of triangles $T_1$, $T_2$ in $\Tk_h$ in two space dimensions that share a common
edge $e$, let $\alpha_i$ be the angle in $T_i$ opposite to $e$ (for $i=1,2$).
Then \eqref{weakly-acute} holds if and only if $\alpha_1 + \alpha_2 \leq 180^\circ$ for every edge $e$.
\end{lem}
Generalizations of Lemma \ref{weak_acuteness_2D} to three dimensions, involving interior dihedral angles of tetrahedra, can be found in \cite{Korotov_MC2001, Brandts_LAA2008}.  Note that a \emph{non-obtuse} mesh (one where \emph{all} interior angles are bounded by $90^\circ$) is automatically weakly-acute.

The method uses the following finite element spaces:
\begin{equation}\label{eqn:discrete_spaces}
\begin{split}
  \Sh &:= \{ s_h \in H^1(\Om) : s_h |_{T} \text{ is affine for all } T \in \Tk_h \}, \\
  \Uh &:= \{ \vu_h \in [H^1(\Om)]^d : \vu_h |_{T} \text{ is affine in
    each component for all } T \in \Tk_h \}, \\
  \Nh &:= \{ \vn_h \in \Uh : |\vn_h(x_i)| = 1 \text{ for all nodes $x_i \in \Nk_h$} \},
\end{split}
\end{equation}
where $\Nh$ imposes the unit length constraint \emph{at the vertices} of the mesh.

Let $I_h$ denote the piecewise linear Lagrange interpolation
operator on mesh $\Tk_h$ with values in either $\Sh$ or $\Uh$.
We have the following \emph{discrete} version of the admissible class:
\begin{equation}\label{admissibleclass_discrete}
    \Admis_h := \{ (s_h, \vu_h) \in \Sh \times \Uh : -1/2 < s_h < 1 \,\text{ in } \Om,
    \, \vu_h = I_h[s_h \vn_h] \text{ where } \vn_h \in \Nh \}.
\end{equation}
Next, we let
$g_h := I_h g$ and $\vr_h := I_h \vr$ be the discrete Dirichlet data,
and introduce the discrete spaces that include (Dirichlet) boundary conditions
\begin{equation*}\label{eqn:discrete_spaces_BC}
\begin{split}
  \Sh (\bdys,g_h) &:= \{ s_h \in \Sh : s_h |_{\bdys} = g_h \}, \quad
  \Uh (\bdyvu,\vr_h) := \{ \vu_h \in \Uh : \vu_h |_{\bdyvu} = \vr_h \},
\end{split}
\end{equation*}
as well as the discrete admissible class with boundary conditions:
\begin{equation}\label{admissibleclass_discrete_BC}
  \Admis_h (g_h,\vr_h) := \left\{ (s_h,\vu_h) \in \Admis_h : s_h \in \Sh (\bdys,g_h),  \vu_h \in \Uh (\bdyvu,\vr_h) \right\}.
\end{equation}
In view of \eqref{gne0}, we can also impose the Dirichlet condition
$\vn_h = I_h[g_h^{-1} \vr_h]$ on $\dOm$.

Our discrete version of $E_1[ s , \vn]$ is ``derived'' by invoking
basic properties of the stiffness matrix entries $k_{ij}$.  First note
$
\sum_{j=1}^n k_{ij} = 0
$
for all $x_i \in \Nk_h$, and for piecewise linear $s_h = \sum_{i=1}^n s_h(x_i) \phi_i$ we have
\begin{align*}
  \int_{\Om} | \nabla s_h |^2 dx = -\sum_{i=1}^n k_{ii} s_h(x_i)^2 - \sum_{i, j = 1, i \neq j}^n k_{ij} s_h(x_i) s_h(x_j).
\end{align*}
Thus, using $k_{ii} = - \sum_{j \neq i} k_{ij}$ and the symmetry $k_{ij}=k_{ji}$, we get
\begin{equation}\label{eqn:dirichlet_integral_identity}
\begin{aligned}
  \int_{\Om} | \nabla s_h |^2 dx &= \sum_{i, j = 1}^n k_{ij} s_h(x_i) \big(s_h(x_i) - s_h(x_j)\big)
\\
&= \frac{1}{2} \sum_{i, j = 1}^n k_{ij} \big(s_h(x_i) - s_h(x_j)\big)^2 = \frac{1}{2} \sum_{i, j = 1}^n k_{ij} \big( \dij s_h \big)^2,
\end{aligned}
\end{equation}
where we define
\begin{equation}\label{eqn:delta_ij}
  \dij s_h := s_h(x_i) - s_h(x_j), \quad \dij \vn_h := \vn_h(x_i) - \vn_h(x_j).
\end{equation}
Therefore, we define the discrete energy to be
\begin{equation}\label{discrete_energy_E1}
\begin{split}
  E_1^h [s_h, \vn_h] := & \frac{\kappa}{2} \sum_{i, j = 1}^n k_{ij} \left( \dij s_h \right)^2 + \frac{1}{2} \sum_{i, j = 1}^n k_{ij} \left(\frac{s_h(x_i)^2 + s_h(x_j)^2}{2}\right) |\dij \vn_h|^2,
\end{split}
\end{equation}
where the second term is a first order approximation of $\iO s^2|\nabla \vn|^2$,
which is novel in the finite element literature \cite{Nochetto_SJNA2017}.
The double well energy is discretized in the usual way:
\begin{equation}\label{discrete_energy_E2}
  E_2^h [s_h] := \int_{\Om} \psi (s_h(x)) dx.
\end{equation}

The specific form of \eqref{discrete_energy_E1} lies on the fact
that it makes the nodal values of $s_h$ and $\vn_h$ \emph{readily
  accessible} for analysis.  The identity in
\eqref{auxiliary_energy_identity} is obtained (at the continuous
level) by taking advantage of the unit length constraint $|\vn| = 1$.
However, at the discrete level, we only impose the unit length
constraint at the \emph{nodes} of the mesh, and we cannot hope
for much more because $\vn_h$ is a piecewise polynomial.  Hence, in
order to obtain a similar identity to
\eqref{auxiliary_energy_identity} (see Lemma
\ref{lemma:energydecreasing} below), we need access to nodal values.
In \cite{Nochetto_SJNA2017}, we show that the discrete energy
\eqref{discrete_energy_E1}
allows us to handle the degenerate coefficient $s_h^2$ \emph{without} regularization.


The discrete formulation is as follows.  Find $(s_h, \vn_h) \in \Sh (\bdys,g_h) \times \Nh (\bdyvn,\vr_h)$ such that the following energy is minimized:
\begin{equation}\label{discrete_energy}
  E^h[s_h, \vn_h] := E_1^h [s_h, \vn_h] + E_2^h [s_h].
\end{equation}

\section{Review of the method}\label{sec:method}

\subsection{Energy inequality}\label{sec:energy_inequality}

Our discrete energy \eqref{discrete_energy_E1} satisfies a discrete
version of \eqref{auxiliary_energy_identity}
\cite{Ambrosio_MM1990a,Lin_CPAM1991}, which is a key component
of our analysis in \cite{Nochetto_SJNA2017}. To see this, we introduce $\widetilde{s}_h:=I_h|s_h|$ and two
discrete versions of the vector field $\vu$
\begin{equation}\label{uh}
\vu_h:=I_h[s_h \vn_h] \in \Uh,
\quad
\widetilde\vu_h := I_h[\widetilde{s}_h \vn_h]\in\Uh.
\end{equation}
Note that both $(s_h, \vu_h)$, $(\widetilde{s}_h, \widetilde{\vu}_h)$
are in $\Admis_h (g_h,\vr_h)$. We now state a discrete
version of \eqref{auxiliary_energy_identity}.

\smallskip
\begin{lem}[discrete energy inequality]
\label{lemma:energydecreasing}
Let the mesh $\Tk_h$ satisfy \eqref{weakly-acute}.
If $(s_h,\vu_h)\in\Admis_h(g_h,\vr_h)$, then,
for any $\kappa > 0$, the discrete energy \eqref{discrete_energy_E1} satisfies
\begin{align}
  E_1^h [s_h, \vn_h]
  \geq
  (\kappa - 1) \iO |\nabla s_h|^2 dx + \iO |\nabla \vu_h|^2 dx
  =: \widetilde E_1^h[s_h,\vu_h],
  \label{energy_inequality}
\end{align}
as well as
\begin{align}\label{abs_inequality}
  E_1^h [s_h, \vn_h]
  \geq
  (\kappa - 1) \iO |\nabla \widetilde{s}_h|^2 dx + \iO |\nabla \widetilde{\vu}_h|^2 dx
  =: \widetilde E_1^h[\widetilde{s}_h,\widetilde\vu_h].
\end{align}
\end{lem}

In fact, the following inequalities are valid \cite{Nochetto_SJNA2017}
\begin{equation}\label{eqn:energyequality}
  E_1^h[s_h, \vn_h] - \widetilde{E}_1^h[s_h,\vu_h] \ge \mathcal{E}_h,
  \qquad
  E_1^h[s_h, \vn_h] - \widetilde{E}_1^h[\widetilde{s}_h,\widetilde{\vu}_h]
  \ge \widetilde{\mathcal{E}}_h,
\end{equation}
where
\begin{equation}\label{eqn:residual}
  \mathcal{E}_h := \frac{1}{4} \sum_{i, j = 1}^n k_{ij} \big(\dij s_h
  \big)^2 \big| \dij \vn_h \big|^2,
  \qquad
  \widetilde{\mathcal{E}}_h := \frac{1}{4} \sum_{i, j = 1}^n k_{ij}
  \big(\dij \widetilde{s}_h \big)^2 \big| \dij \vn_h \big|^2.
\end{equation}
Note that $\mathcal{E}_h,\widetilde{\mathcal{E}}_h \geq 0$ because $k_{ij} \geq 0$
for $i \neq j$. We refer to \cite{Nochetto_SJNA2017} for further details.

\subsection{Minimization scheme}\label{sec:min_algorithm}

We summarize the discrete quasi-gradient flow scheme in \cite{Nochetto_SJNA2017} which we use to compute discrete minimizers.

\subsubsection{Boundary conditions}

In the continuous setting, Dirichlet boundary conditions are enforced in
the space. Let
\begin{equation}\label{eqn:dirichlet_BC}
  \Hbdy{\Gamma} := \{ v \in H^1(\Om) : v = 0  \text{ on } \Gamma \},
\end{equation}
where $\Gamma\subset\partial\Omega$ is either
$\bdys,\bdyvn$, and $s\in g + H^1_{\bdys}(\Omega)$,
$\vu \in \vr + [H^1_{\bdyvu}(\Omega)]^d$. We assume
$\bdyvn = \bdyvu \subset \bdys$ and \eqref{gne0} to be
valid on $\bdys$. The trace $\vn=\vq:=g^{-1}\vr$ is thus
well defined on $\bdyvn$.

The superscript $k$ will stand for an iteration
counter. Therefore, $s_h^k\in\Sh (\bdys,g_h)$ and $\vn_h^k\in\Nh (\bdyvn,\vq_h)$
indicate iterates satisfying Dirichlet boundary conditions where
$g_h := I_h g$ and $\vq_h := I_h\vq$.
We will further simplify the notation in some places upon writing:
\[
s_i^k := s_h^k(x_i),
\quad
\vn_i^k := \vn_h^k(x_i),
\quad
z_i := z_h(x_i),
\quad
\vv_i := \vv_h(x_i).
\]

\subsubsection{First order variation}

We start with the energy $E_1^h$.
Due to the unit length constraint at the nodes in $\Nh$ (see \eqref{eqn:discrete_spaces}), we introduce the space of discrete tangential variations:
\begin{equation}\label{eqn:discrete_tangent_variation_space}
\begin{split}
  \Yh (\vn_h) &= \{ \vv_h \in \Uh : \vv_h(x_i) \cdot \vn_h(x_i) = 0
  \text{ for all nodes } x_i \in \Nk_h \}.
\end{split}
\end{equation}
Next, the first order variation of $E^h_1[s_h^k, \vn_h^k]$ in the direction
  $\vv_h \in \Yh (\vn_h^k) \cap \Hbdy{\bdyvn}$ reads
\begin{equation}\label{first_order_discrete_variation_dvN}
\begin{split}
 \delta_{\vn_h} E^h_1 [s_h^k, \vn_h^k ; \vv_h] &=
 \sum_{i, j = 1}^n k_{ij} \left(\frac{(s_i^k)^2 + (s_j^k)^2 }{2}\right) ( \dij \vn_h^k ) \cdot ( \dij \vv_h ).
\end{split}
\end{equation}
The first order variation of $E^h_1[s_h^k, \vn_h^k]$ in the direction
$z_h \in \Sh \cap \Hbdy{\bdys}$ consists of two terms
\begin{equation}\label{first_order_discrete_variation_dS}
\begin{split}
 \delta_{s_h} E^h_1 [s_h^k ,\vn_h^{k} ; z_h] &=
 \kappa \sum_{i, j = 1}^n k_{ij} \left( \dij s_h^k \right) \left( \dij z_h \right)
 + \sum_{i, j = 1}^n k_{ij}  |\dij \vn_h^{k}|^2 \left(\frac{s_i^k z_i+ s_j^k z_j}{2}\right) .
\end{split}
\end{equation}

We next consider the energy $E_2^h$.
In order to guarantee a monotonically energy decreasing scheme, we employ the convex splitting technique in \cite{Wise_SJNA2009, Shen_DCDS2010, Shen_SJSC2010}, i.e. we split the double well potential $\psi$ into a convex and concave part.
Let $\psi_c $ and $\psi_e$ be both convex for all $s \in (-1/2, 1)$
so that $\psi (s) = \psi_c(s) - \psi_e(s)$,
and set
\begin{align}\label{variationE2}
 \delta_{s_h} E^h_2 [ s_h^{k+1}; z_h ] := \iO \big[ \psi_c'(s_h^{k+1}) - \psi_e'(s_h^{k}) \big] z_h dx,
\end{align}
which yields the inequality
\begin{equation}\label{convex_split_inequality}
  \iO \psi(s_h^{k+1}) dx - \iO \psi(s_h^{k})  dx  \leq \delta_{s_h} E^h_2 [ s_h^{k+1}; s_h^{k+1} - s_h^k ],
\end{equation}
for any $s_h^{k}$ and $s_h^{k+1}$ in $\Sh$ \cite{Nochetto_SJNA2017}.  Note that
\begin{align*}
\delta_{\vn_h} E^h[s_h^k, \vn_h^k; \vv_h] &= \delta_{\vn_h} E^h_1[s_h^k, \vn_h^k; \vv_h], \\
\delta_{s_h} E^h[s_h^k, \vn_h^k; z_h] &= \delta_{s_h} E_1^h[s_h^k, \vn_h^k; z_h] + \delta_{s_h} E_2^h[s_h^k, \vn_h^k; z_h].
\end{align*}

\subsubsection{Discrete quasi-gradient flow algorithm}\label{algorithm}

Our scheme for minimizing $E^h [ s_h ,\vn_h]$,
defined in \eqref{discrete_energy},
is an alternating direction method, which minimizes with respect to
$\vn_h$ and evolves $s_h$ separately in the steepest descent direction during each
iteration. Therefore, this algorithm is not a standard gradient flow
but rather a quasi-gradient flow.

\medskip\noindent
{\bf Algorithm:}
Given $(s_h^{0}, \vn_h^0)$ in $\Sh (\bdys,g_h) \times \Nh (\bdyvn,\vq_h)$,
iterate Steps (a)-(c) for $k\ge0$.

\smallskip\noindent
{\bf Step (a): Minimization.} Find $\vt_h^k\in\Yh (\vn_h^k) \cap \Hbdy{\bdyvn}$ such that $\vn_h^k + \vt_h^k$ minimizes the energy
$
E^h [ s_h^k ,\vn_h^k + \vv_h]
$
for all $\vv_h$ in $\Yh (\vn_h^k) \cap \Hbdy{\bdyvn}$, i.e. $\vt_h^k$ satisfies
\begin{align*}
\delta_{\vn_h} E^h [s_h^k ,\vn_h^k + \vt_h^k; \vv_h] = 0, \quad\forall \vv_h \in \Yh (\vn_h^k) \cap \Hbdy{\bdyvn}.
\end{align*}
{\bf Step (b): Projection.} Normalize
$\vn_i^{k+1} := \frac{ \vn_i^k +   \vt_i^k } { | \vn_i^k + \vt_i^k  | }$ at all nodes $x_i \in \Nk_h$.

\noindent
{\bf Step (c): Gradient flow.} Using $(s_h^k, \vn_h^{k+1})$, find $s_h^{k+1}$ in $\Sh (\bdys,g_h)$ such that
\begin{align*}
\iO \frac{s_h^{k+1} - s_h^{k}}{\delta t} z_h = - \delta_{s_h} E^h [s_h^{k+1} ,\vn_h^{k+1} ; z_h], \quad\forall z_h \in \Sh \cap \Hbdy{\bdys}.
\end{align*}
In the numerical experiments, we impose Dirichlet boundary conditions for both $s_h^k$ and $\vn_h^k$ on subsets of the boundary. Note that the scheme has no restriction on the time step thanks to the implicit Euler method in Step (c).

The quasi-gradient flow scheme has a monotone energy decreasing property, provided the mesh $\Tk_h$ is weakly acute \eqref{weakly-acute} \cite{Ciarlet_CMAME1973, Strang_FEMbook2008}.
\begin{thm}[monotonicity \cite{Nochetto_SJNA2017}]
  \label{energydecreasing}
Let $\Tk_h$ satisfy \eqref{weakly-acute}. The iterate
$(s_h^{k+1}, \vn_h^{k+1})$ of the Algorithm (discrete quasi-gradient
flow) of Section \eqref{algorithm} exists and satisfies
  \[
    E^h [s_h^{k+1} ,\vn_h^{k+1}] \leq E^h [s_h^k ,\vn_h^k ] - \frac{1}{\delta t} \iO (s_h^{k+1} - s_h^k)^2 dx.
  \]
  Equality holds if and only if $(s_h^{k+1}, \vn_h^{k+1}) = (s_h^k,\vn_h^k)$
  (equilibrium state).
\end{thm}
We extend this result in Theorem \ref{thm:energydecreasing_general} to include additional energy terms; see \eqref{discrete_energy_total}.

\section{Implementation}\label{sec:implement}

We implemented our method using the MATLAB/C++ finite element toolbox
FELICITY \cite{FELICITY_REF}.  In this section, we give details on
forming the ensuing discrete systems, and how to solve part (a) of the quasi-gradient flow algorithm in 3-D using the tangent space.  For all 3-D simulations, we used the algebraic multi-grid solver (AGMG) \cite{Notay_ETNA2010, Napov_NLAA2011, Napov_SISC2012, Notay_SISC2012} to solve the linear systems in parts (a) and (c) of the quasi-gradient flow algorithm.  In 2-D, we simply used the ``backslash'' command in MATLAB.

\subsection{Finite element matrices}

Implementing the algorithm requires construction of the discrete
energy, as well as its variational derivative.  This requires the
symmetric mass and stiffness finite element matrices:
$M := (m_{ij})^n_{i,j = 1}$, $K := (-k_{ij})^n_{i,j = 1}$, where
\begin{equation}\label{eqn:discrete_FE_matrices}
\begin{split}
     m_{ij} &= \iO \phi_i \, \phi_j, \quad k_{ij} = -\iO \nabla \phi_i \cdot \nabla \phi_j,
\end{split}
\end{equation}
and $\{ \phi_i \}^n_{i=1}$ is the set of basis functions of the space $\Sh$.

\subsection{Finite element functions and coefficient vectors}

The function $s_h$ is represented by a linear combination of $\{ \phi_i \}^n_{i=1}$.  If the dimension of $\Om$ is $d$, then the vector field $\vn_h$ has $d$ components, where each component is written as a linear combination of $\{ \phi_i \}^n_{i=1}$.  The nodal values of $s_h$ and $\vn_h$, at node $x_i$, are denoted by $s_i$ and $\vn_i$.

The corresponding coefficient vectors (arrays) are denoted with non-italicized capital letters, i.e. $\coefS \in \R^n$, $\coefN \in \R^{d n}$, such that
\begin{equation}\label{eqn:coefficient_vectors}
  \coefS = [\coefS_1, \coefS_2, ..., \coefS_n ]\tp, \quad \coefN = [\coefN_1, \coefN_2, ..., \coefN_n, \coefN_{n+1}, ... \coefN_{dn}]\tp,
\end{equation}
where $\coefS_i := s_i$ and
\begin{equation}\label{eqn:coefN_correspondance}
\begin{split}
  \coefN_i &= \vn_i \cdot \ve_1, \\
  \coefN_{i+n} &= \vn_i \cdot \ve_2, \\
  \coefN_{i+2n} &= \vn_i \cdot \ve_3, \\
  & \vdots \\
  \coefN_{i + (k-1) n} &= \vn_i \cdot \ve_k, \\
  & \vdots \\
  \coefN_{i + (d-1) n} &= \vn_i \cdot \ve_d,
\end{split}
\end{equation}
for $1 \leq i \leq n$, where $\{ \ve_i \}^d_{i=1}$ are the canonical basis vectors of $\R^d$.  In other words, we store the coefficients of $\vn_h$ so that the $\ve_1$ components are first, followed by the $\ve_2$ components, and so on. Therefore,
\begin{equation*}
  \vn_h(x_i) \equiv \vn_i = (\coefN_i, \coefN_{i+n}, \coefN_{i+2n}, ..., \coefN_{i+(d-1)n})\tp, \quad \text{for all } 1 \leq i \leq n.
\end{equation*}

\subsection{Discrete variations}

Let us write $\delta_{s_h} E^h_1 [\Omega, s_h, \vn_h ; z_h]$ from \eqref{first_order_discrete_variation_dS} in a different form:
\begin{equation}\label{eqn:discrete_sub_energy_delta_S_temp}
\begin{split}
  \delta_{s_h} E^h_1 [s_h, \vn_h ; z_h] &= 2 \kappa \coefZ\tp K \coefS + \frac{1}{2} \sum^{n}_{i,j=1} (A(\vn_h))_{ij} \left( \coefS_i \coefZ_i + \coefS_j \coefZ_j \right),
\end{split}
\end{equation}
where $\coefZ \in \R^n$ is the coefficient vector corresponding to $z_h \in \Sh$, and $A(\vn_h) \equiv (A(\vn_h))^{n}_{i,j = 1}$ is the symmetric matrix defined by
\begin{equation}\label{eqn:modified_K_vN_matrix_for_S}
\begin{split}
    (A(\vn_h))_{ij} &= k_{ij} \sum^{d}_{r=1} (\coefN_{i + (r-1)n} - \coefN_{j + (r-1)n})^2.
\end{split}
\end{equation}

\begin{lem}\label{prop:diag_matrix_identity}
Let $A$ be an arbitrary $n \times n$ matrix, and $\coefY$, $\coefZ$ be arbitrary $n \times 1$ column vectors.  Then
\begin{equation}\label{eqn:diag_matrix_identity}
\begin{split}
    \sum^n_{i,j = 1} (A)_{ij} \coefY_i \coefZ_i = \sum^n_{i=1} \coefY_i \coefZ_i \underbrace{\sum^n_{j=1} (A)_{ij}}_{=:(\widehat{A})_{i}} = \coefZ\tp [\mathrm{diag} (\widehat{A})] \coefY,
\end{split}
\end{equation}
where $\widehat{A}$ is a $n \times 1$ column vector and $[\mathrm{diag} (\widehat{A})]$ is a $n \times n$ diagonal matrix formed from $\{ (\widehat{A})_{1}, (\widehat{A})_{2}, ..., (\widehat{A})_{n} \}$.
\end{lem}

Continuing with \eqref{eqn:discrete_sub_energy_delta_S_temp}, and using \eqref{eqn:diag_matrix_identity}, we get
\begin{equation}\label{eqn:discrete_sub_energy_delta_S}
\begin{split}
  \delta_{s_h} E^h_1 [s_h, \vn_h ; z_h] &= 2 \kappa \coefZ\tp K \coefS + \frac{1}{2} \sum^{n}_{i,j=1} (A(\vn_h))_{ij} \coefS_i \coefZ_i + \frac{1}{2} \sum^{n}_{i,j=1} (A(\vn_h))_{ij} \coefS_j \coefZ_j, \\
  &= 2 \kappa \coefZ\tp K \coefS + \coefZ\tp D(\vn_h) \coefS,
\end{split}
\end{equation}
\begin{equation*}
  \text{where} \quad D(\vn_h) = \mathrm{diag} (\widehat{A(\vn_h)}), \quad (\widehat{A(\vn_h)})_{i} = \sum^{n}_{j=1} (A(\vn_h))_{ij}.
\end{equation*}

Next, we write out $\delta_{\vn_h} E^h_1 [s_h, \vn_h; \vv_h]$ from \eqref{first_order_discrete_variation_dvN} in a different form:
\begin{equation}\label{eqn:discrete_sub_energy_delta_vN_temp_1}
\begin{split}
  \delta_{\vn_h} E^h_1 [s_h, \vn_h; \vv_h] &= \sum^{d}_{r=1} \underbrace{\sum^n_{i,j=1} k_{ij} \left( \frac{\coefS_i^2 + \coefS_j^2}{2} \right) \Theta_{ij}(r)}_{\Xi_r :=},
\end{split}
\end{equation}
where we defined
\begin{equation}\label{eqn:discrete_sub_energy_delta_vN_temp_2}
\begin{split}
  \Theta_{ij}(r) := \left( \coefN_{i + (r-1) n} - \coefN_{j + (r-1) n} \right) \cdot \left( \coefV_{i + (r-1) n} - \coefV_{j + (r-1) n} \right),
\end{split}
\end{equation}
with $\coefV \in \R^{dn}$ the coefficient vector corresponding to $\vv_h \in \Uh$.

Let us now focus on $\Xi_1$:
\begin{equation}\label{eqn:discrete_sub_energy_delta_vN_temp_3}
\begin{split}
  \Xi_1 &= \sum^{n}_{i,j=1} k_{ij} \left( \frac{\coefS_i^2 + \coefS_j^2}{2} \right) \left( \coefN_i - \coefN_j \right) \cdot \left( \coefV_i - \coefV_j \right), \\
  &= \frac{1}{2} \sum^{n}_{i,j=1} (\widetilde{\vA}(s_h))_{ij} \left( \coefN_i \cdot \coefV_i - \coefN_i \cdot \coefV_j - \coefN_j \cdot \coefV_i + \coefN_j \cdot \coefV_j \right),
\end{split}
\end{equation}
where $\widetilde{\vA}(s_h)$ is the $n \times n$ symmetric matrix defined by
\begin{equation}\label{eqn:modified_K_vN_matrix_for_vN}
  \widetilde{\vA}(s_h) \equiv
     ((\widetilde{\vA}(s_h))_{ij})^{n}_{i,j = 1},
     \quad (\widetilde{\vA}(s_h))_{ij} = k_{ij} \left( \coefS_i^2 + \coefS_j^2 \right).
\end{equation}
Using symmetry gives
\begin{equation}\label{eqn:discrete_sub_energy_delta_vN_temp_4}
\begin{split}
  \Xi_1 &= \frac{1}{2} \sum^{n}_{i,j=1} 2 (\widetilde{\vA}(s_h))_{ij} \left( \coefN_i \cdot \coefV_i - \coefN_j \cdot \coefV_i \right) \\
  &= \sum^{n}_{i,j=1} (\widetilde{\vA}(s_h))_{ij} \left( \coefN_i \cdot \coefV_i \right) - \sum^{n}_{i,j=1} (\widetilde{\vA}(s_h))_{ij} \left( \coefN_j \cdot \coefV_i \right) \\
  &= \coefV(1:n)\tp \widetilde{\vD}(s_h) \coefN(1:n) - \coefV(1:n)\tp \widetilde{\vA}(s_h) \coefN(1:n),
\end{split}
\end{equation}
where $\coefV(1:n)$ denotes the first $n$ components of $\coefV$, etc., and
\begin{equation}\label{eqn:tilde-D}
  \widetilde{\vD}(s_h) = \mathrm{diag} (\widehat{\widetilde{\vA}(s_h)}), \quad (\widehat{\widetilde{\vA}(s_h)})_{i} = \sum^{n}_{j=1} (\widetilde{\vA}(s_h))_{ij}.
\end{equation}

Applying the same argument to $\Xi_r$, we get
\begin{equation}\label{eqn:discrete_sub_energy_delta_vN_temp_5}
\begin{split}
  \Xi_r &= \coefV((r-1)n + 1:rn)\tp \widetilde{\vD}(s_h) \coefN((r-1)n + 1:rn) \\
  &\quad - \coefV((r-1)n + 1:rn)\tp \widetilde{\vA}(s_h) \coefN((r-1)n + 1:rn),
\end{split}
\end{equation}
Therefore, recalling \eqref{eqn:discrete_sub_energy_delta_vN_temp_1}, we obtain
\begin{equation}\label{eqn:discrete_sub_energy_delta_vN}
\begin{split}
  \delta_{\vn_h} E^h_1 [s_h, \vn_h; \vv_h] &= \coefV\tp \vD(s_h) \coefN - \coefV\tp \vA(s_h) \coefN,
\end{split}
\end{equation}
where $\vD(s_h)$ is $d n \times d n$ block diagonal (with $d$
identical blocks), where each block equals $\widetilde{\vD}(s_h)$
defined in \eqref{eqn:tilde-D}; similarly, $\vA(s_h)$ is $d n
\times d n$ block diagonal (with $d$ identical blocks), where each
block equals $\widetilde{\vA}(s_h)$
defined in \eqref{eqn:modified_K_vN_matrix_for_vN}.

\subsection{Discrete quasi-gradient flow}

Given $(s_h^{k}, \vn_h^k)$, we have the corresponding coefficient
vectors $(\coefS^{k}, \coefN^{k})$. We now rewrite the Algorithm
in Section \ref{algorithm} in terms of the matrices and
vectors introduced earlier.

\textbf{Step (a):}
By \eqref{eqn:discrete_sub_energy_delta_vN}, we solve the following linear system \emph{in the tangent space} (see Section \ref{sec:tangent_variations}) to obtain $\vt_h^k$:
\begin{equation}\label{eqn:Algorithm_step_a_FE_matrices}
\begin{split}
    \left[ \vD(s_h^k) - \vA(s_h^k) \right] \coefT^k &= - \left[ \vD(s_h^k) - \vA(s_h^k) \right] \coefN^k,
\end{split}
\end{equation}
where $\coefT^k$ in $\R^{d n}$ is the coefficient vector corresponding to $\vt_h^k$ in $\Yh (\vn_h^k) \cap \Hbdy{\bdyvn}$.  Note that \eqref{eqn:Algorithm_step_a_FE_matrices} must be modified to enforce Dirichlet boundary conditions (if necessary).

\begin{remark}[solving a degenerate system]\label{rem:how_to_solve_degenerate_system}

The system matrix in \eqref{eqn:Algorithm_step_a_FE_matrices} is symmetric positive semi-definite, which is easily verified from the properties of $\vD(s_h^k)$.  Moreover, it is positive definite if $|s_h^k| > 0$ everywhere.  Hence, the system can be solved by any method for symmetric positive definite matrices.

When $s_h^k = 0$ at a sufficient number of nodes, the matrix will be
singular.  In this case, one could use a conjugate gradient method
\cite{GolubVanLoan_book2013};
note that the right-hand-side of \eqref{eqn:Algorithm_step_a_FE_matrices}
is guaranteed to be in the column space of the system matrix.

In 3-D, we solve \eqref{eqn:Algorithm_step_a_FE_matrices} using AGMG \cite{Notay_ETNA2010, Napov_NLAA2011, Napov_SISC2012, Notay_SISC2012}, which has the following condition: all the diagonal entries of the matrix must be positive.  If this is not the case, we must modify the system matrix in \eqref{eqn:Algorithm_step_a_FE_matrices} accordingly.  This is most easily done by using the minimizing movement strategy described in Section \ref{sec:min_movements}, which effectively adds an identity matrix (with small weight $\mmcoef > 0$) to the system matrix in \eqref{eqn:Algorithm_step_a_FE_matrices}.

\end{remark}

\textbf{Step (b):} apply the normalization step at all nodes to obtain $\vn_h^{k+1}$, i.e.
\begin{equation*}
\begin{split}
  \text{(1):}& ~~ \coefW := \coefN^k + \coefT^k, \\
  \text{(2):}& ~~ \alpha_i := \left( \coefW_i^2 + \coefW_{i + n}^2 + \cdots + \coefW_{i + (d-1)n}^2 \right)^{1/2}, ~ \text{for all } 1 \leq i \leq n, \\
  \text{(3):}& ~~ \coefN^{k+1}_{i + (r-1)n} := \coefW_{i + (r-1)n} / \alpha_i, ~ \text{for all } 1 \leq i \leq n, \text{ and } r = 1,2,...,d,
\end{split}
\end{equation*}
where $\coefN^{k+1}$ is the coefficient vector corresponding to $\vn_h^{k+1}$.

\textbf{Step (c):} Use the following convex splitting of the double well: $\psi(s) = \psi_c(s) - \psi_e(s)$, where we choose
\begin{equation*}
  \psi_c(s) = c_0 s^2, \qquad \psi_e(s) = c_0 s^2 - \psi(s),
\end{equation*}
and select $c_0 > 0$ large enough to ensure that $\psi_c(s)$, $\psi_e(s)$ are convex for all $-1/2 < s < 1$.  Recall \eqref{variationE2} and note that $\psi_c'(s)$ is linear.  Hence, we can write
\begin{equation*}
\begin{split}
    \delta_{s_h} E^h_2 [s_h^{k+1}; z_h ] &= \iO [ \psi_c'(s_h^{k+1}) - \psi_e'(s_h^{k}) ] z_h dx \\
    &= 2 c_0 \coefZ\tp M \coefS^{k+1} - \coefZ\tp \coefB(s_h^{k}),
\end{split}
\end{equation*}
where $\coefZ$ is the coefficient vector corresponding to $z_h$ in $\Sh \cap \Hbdy{\bdys}$, and $\coefB(s_h^{k}) = (\coefB_1, ..., \coefB_n)\tp$ in $\R^{n}$ is a column vector defined by
\begin{equation}\label{eqn:explicit_double_well_split}
  \coefB_i = \iO \psi_e'(s_h^{k}) \phi_i dx, \quad \text{for all } 1 \leq i \leq n.
\end{equation}

Therefore, using \eqref{eqn:discrete_sub_energy_delta_S}, we solve the following linear system for $\coefS^{k+1}$:
\begin{equation}\label{eqn:Algorithm_step_c_FE_matrices}
\begin{split}
    \left[ M + 2 \dt \kappa K + \dt D(\vn_h^{k+1}) + 2 c_0 \dt M \right] \coefS^{k+1} = M \coefS^{k} + \dt \coefB(s_h^{k}),
\end{split}
\end{equation}
where $\coefS^{k+1}$ is the coefficient vector corresponding to $s_h^{k+1}$.  Note that \eqref{eqn:Algorithm_step_c_FE_matrices} must be modified to enforce Dirichlet boundary conditions.

\subsection{Tangential variations}\label{sec:tangent_variations}

Solving Step (a) of the Algorithm requires a tangential basis for the test function and the solution.  However, forming the matrix system is easily done by \emph{first ignoring} the tangential variation constraint (i.e. arbitrary variations), followed by a simple modification of the matrix system.  For a concrete realization of the procedure, we consider the case $d=3$.

Let $A \coefT^{k} = \coefC$ represent the linear system in Step (a) (ignoring the tangent space constraint), where $A$ is a $d n \times d n$ matrix, and $\coefT^{k}$, $\coefC$ are column vectors in $\R^{d n}$. Note that the solution vector $\coefT^{k}$ is the coefficient vector associated with the finite element function $\vt_h^{k}$ in $\Yh (\vn_h^{k})$ (recall Section \ref{algorithm}).

Multiplying the linear system by a column vector $\coefV$ in $\R^{d n}$, we seek to find $\coefT^{k}$ in $\R^{d n}$ such that
\begin{equation}\label{eqn:discrete_variational_matrix_eqn}
  \coefV\tp A \coefT^{k} = \coefV\tp \coefC, \quad \text{for all } \coefV \in \R^{d n}.
\end{equation}
Next, using $\vn_h^{k}$ in $\Nh$, find $\vq_h^k$, $\vw_h^k$ in $\Yh (\vn_h^{k})$ such that $\{ \vn_h^{k}(x_i), \vq_h^k(x_i), \vw_h^k(x_i) \}$ forms an orthonormal basis of $\R^3$ at each node $x_i$, i.e. find an orthonormal basis of $\Uh$.  Let $\coefQ^{k}$, $\coefW^{k}$ in $\R^{d n}$ be the coefficient vectors associated with $\vq_h^k$, $\vw_h^k$.

Since $\vt_h^{k}$ is in the tangent space, we can expand $\coefT^{k}$ as
\begin{equation}\label{eqn:vt_tangent_expansion}
  \coefT^{k}_{i + (r-1)n} = \coefF_i \coefQ^{k}_{i + (r-1)n} + \coefG_i \coefW^{k}_{i + (r-1)n}, \quad \text{for all } 1 \leq i \leq n, 1 \leq r \leq 3,
\end{equation}
where $\coefF = (\coefF_1, ..., \coefF_n)\tp$, $\coefG = (\coefG_1, ..., \coefG_n)\tp$ are unknown solution (column) vectors in $\R^{n}$.  With this, we can write the expansion of $\coefT^{k}$ as
\begin{equation}\label{eqn:coef_soln_vector_tangent_decomp}
  \coefT^{k} =
\left[
  \begin{array}{c}
    D_1 (\coefQ^{k}) \\
    D_2 (\coefQ^{k}) \\
    D_3 (\coefQ^{k}) \\
  \end{array}
\right] \coefF +
\left[
  \begin{array}{c}
    D_1 (\coefW^{k}) \\
    D_2 (\coefW^{k}) \\
    D_3 (\coefW^{k}) \\
  \end{array}
\right] \coefG,
\end{equation}
where $D_1 (\cdot)$, $D_2 (\cdot)$, $D_3 (\cdot)$ are $n \times n$ diagonal matrices defined by
\begin{equation}\label{eqn:diag_exp_matrices}
\begin{split}
    D_r (\coefQ^{k}) &:= \mathrm{diag} [ \coefQ^{k}((r-1)n+1 : r n) ], ~ \text{ for } r = 1,2,3, \\
    D_r (\coefW^{k}) &:= \mathrm{diag} [ \coefW^{k}((r-1)n+1 : r n) ], ~ \text{ for } r = 1,2,3.
\end{split}
\end{equation}
Furthermore, we make a similar tangential expansion for $\coefV$:
\begin{equation}\label{eqn:coef_test_vector_tangent_decomp}
  \coefV =
\left[
  \begin{array}{c}
    D_1 (\coefQ^{k}) \\
    D_2 (\coefQ^{k}) \\
    D_3 (\coefQ^{k}) \\
  \end{array}
\right] \coefY +
\left[
  \begin{array}{c}
    D_1 (\coefW^{k}) \\
    D_2 (\coefW^{k}) \\
    D_3 (\coefW^{k}) \\
  \end{array}
\right] \coefZ,
\end{equation}
where $\coefY = (\coefY_1, ..., \coefY_n)\tp$, $\coefZ = (\coefZ_1, ..., \coefZ_n)\tp$ are arbitrary column vectors in $\R^{n}$.

Now note that $A$ is symmetric block diagonal with $d$ \emph{identical} blocks $\widetilde{A}$ (recall \eqref{eqn:discrete_sub_energy_delta_vN}).  And the orthogonality of $\vq_h^k$ and $\vw_h^k$ at the nodes is equivalent to
\begin{equation}\label{eqn:coef_vector_ortho}
  \coefQ^k_{i} \coefW^k_{i} + \coefQ^k_{n + i} \coefW^k_{n + i} + \coefQ^k_{2 n + i} \coefW^k_{2 n + i} = \sum^3_{r = 1} \coefQ^{k}_{(r-1)n + i} \coefW^{k}_{(r-1)n + j} = 0,
\end{equation}
for all $1 \leq i \leq n$.  From this, one can show that
\begin{equation}\label{eqn:zero_matrix_cross_term}
\begin{split}
\left[
  \begin{array}{ccc}
    D_1 (\coefQ^{k}) & D_2 (\coefQ^{k}) & D_3 (\coefQ^{k}) \\
  \end{array}
\right] A
\left[
  \begin{array}{c}
    D_1 (\coefW^{k}) \\
    D_2 (\coefW^{k}) \\
    D_3 (\coefW^{k}) \\
  \end{array}
\right] &= \\
\sum^3_{r = 1} D_r (\coefQ^{k}) \widetilde{A} D_r (\coefW^{k}) &= 0 \in \R^{n \times n}.
\end{split}
\end{equation}
Indeed, looking at the $(i,j)$ entry and using \eqref{eqn:diag_exp_matrices}, we have
\begin{equation*}
\begin{split}
  \sum^3_{r = 1} [ D_r (\coefQ^{k}) \widetilde{A} D_r (\coefW^{k}) ]_{ij} &= \sum^3_{r = 1} \coefQ^{k}_{(r-1)n + i} (\widetilde{A})_{ij} \coefW^{k}_{(r-1)n + j} \\
  &= (\widetilde{A})_{ij} \sum^3_{r = 1} \coefQ^{k}_{(r-1)n + i} \coefW^{k}_{(r-1)n + j} = 0,
\end{split}
\end{equation*}
by \eqref{eqn:coef_vector_ortho}.

Therefore, plugging the expansions \eqref{eqn:coef_soln_vector_tangent_decomp}, \eqref{eqn:coef_test_vector_tangent_decomp} into \eqref{eqn:discrete_variational_matrix_eqn}, accounting for \eqref{eqn:zero_matrix_cross_term}, and using the arbitrariness of $\coefY$, $\coefZ$, we obtain two decoupled, $n \times n$ linear systems to solve:
\begin{equation}\label{eqn:discrete_variational_matrix_eqn_tangent_valpha}
\begin{split}
A(\coefQ^k) \coefF = \coefC(\coefQ^k), \qquad A(\coefW^k) \coefG = \coefC(\coefW^k),
\end{split}
\end{equation}
which are defined by
\begin{equation*}
\begin{split}
A(\coefQ^k) &=
\left[
  \begin{array}{ccc}
    D_1 (\coefQ^{k}) & D_2 (\coefQ^{k}) & D_3 (\coefQ^{k}) \\
  \end{array}
\right] A
\left[
  \begin{array}{c}
    D_1 (\coefQ^{k}) \\
    D_2 (\coefQ^{k}) \\
    D_3 (\coefQ^{k}) \\
  \end{array}
\right], \\
A(\coefW^k) &=
\left[
  \begin{array}{ccc}
    D_1 (\coefW^{k}) & D_2 (\coefW^{k}) & D_3 (\coefW^{k}) \\
  \end{array}
\right] A
\left[
  \begin{array}{c}
    D_1 (\coefW^{k}) \\
    D_2 (\coefW^{k}) \\
    D_3 (\coefW^{k}) \\
  \end{array}
\right],
\end{split}
\end{equation*}
\begin{equation*}
\begin{split}
\coefC(\coefQ^k) &= \left[
  \begin{array}{ccc}
    D_1 (\coefQ^{k}) & D_2 (\coefQ^{k}) & D_3 (\coefQ^{k}) \\
  \end{array}
\right] \coefC, \\
\coefC(\coefW^k) &= \left[
  \begin{array}{ccc}
    D_1 (\coefW^{k}) & D_2 (\coefW^{k}) & D_3 (\coefW^{k}) \\
  \end{array}
\right] \coefC.
\end{split}
\end{equation*}
After solving for $\coefF$ and $\coefG$, we compute $\coefT^k$ via \eqref{eqn:vt_tangent_expansion} or \eqref{eqn:coef_soln_vector_tangent_decomp}.  This yields the finite element function $\vt_h^{k}$ in Step (a) of Section \ref{algorithm}.

\subsection{Experimental order of convergence}\label{sec:eoc}

We test the accuracy of our method against an exact solution found in \cite[Sec. 6.4]{Virga_book1994} that represents a plane defect (see Figure \ref{fig:Plane_Defect_3D_Director_all}).  The computational domain is a cube $\Om = [0, 1]^3$ and we set $\kappa=0.2$.  The double well potential is removed, so the energy is $E[s,\vn] \equiv E_1[s,\vn]$.

\begin{figure}[h!]
\begin{center}
\subfloat{

\includegraphics[width=2.5in]{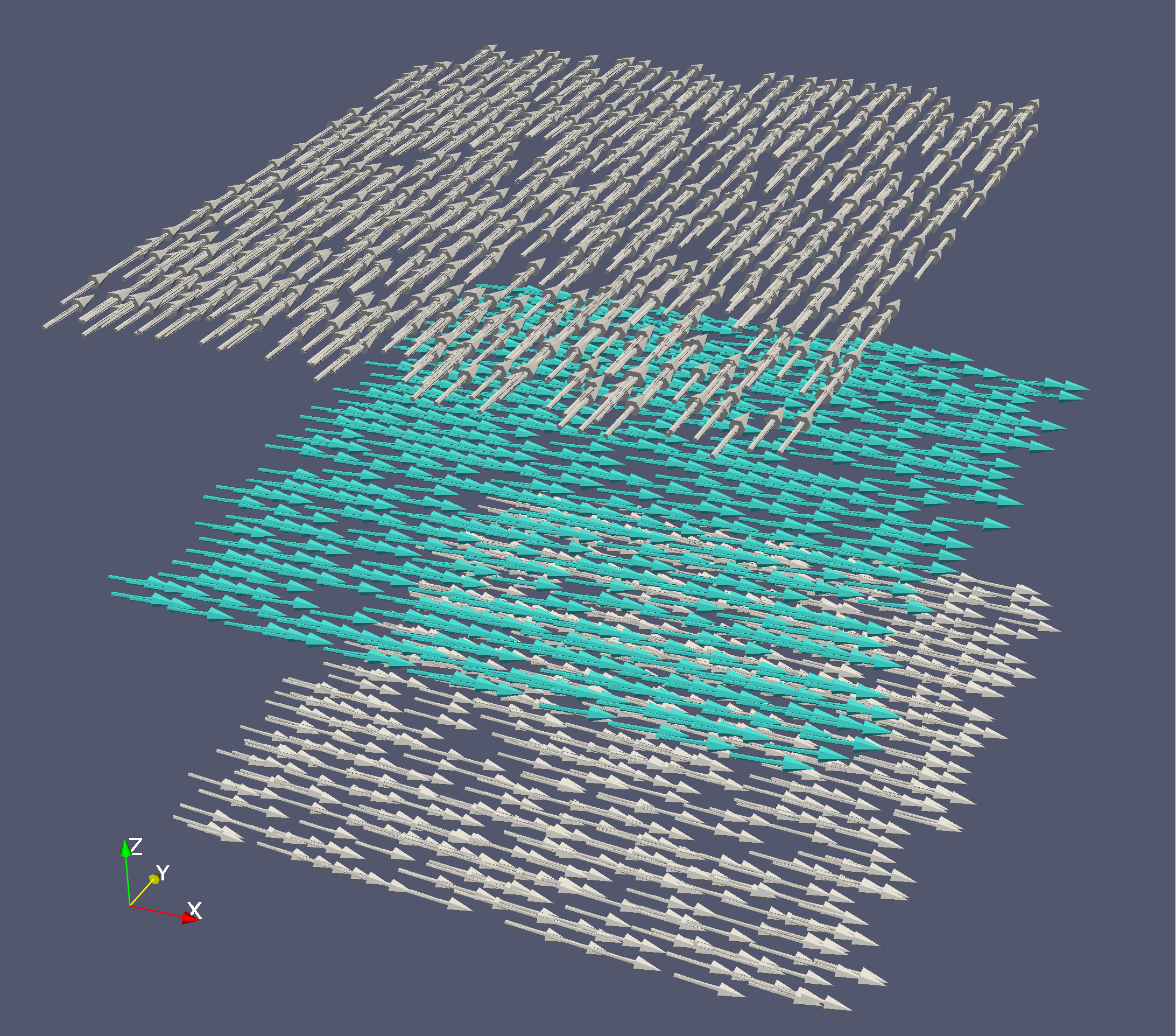}
}
\hspace{0.1in}
\subfloat{

\includegraphics[width=2.5in]{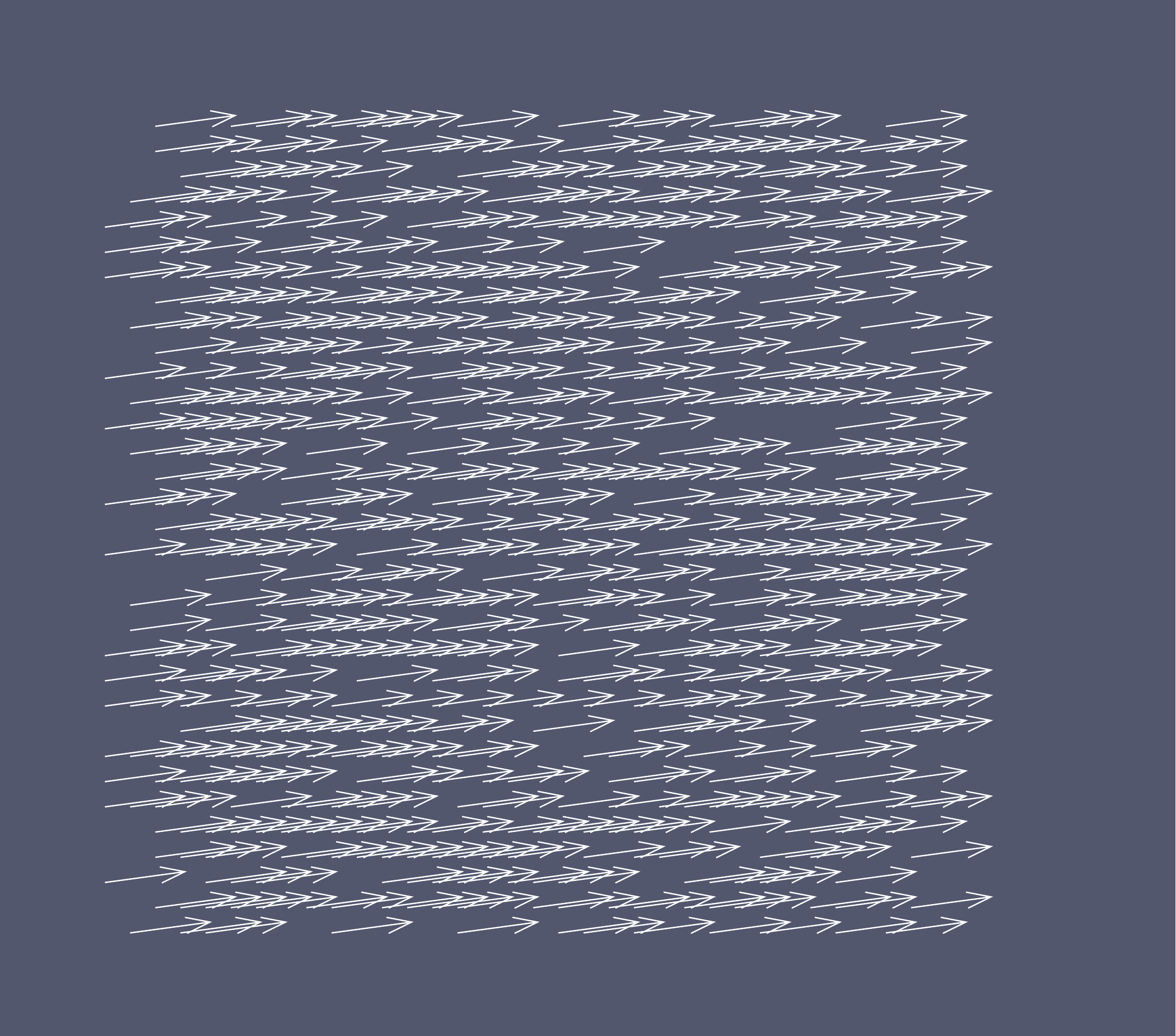}
}
\caption{Simulation results for Section \ref{sec:eoc}.  The director field is shown on the planes $z = 0.1, 0.5, 0.9$, which indicates that it is piecewise constant (see \eqref{eqn:bc_plane_defect_exact_soln}).  On the right, the $z=0.5$ plane case indicates that the director field takes on an intermediate value between $(1,0,0)\tp$ and $(0,1,0)\tp$.  The numerical solution was computed at mesh level $\ell=5$.}
\label{fig:Plane_Defect_3D_Director_all}
\end{center}
\end{figure}

The following Dirichlet boundary conditions on $\overline{\partial \Om} \cap ( \{ z = 0\} \cup \{ z = 1\} )$ are imposed for $(s,\vn)$:
\begin{equation}\label{eqn:bc_plane_defect}
\begin{split}
  z=0:& \quad s = s^*, \qquad \vn = (1,0,0), \\
  z=1:& \quad s = s^*, \qquad \vn = (0,1,0),
\end{split}
\end{equation}
and Neumann conditions are imposed on the remaining part of $\partial \Om$, i.e. $\vnu \cdot \nabla s = 0$ and $\vnu \cdot \nabla \vn = 0$.  The exact solution $(s,\vn)$ (at equilibrium) only depends on $z$ and is given by
\begin{equation}\label{eqn:bc_plane_defect_exact_soln}
\begin{split}
  \vn(z) &= (1,0,0), ~ \text{for } z < 0.5, \quad  \vn(z) = (0,1,0), ~ \text{for } z > 0.5, \\
  s(z) &= 0, ~\text{at } z = 0.5, \text{ and } s(z) \text{ is linear for } z \in (0, 0.5) \cup (0.5, 1.0).
\end{split}
\end{equation}
Figure \ref{fig:Plane_Defect_3D_Director_all} gives an illustration of $\vn$, and Figure \ref{fig:Plane_Defect_3D_Scalar} shows a one-dimensional slice of $s$.
\begin{figure}[h!]
\begin{center}

\includegraphics[width=4.5in]{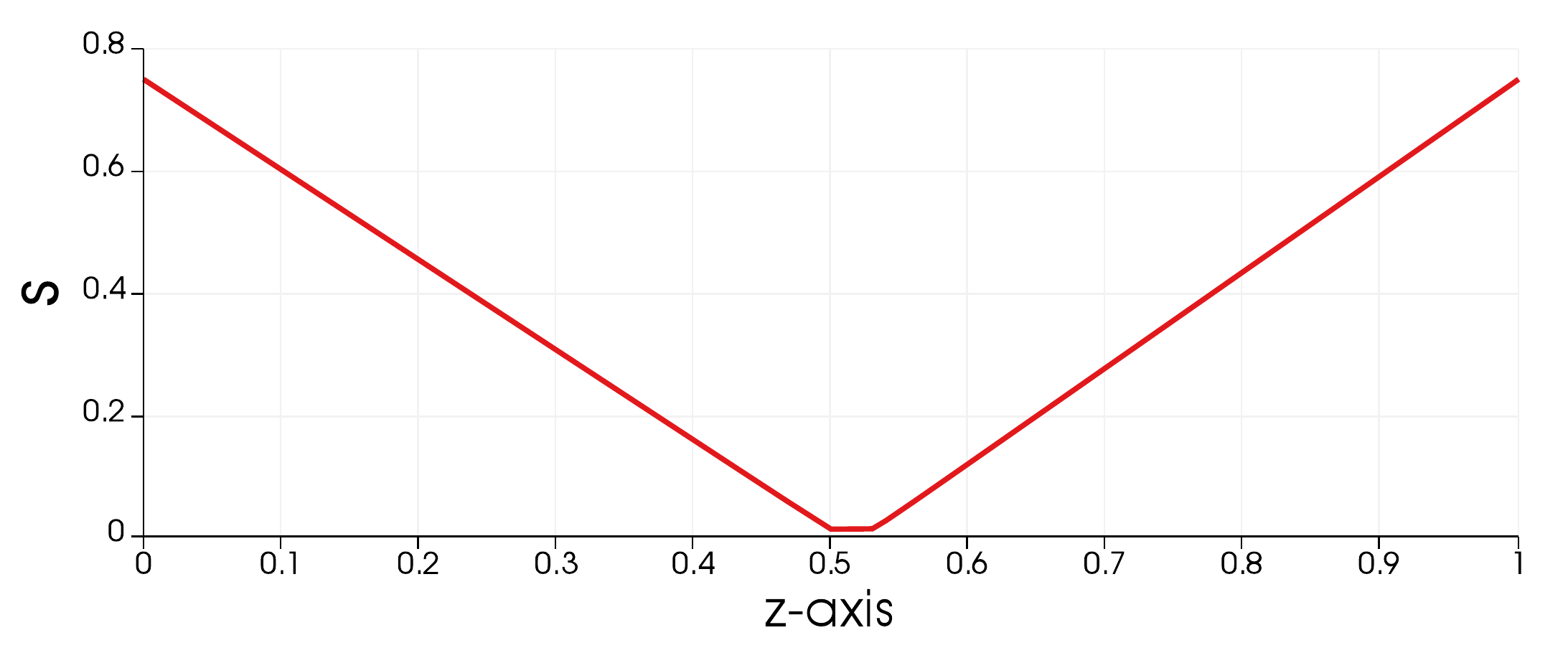}

\caption{Simulation results for Section \ref{sec:eoc}.  The scalar field $s$ is shown evaluated along a line parallel to the $z$-axis and passing through $x = 0.5$, $y = 0.5$.  It is close to the exact piecewise linear solution in \eqref{eqn:bc_plane_defect_exact_soln}. The numerical solution was computed at mesh level $\ell=5$.}
\label{fig:Plane_Defect_3D_Scalar}
\end{center}
\end{figure}

\begin{table}[h!]
\begin{center}
\caption{\label{tbl:plane_defect_3D_EOC} Numerical error vs. mesh refinement for plane defect in three dimensions.  The estimated order of convergence (EOC) is given in the last row.}
\begin{tabular}{|c|c|c|c|c|c|}
  \hline
  Level $\ell$ & $\| s - s_h \|_{L^2(\Om)}$ & $\| s - s_h \|_{H^1(\Om)}$ & $\| \vn - \vn_h \|_{L^2(\Om)}$ & $\| \vu - \vu_h \|_{L^2(\Om)}$ & $\| \vu - \vu_h \|_{H^1(\Om)}$ \\ \hline
  3 & 5.5087E-02 & 5.5090E-01 & 2.6693E-01 & 5.7355E-02 & 4.6602E-01 \\
  4 & 2.9158E-02 & 3.9858E-01 & 2.0545E-01 & 2.9840E-02 & 3.2646E-01 \\
  5 & 1.4981E-02 & 2.7986E-01 & 1.4642E-01 & 1.5207E-02 & 2.2661E-01 \\
  6 & 7.5964E-03 & 1.9726E-01 & 1.0398E-01 & 7.6800E-03 & 1.5878E-01 \\ \hline \hline
  EOC & 0.9797 & 0.5046 & 0.4938 & 0.9855 & 0.5132 \\
  \hline
\end{tabular}
\end{center}
\end{table}
Numerical errors are given in Table \ref{tbl:plane_defect_3D_EOC}.
The meshes were created by partitioning $\Om$ into $(2^\ell)^3$
uniform cubes, where $\ell$ is the mesh level, and sub-dividing each
cube into six non-obtuse tetrahedra.  The estimated order of
convergence is given in the last row of the table.  The $L^2$-accuracy
appears to be first order for both $(s,\vu)$, the smoother
variables, and half order for $\vn$. Since $\vn$ is discontinuous
across the plane $\{x_3=0\}$ (plane defect), we cannot expect better accuracy in
$L^2$ with continuous elements; moreover, $s$ does not have better regularity
than $H^1$.
Furthermore, our discrete energy $E_1^h$ uses a first order approximation
\eqref{discrete_energy} of $\iO s^2 |\nabla \vn|^2$,
which is accounted for by the consistency
errors \eqref{eqn:residual}. These two facts are most likely responsible for the
reduced linear order for $(s,\vu)$.

\section{Colloidal effects}\label{sec:colloids}

Colloidal particles immersed in a liquid crystal can induce interesting equilibrium states with non-trivial defect configurations.  One example is the famous Saturn ring defect \cite{Gu_PRL2000, Alama_PRE2016}, which is a circular ring of defect surrounding a spherical hole inside the liquid crystal domain (see Figure \ref{fig:Diagram_Saturn_Ring_Defect}), i.e. $\Om$ is the region outside the sphere.
\begin{figure}
\begin{center}

\includegraphics[width=2.5in]{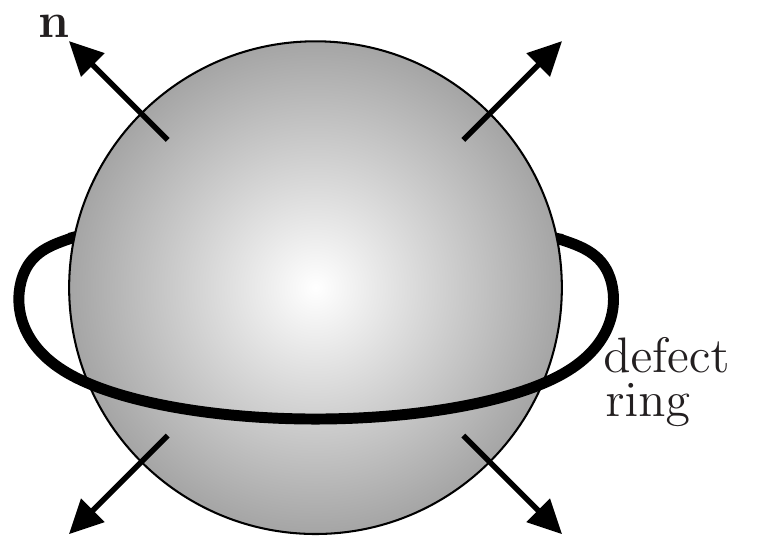}

\caption{Illustration of the Saturn ring defect.  A spherical colloidal particle is shown with normal anchoring conditions on its boundary (i.e. the director field $\vn$ is parallel to the normal vector $\vnu$ of the sphere).  The region where $s=0$ (i.e. the singular set $\Sing$) is marked by the thick curve and occurs depending on the outer boundary conditions (away from the sphere) imposed on $\vn$.}
\label{fig:Diagram_Saturn_Ring_Defect}
\end{center}
\end{figure}

In this section, we demonstrate that the Ericksen model and our
numerical method are able to capture interesting defect structures in the presence of colloids. The colloid particle is modeled as a spherical inclusion inside the liquid crystal domain.  Section \ref{sec:conforming_mesh} shows a direct simulation (with a conforming mesh) which gives rise to a Saturn ring-like defect structure (depending on outer boundary conditions).  In Section \ref{sec:immersed_boundary}, we combine our method for the Ericksen model with an immersed boundary approach and compare with our conforming mesh approach.  For both subsections, we use the following notation.  The liquid crystal domain is denoted by $\Om$, with boundary $\dOm$ that decomposes into a disjoint union $\dOm = \Gm_{i} \cup \Gm_{o}$, where $\Gm_{i}$ is the boundary of the interior ``hole'' and $\Gm_{o}$ is the outer boundary of the cylindrical domain that contains the hole (see Figures \ref{fig:LC_Sphere_Inclusion_BCs_A} and \ref{fig:LC_Sphere_Inclusion_BCs_B}).

\subsection{Conforming non-obtuse mesh}\label{sec:conforming_mesh}

\subsubsection{Meshing the domain}

It is quite difficult to generate a conforming, \textbf{non-obtuse}, tetrahedral mesh of a general domain; in fact, it is still an open question whether it is always possible to generate a non-obtuse tetrahedral mesh of a general three dimensional domain.  For our purposes, we managed to create a non-obtuse mesh of a cylindrical domain with a hole cut out, but the procedure is not general.  However, the resulting mesh is valid for testing our method.

We start by describing the domain $\Om$, which is essentially a cylinder with square cross-section with a spherical hole removed from the interior.  First, we create a tetrahedral mesh of a rectangular solid with the following dimensions: $[-0.5, 1.5] \times [-0.5, 1.5] \times [-2.5, 3.5]$.  We partition the solid into $8\times 8 \times 24$ uniformly sized cubes, of side length $h=0.25$, and further partition each cube into six tetrahedra.

Next, we shear the mesh by mapping the vertices with the following linear map:
\begin{equation}\label{eqn:ideal_tet_map}
\vx =
\left[
  \begin{array}{c}
    x \\
    y \\
    z \\
  \end{array}
\right]
=
\left[
  \begin{array}{ccc}
    1/\sqrt{2} & 0 & 0 \\
    0 & 1/\sqrt{2} & 0 \\
    -1/2 & -1/2 & 1 \\
  \end{array}
\right] \hat{\vx},
\end{equation}
where $\hat{\vx}$ are coordinates in the initial rectangular solid.  This results in a mesh of so-called \emph{ideal} tetrahedra, whose circumcenters coincide with their barycenters \cite{VanderZee_SJSC2010, Walton_2015} (so-called ``well-centered'' tetrahedra).  The resulting ``prism'' is a cylinder with square cross-section given by $[-0.25 \sqrt{2}, 0.75 \sqrt{2}]^2$ and is centered about the $z=0$ plane (see Figure \ref{fig:non-obtuse_mesh}).

\begin{figure}[h!]
\begin{center}

\includegraphics[width=1.55in]{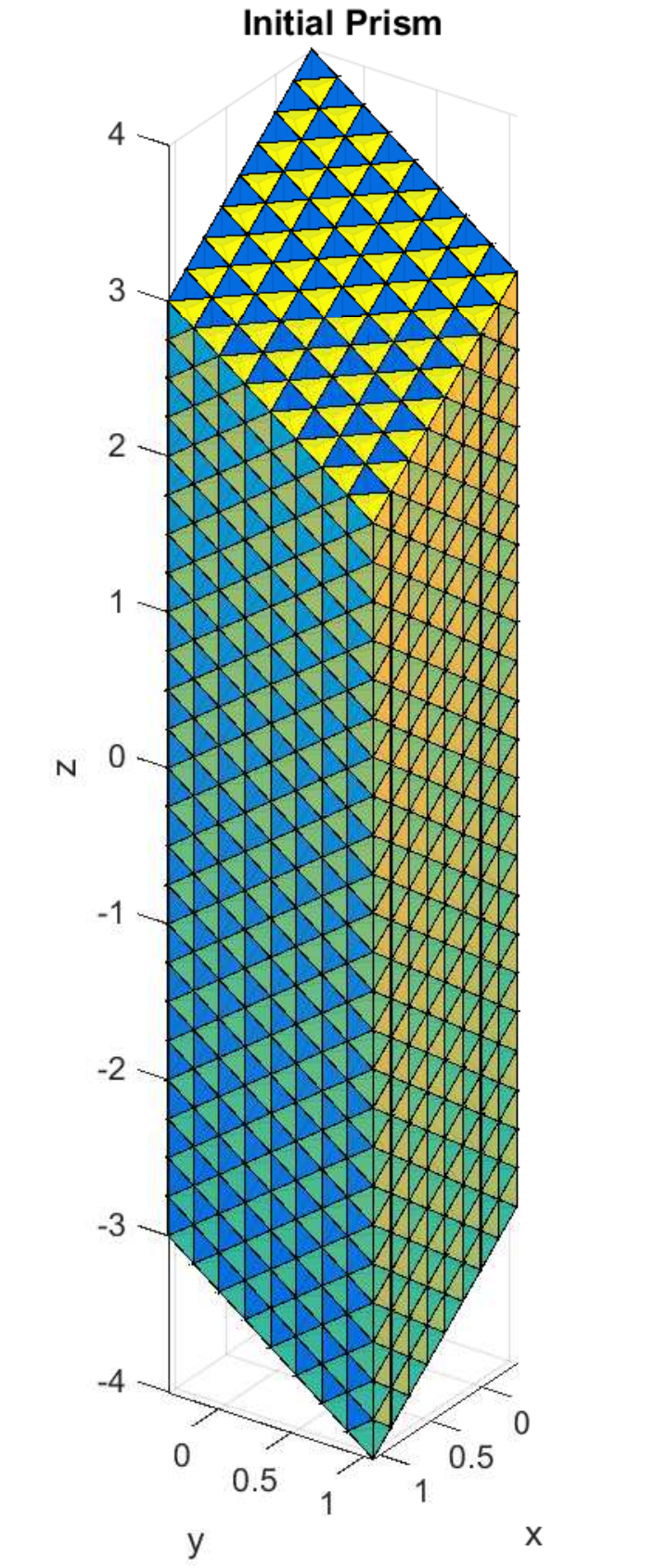}
\hspace{0.5in}
\includegraphics[width=1.3in]{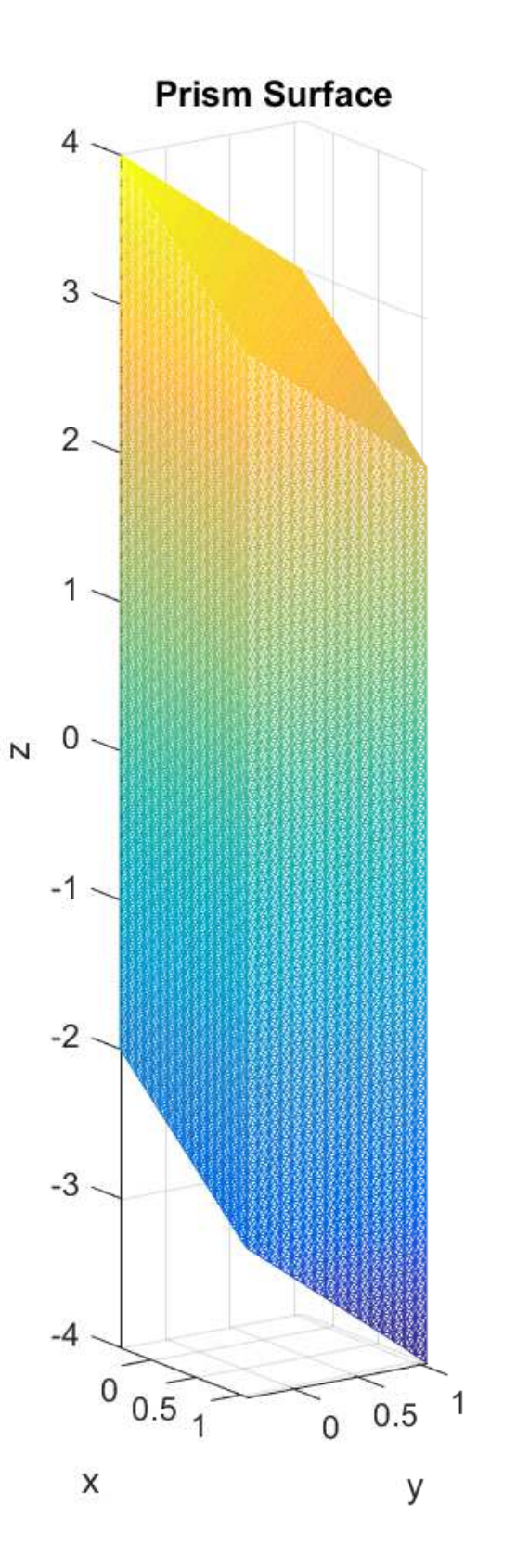}

\caption{Non-obtuse mesh of cylindrical domain with interior spherical hole.
Left figure depicts the initial mesh obtained by first
  partitioning a right prism into $8\times8\times24$ cubes and dividing
  them into six tetrahedra, and next mapping them via
  \eqref{eqn:ideal_tet_map}. Right picture (rotated 90 degrees) displays
  the final mesh, with $1,764,864$ tetrahedra, after applying a 1-to-8 refinement
  and a 1-to-24 ``yellow'' refinement to the previous mesh.}
\label{fig:non-obtuse_mesh}
\end{center}
\end{figure}

We then remove all tetrahedra (from the prism mesh) whose circumcenters are inside a sphere of radius $R = 0.283 / \sqrt{2}$ centered at $(0.5/\sqrt{2},0.5/\sqrt{2},0)\tp$.  The internal cavity represents our spherical colloid.  We make a small adjustment of the vertex positions on the boundary of the cavity so that they lie exactly on the given sphere boundary.  Thus, the tetrahedra are slightly off from being exactly well-centered.

In order to have a more accurate simulation, we apply two well chosen refinements in the following way.  We use a standard 1-to-8 uniform refinement of the mesh \cite{Joe_MC1996}, while choosing the best diagonal to maintain the well-centered property.  The (new) vertices on the boundary of the internal cavity are \emph{adjusted} so that the mesh conforms to the sphere.  The resulting mesh is not as well-centered, but all tetrahedra still strictly contain their circumcenters.  Therefore, we use the ``yellow'' refinement described in \cite[pg. 1108-1109]{Korotov_CMA2005}, which partitions each tetrahedron into 24 tetrahedra where each new tetrahedron has the circumcenter as a vertex.  This final refinement is guaranteed to yield a non-obtuse mesh; see Figure \ref{fig:non-obtuse_mesh} for a view of the surface mesh of the prism.

\begin{figure}[h]
\begin{center}

\includegraphics[width=2.6in]{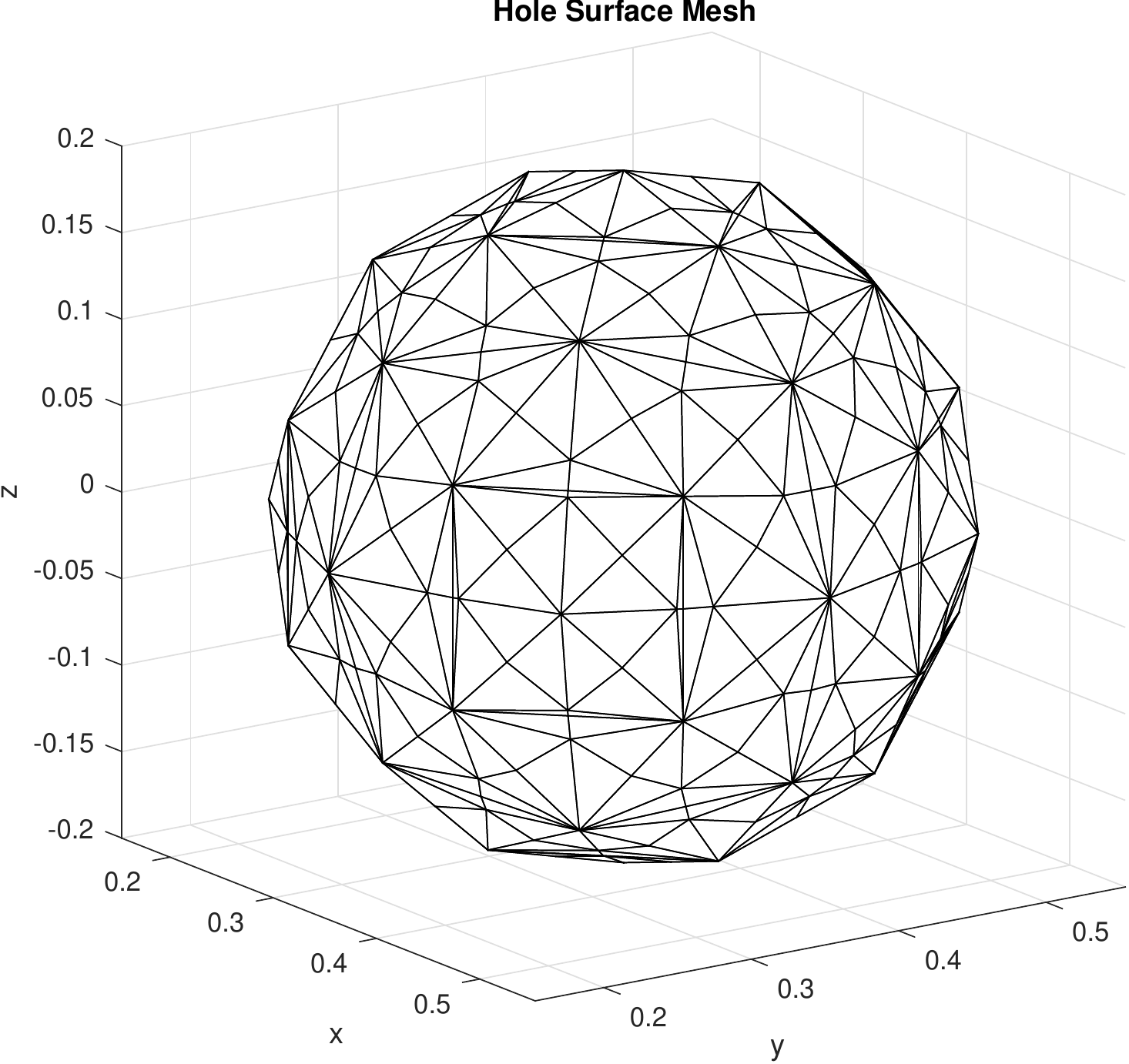}

\caption{Surface mesh of the interior hole. Not all vertices lie on
the spherical cavity, so as to ensure that all tetrahedra are non-obtuse, thereby
giving a slightly faceted surface mesh. This is compatible with the
Saturn ring defect not being very sensitive to fine details of the hole geometry.}
\label{fig:hole_surface_mesh}
\end{center}
\end{figure}
Note that we do not adjust the new vertices (generated by the second refinement) to lie on the spherical hole's boundary.  Any adjustment seems to yield an obtuse mesh, because the domain with hole is not convex.  Thus, the surface mesh of the internal spherical hole boundary is slightly faceted (see Figure \ref{fig:hole_surface_mesh}).  This is allowable because the defect structures of interest are not very sensitive to the fine details of the geometry of the hole.

The final mesh has $1,764,864$ tetrahedra and $329,698$ vertices. The dihedral angles are between $4.156^\circ$ and $90^\circ$ (non-obtuse); the surface mesh angles are between $2.726^\circ$ and $90^\circ$.  The minimum angles are not great, but acceptable for numerical simulation.

\subsubsection{Simulating a disperse/point defect}\label{sec:conform_disperse-point_defect}

\begin{figure}[ht]
\begin{center}

\includegraphics[width=1.7in]{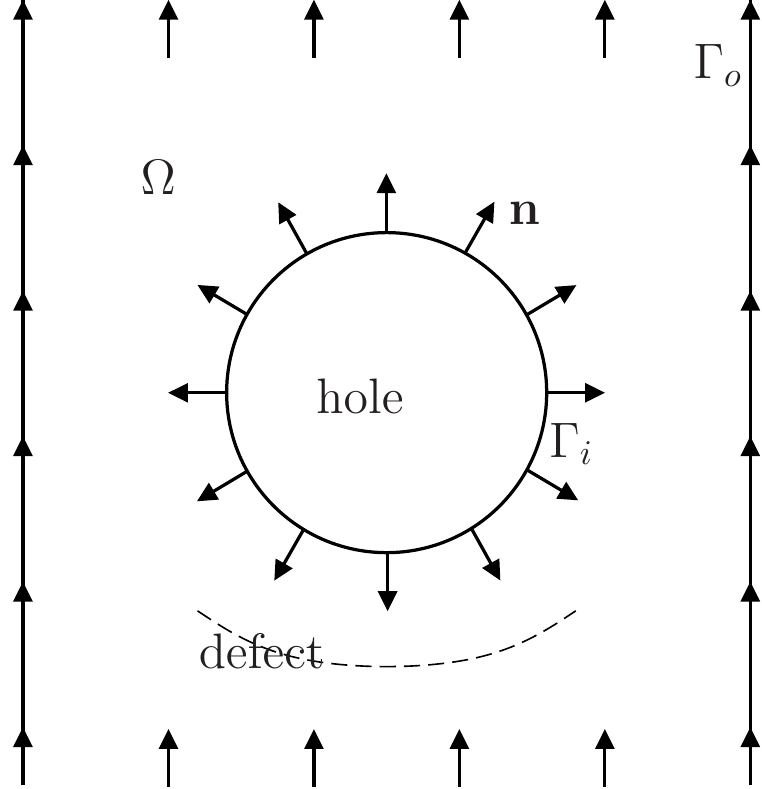}

\caption{Boundary conditions for a disperse/point defect (Section \ref{sec:conform_disperse-point_defect}).  One can see why a defect arises because of the incompatibility of the director fields on the bottom of the sphere and the bottom of the prism.}
\label{fig:LC_Sphere_Inclusion_BCs_A}
\end{center}
\end{figure}

Consider the boundary conditions shown in Figure \ref{fig:LC_Sphere_Inclusion_BCs_A}, which is the director field version of the Landau-deGennes model considered in \cite{Alama_PRE2016}.  The precise strong anchoring condition is given by
\begin{equation}\label{eqn:disperse_BCs}
  \vn = \vnu, ~ \text{on } \Gm_{i}, \quad
  \vn = (0,0,1)\tp, ~ \text{on } \Gm_{o}, \quad
  s = s^*, ~ \text{on } \dOm,
\end{equation}
where $\vnu$ is the outer normal vector of the spherical inclusion,
and $s^*$ is the global minimum of the double well potential $\psi$.
Moreover, the double well potential has the convex splitting
$\psi (s) = (0.3)^{-2} \left( \psi_c (s) - \psi_e (s) \right)$
for $-\frac{1}{2} < s < 1$, where
\begin{equation}\label{eqn:double_well_defn}
\begin{split}
  \psi_c (s) := 63.0 s^2
  \qquad
  \psi_e (s) := -16.0 s^4 + 21.33333333333 s^3 + 57.0 s^2,
\end{split}
\end{equation}
with a local minimum at $s=0$ and global minimum at $s=s^* := 0.750025$.  The initial conditions in $\Om$ for the gradient flow are: $s = s^*$ and $\vn = (0,0,1)\tp$.

The equilibrium solution, for $\kappa = 0.1$, is shown in Figure \ref{fig:Prism_Sphere_Hole_kappa_0_1_all}.  The low value of $\kappa$ leads to a large disperse defect region, which is induced by the ``frustrated'' boundary conditions between the bottom of the sphere and the bottom of the cylinder (see Figure \ref{fig:LC_Sphere_Inclusion_BCs_A}).
\begin{figure}
\begin{center}
\subfloat{

\includegraphics[width=2.5in]{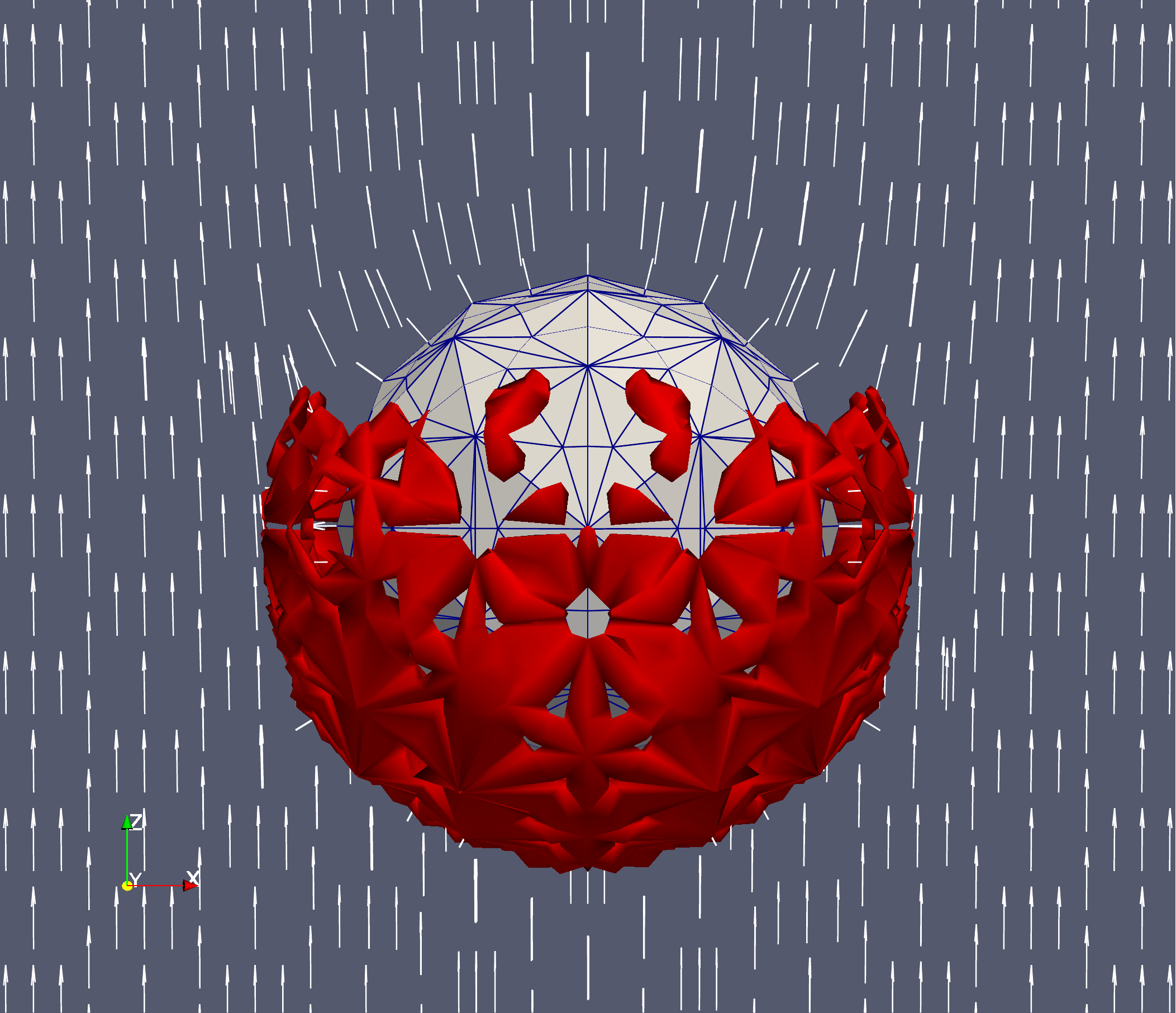}
}
\hspace{0.1in}
\subfloat{

\includegraphics[width=2.5in]{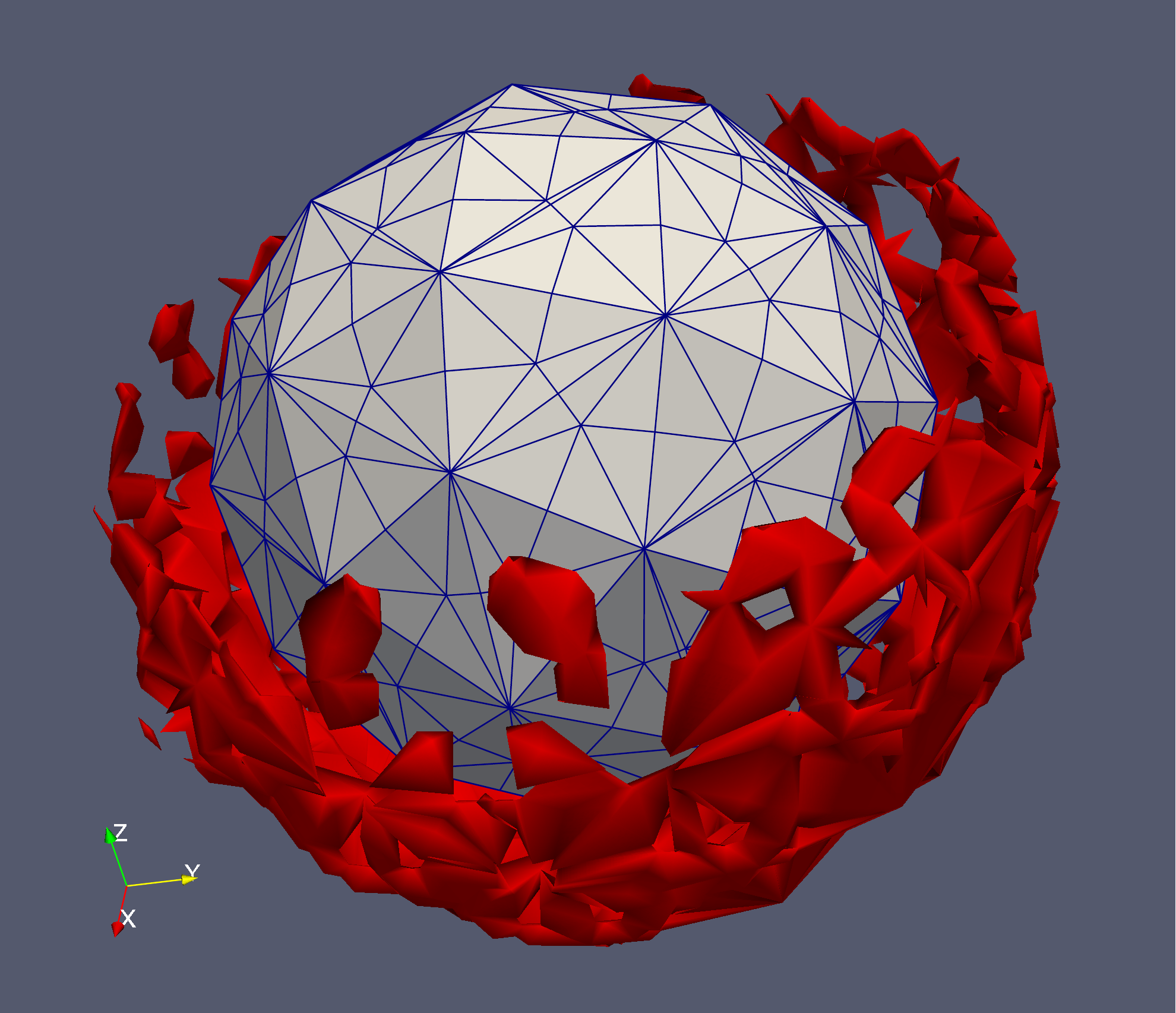}
}
\caption{Simulation results for Figure \ref{fig:LC_Sphere_Inclusion_BCs_A} ($\kappa = 0.1$).  The surface mesh of the internal hole is shown and the $s=0.04$ iso-surface is plotted in red which indicates the defect region; the director field is depicted with white arrows.  The disperse defect region has a bowl-like shape underneath the internal hole.}
\label{fig:Prism_Sphere_Hole_kappa_0_1_all}
\end{center}
\end{figure}

A different equilibrium solution is obtained with $\kappa = 1.0$, which is shown in Figure \ref{fig:Prism_Sphere_Hole_kappa_1_all}.  The larger value of $\kappa$ leads to a smaller defect region compared to Figure \ref{fig:Prism_Sphere_Hole_kappa_0_1_all}.  The center of the hole is $\approx (0.354, 0.354, 0)\tp$ and the location of the defect region is $\approx (0.354, 0.354, -0.275)\tp$.
\begin{figure}
\begin{center}
\subfloat{

\includegraphics[width=2.5in]{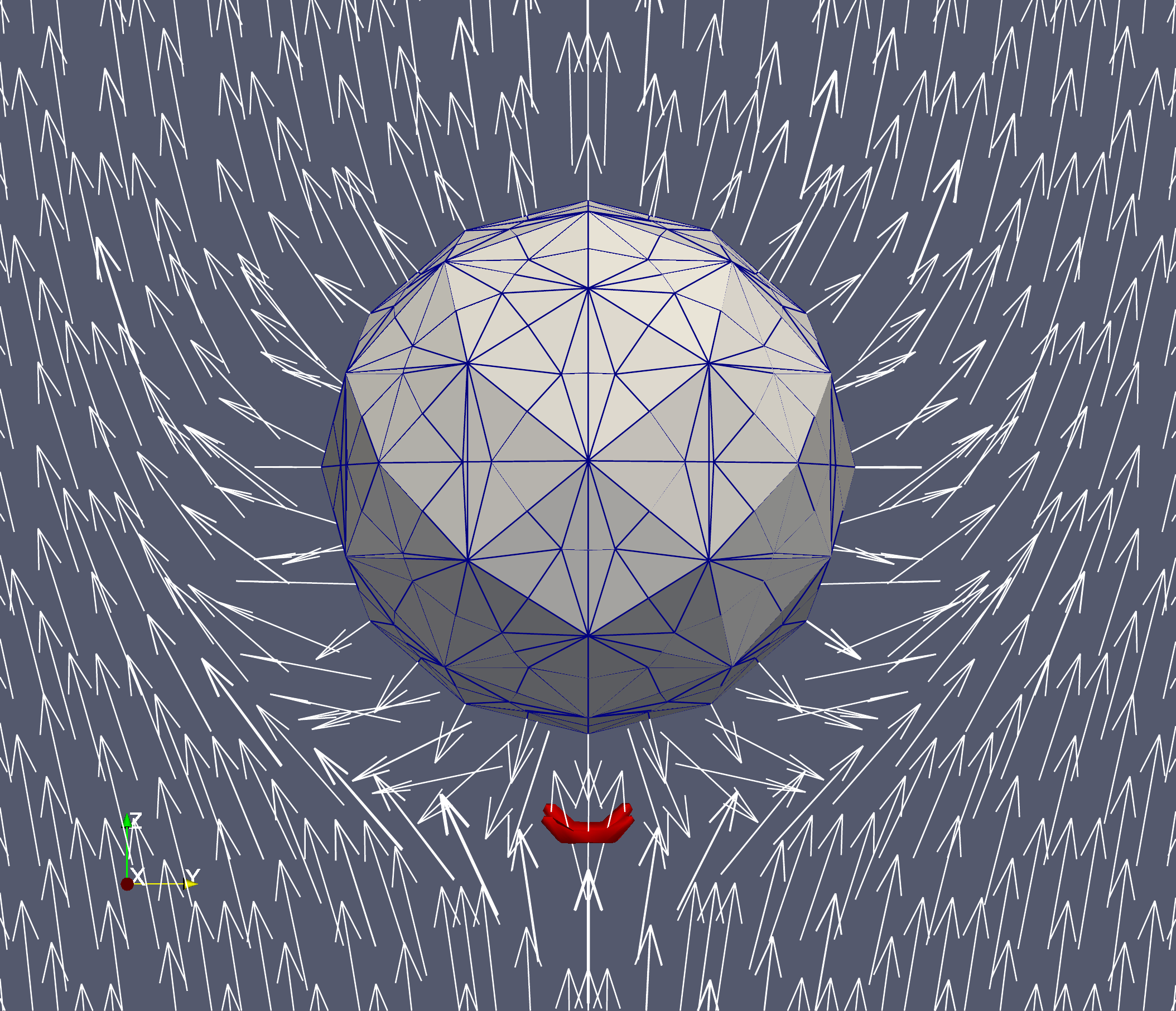}
}
\hspace{0.1in}
\subfloat{

\includegraphics[width=2.5in]{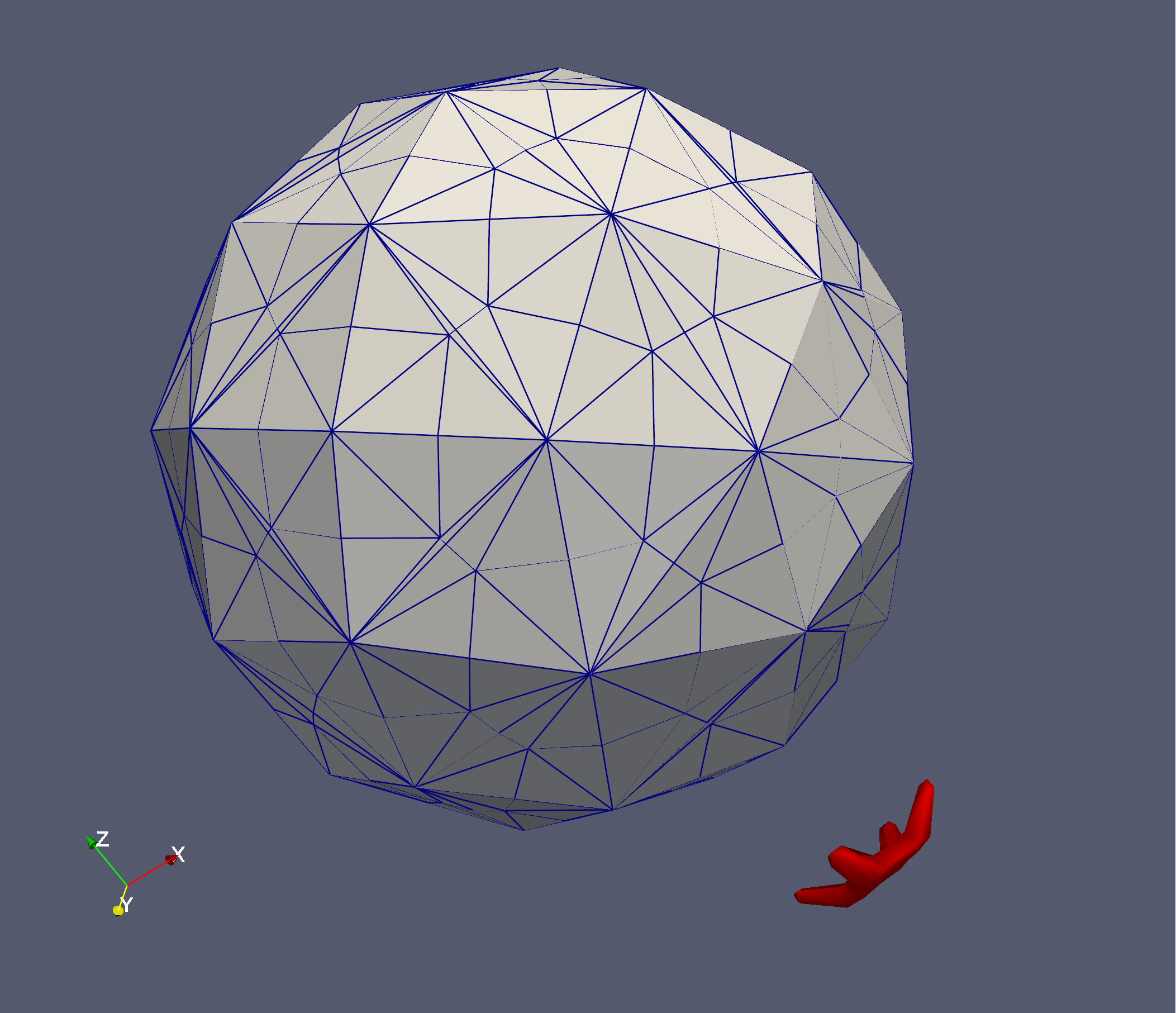}
}
\caption{Simulation results for Figure \ref{fig:LC_Sphere_Inclusion_BCs_A} ($\kappa = 1.0$).  The surface mesh of the internal hole is shown and the $s=0.1$ iso-surface is plotted in red which indicates the defect region; the director field is depicted with white arrows.  The defect region is more localized to a ``point'' below the internal hole (c.f. Figure \ref{fig:Prism_Sphere_Hole_kappa_0_1_all}).}
\label{fig:Prism_Sphere_Hole_kappa_1_all}
\end{center}
\end{figure}

\subsubsection{Simulating a Saturn ring-like defect}\label{sec:conform_ring_defect}

Consider the boundary conditions shown in Figure \ref{fig:LC_Sphere_Inclusion_BCs_B}, which is another director field version of the Landau-deGennes model considered in \cite{Alama_PRE2016}.  The strong anchoring condition is given by
\begin{equation}\label{eqn:saturn_ring_BCs}
  \vn = \vnu, ~ \text{on } \Gm_{i}, \quad
  s = s^*, ~ \text{on } \dOm,
\end{equation}
where $\vn$ smoothly interpolates between $(0,0,-1)\tp$ and $(0,0,1)\tp$ on $\Gm_{o}$.  The same double well potential is used as in \eqref{eqn:double_well_defn}.  The initial conditions in $\Om$ for the gradient flow are: $s = s^*$ and
\begin{equation*}
\begin{split}
  \vn(x,y,z) &= (0,0,-1)\tp, ~ \text{ if } z < 0, \\
  \vn(x,y,z) &= (0,0,+1)\tp, ~ \text{ if } z \geq 0.
\end{split}
\end{equation*}
\begin{figure}
\begin{center}

\includegraphics[width=1.7in]{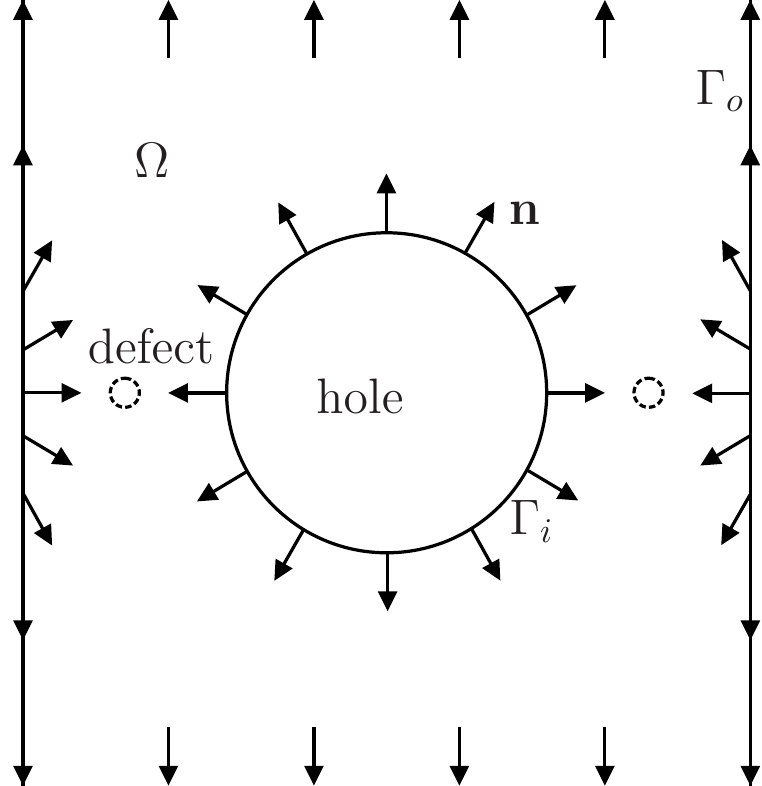}

\caption{Boundary conditions for Section \ref{sec:conform_ring_defect}.  A ring-like defect arises in this case because the director field on the sides of the cylinder are incompatible with the director field on the sphere.}
\label{fig:LC_Sphere_Inclusion_BCs_B}
\end{center}
\end{figure}

The equilibrium solution, for $\kappa = 1.0$, is shown in Figure
\ref{fig:Prism_Sphere_Hole_BC2_kappa_1_all}.  The choice of boundary
conditions in Figure \ref{fig:LC_Sphere_Inclusion_BCs_B} essentially
induces the Saturn ring defect.  The hole's radius is $\approx
0.200111$ and the radius of the Saturn ring is $\approx 0.314$.  Note
that the structure of the director field is not the same as would be
obtained with the Landau-deGennes model \cite{Alama_PRE2016}.  For
instance, the line field in the Saturn ring defect structure of
\cite{Alama_PRE2016} displays a $1/2$ degree point defect,
  whereas in our model the point defect of the director field is of
  degree 1 (see Figure \ref{fig:LC_Sphere_Inclusion_BCs_B}).
This is a limitation of the Ericksen's model, which is sensitive to
the orientation of the director field $\vn$.

\begin{figure}
\begin{center}
\subfloat{

\includegraphics[width=2.5in]{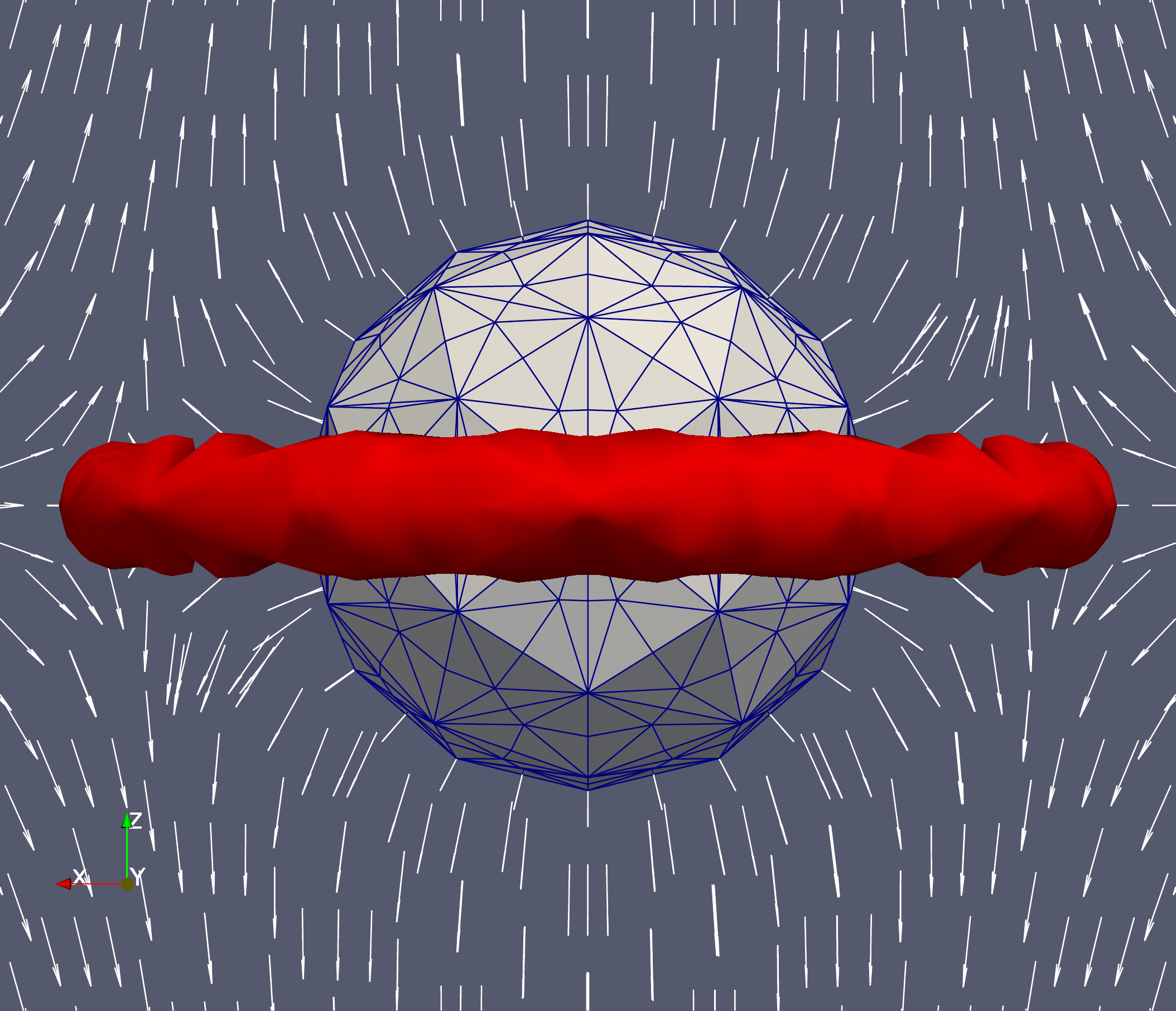}
}
\hspace{0.1in}
\subfloat{

\includegraphics[width=2.5in]{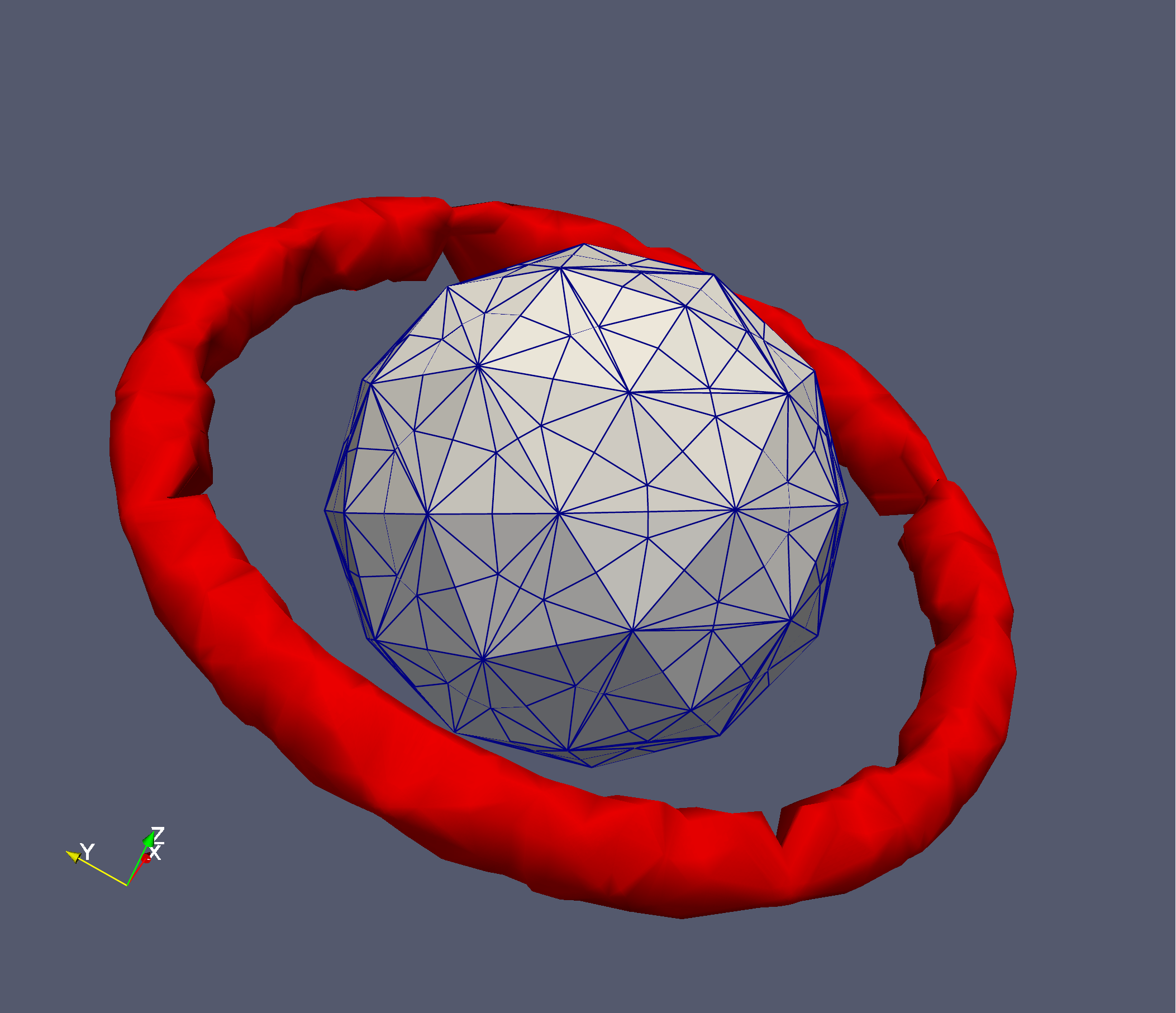}
}
\caption{Simulation results for Figure \ref{fig:LC_Sphere_Inclusion_BCs_B} ($\kappa = 1.0$).  The surface mesh of the internal hole is shown and the $s=0.12$ iso-surface is plotted in red which indicates the defect region; the director field is depicted with white arrows.  The defect region mimics the classic Saturn ring defect.}
\label{fig:Prism_Sphere_Hole_BC2_kappa_1_all}
\end{center}
\end{figure}

\subsection{Immersed boundary method}\label{sec:immersed_boundary}

The weakly acute (or non-obtuse) condition on the mesh is extremely difficult to satisfy in practice in three dimensions.  Therefore, we propose an immersed boundary approach to deal with general colloid shapes.  We define a fixed ``phase-field'' function to represent the colloidal region inside the liquid crystal domain, and add a special ``boundary'' energy term $\Ea$ to enforce boundary conditions on the colloid's boundary.  Specifically, we generalize the continuous \eqref{energy0} and discrete \eqref{discrete_energy} total energies to
\begin{equation}\label{energy_with_colloid}
  E[s, \vn] := E_1 [s, \vn] + E_2 [s] + \Ea[s, \vn],
\end{equation}
\begin{equation}\label{discrete_energy_with_colloid}
  E^h[s_h, \vn_h] := E_1^h [s_h, \vn_h] + E_2^h [s_h] + \Ea^h[s_h, \vn_h],
\end{equation}
where $\Ea$, $\Ea^h$ can take two different forms described in Sections \ref{sec:weak-anchoring} and \ref{sec:penalize_dirichlet}.

\subsubsection{Representing a colloid}

Let $\Om \subset \R^3$ be the ``hold-all'' domain that contains the liquid crystal material \emph{and} colloids.  Moreover, let $\Ochat \subset \R^3$ be the reference domain for a rigid solid (i.e. colloid), and let $\Oc$ be obtained from $\Ochat$ by a rigid motion.  We use $\Oc$ to represent the true colloid domain, with $\Ochat$ as a reference shape.  We assume throughout that $\Oc \subset \subset \Om$.  Thus, the region of interest for the liquid crystals is given by $\Om \setminus \overline{\Oc}$.

Let $\cdisthat : \R^3 \rightarrow \R$ be the signed distance
function to $\partial \Ochat$, i.e.
\begin{equation}\label{eqn:ref_colloid_dist_func}
  \cdisthat(\hat{\vx}) = \dist(\partial \Ochat, \hat{\vx}), \quad \forall \hat{\vx} \in \R^3,
\end{equation}
which is positive inside of $\Ochat$; thus $\partial \Ochat$ is the
zero level set of $\cdisthat$.
Next, define an affine map $\collmap : \R^3 \rightarrow \R^3$ such that $\Ochat = \collmap(\Oc)$ by
\begin{equation}\label{eqn:colloid_map_function}
  \hat{\vx} = \collmap(\vx) = \collrot \vx + \colltrans,
\end{equation}
where $\collrot$ is a constant rotation matrix, and $\colltrans$ is a translation vector.  Thus, the distance function for $\Oc$ is given by:
\begin{equation}\label{eqn:true_colloid_dist_func}
  \cdist(\vx) = \cdisthat(\collmap(\vx)) = \cdisthat(\hat{\vx}),
\end{equation}
with derivative formula:
\begin{equation}\label{eqn:true_colloid_dist_func_derivative}
\begin{split}
  \nabla \cdist(\vx) &= \hat{\nabla} \cdisthat(\collmap(\vx)) \nabla \collmap(\vx) = \hat{\nabla} \cdisthat(\collmap(\vx)) \collrot.
\end{split}
\end{equation}

\subsubsection{Phase-field}

Define a one dimensional phase-field function:
\begin{equation}\label{eqn:1D_phase_field_ref}
  \phaseref(t) = \frac{1}{2} \left[\frac{2}{\pi} \arctan \left(-\frac{t}{\epsilon}\right) + 1\right],
\end{equation}
where $\phaseref : (-\infty, \infty) \rightarrow (0, 1)$.  The parameter $\epsilon > 0$ is the thickness of the transition.  The derivative is given by:
\begin{equation}\label{eqn:1D_phase_field_ref_deriv}
  \phaseref'(t) = - \frac{1}{\pi \epsilon} \frac{1}{1 + \left( \frac{t}{\epsilon} \right)^2}.
\end{equation}

The phase-field function associated with the colloidal sub-domain $\Oc$ is
\begin{equation}\label{eqn:phase_field_of_colloid}
  \phase(\vx) = \phaseref(\cdist(\vx)) = \phaseref(\cdisthat(\collmap(\vx))).
\end{equation}
Thus, $\phase$ is essentially $0$ inside the colloidal
  inclusion and $1$ outside; so $\phase \approx 1$ ``marks'' where the liquid
  crystal domain is. The gradient is given by
\begin{equation}\label{eqn:phase_field_of_colloid_gradient}
\begin{split}
  \nabla \phase(\vx) &= \phaseref'(\cdist(\vx)) \nabla \cdist(\vx) = \phaseref'(\cdisthat(\collmap(\vx))) \hat{\nabla} \cdisthat(\collmap(\vx)) \nabla \collmap(\vx) \\
  &= \phaseref'(\cdisthat(\collmap(\vx))) \hat{\nabla}
  \cdisthat(\collmap(\vx)) \collrot \
  = -\frac{1}{\pi \epsilon} \frac{1}{1 + \left( \frac{\cdisthat(\collmap(\vx))}{\epsilon} \right)^2} \hat{\nabla} \cdisthat(\collmap(\vx)) \collrot,
\end{split}
\end{equation}
and so we have
\begin{equation}\label{eqn:phase_field_of_colloid_gradient_sq}
\begin{split}
  |\nabla \phase(\vx)|^2 &= \left(\frac{1}{\pi \epsilon}\right)^2 \frac{1}{\left(1 + \left( \frac{\cdisthat(\collmap(\vx))}{\epsilon} \right)^2 \right)^2} |\hat{\nabla} \cdisthat (\collmap(\vx))|^2.
\end{split}
\end{equation}

We note the following relation between bulk and surface integrals.
Given $f\in C(\overline\Om)$, define
\begin{equation}\label{eqn:phase_field_perimeter}
  J_\epsilon(f) = C_0 \frac{\epsilon}{2} \iO f(\vx)|\nabla \phase(\vx)|^2 d\vx, \quad \text{where } C_0 := 4 \pi.
\end{equation}
Then one can show that (since $|\hat{\nabla} \cdisthat (\cdot)|=1$ near the zero level set of $\cdisthat(\cdot)$)
\begin{equation*}
  \lim_{\epsilon \rightarrow 0} J_\epsilon(f) = \int_{\partial \Oc} f(\vx) ds(\vx).
\end{equation*}
In particular, $\lim_{\epsilon \rightarrow 0} J_\epsilon(1) = |\partial \Oc|$
is the surface area of the boundary of the colloid $\Oc$.

\subsubsection{Weak anchoring}\label{sec:weak-anchoring}

We model boundary conditions on the colloid's surface by imposing
\emph{weak anchoring} \cite{Virga_book1994, deGennes_book1995} with $\Ea$, $\Ea^h$.
A standard, but somewhat ad-hoc, form for the energy in the $\vQ$-tensor model
\cite[eqn. (66)]{Mottram_arXiv2014} is
\begin{equation}\label{eqn:vQ-tensor_colloid_surface_energy}
\begin{split}
  J(\vQ) = \frac{\anchorcoef}{2} \int_{\partial \Oc} |\vQ - \vQ_0|^2,
\end{split}
\end{equation}
where $\vQ_0$ is the desired value of $\vQ$ on the boundary $\partial \Oc$
and $\anchorcoef$ is a large weighting parameter (penalty approach).
For our purposes, we will focus on imposing homeotropic anchoring, i.e. we take $\vQ_0$ to have the form of a uniaxial nematic:
\begin{equation*}
  \vQ_0 = s^* \left( \vnu \otimes \vnu  - \frac{1}{3} \vI \right),
\end{equation*}
where $\vnu$ is the normal of $\partial \Oc$
and $s^*$ is the global minimum of the double well potential.
Using the expression
$\vQ=s\big(\vn\otimes\vn - \frac13\vI\big)$ for uniaxial nematics,
along with the facts that $\vQ, \vQ_0$ are symmetric, $|\vQ -
\vQ_0|^2 = \tr \left[ (\vQ - \vQ_0)^2 \right]$, and $|\vn|=|\vnu|=1$,
a straightforward calculation gives
\begin{equation}\label{eqn:vQ_tensor_weak_energy_mottram}
  |\vQ - \vQ_0|^2 = 2 s s^* \left[|\vn|^2 |\vnu|^2 - (\vn \cdot \vnu)^2 \right] + \frac{2}{3} (s - s^*)^2 |\vnu|^2.
\end{equation}

We use \eqref{eqn:vQ_tensor_weak_energy_mottram} as motivation for our
weak anchoring energy in the context of the Ericksen model combined with
the immersed boundary method. In fact, noting that $\vnu = \nabla
\phase / |\nabla \phase|$, simplifying $s s^*$ in
\eqref{eqn:vQ_tensor_weak_energy_mottram} with $s^2$, and normalizing
the constants, we resort to \eqref{eqn:phase_field_perimeter} to
define the continuous weak anchoring energy as
\begin{equation}\label{eqn:colloid_weak_anchoring_energy}
\begin{split}
  \Ea[s,\vn] := \frac{\anchorcoef}{2} C_0 \epsilon \iO s^2 \left[ |\vn|^2 |\nabla \phase|^2 - (\nabla \phase \cdot \vn)^2 \right]
   + \frac{\anchorcoef}{2} C_0 \epsilon \iO |\nabla \phase|^2 (s - s^*)^2.
\end{split}
\end{equation}
Note that \eqref{eqn:colloid_weak_anchoring_energy} imposes normal anchoring
of the director field when minimized.
However, since \eqref{eqn:colloid_weak_anchoring_energy}
is invariant with respect to arbitrary changes in the sign of
$\vn$, we expect a different behavior of the director field
$\vn$ close to the colloid boundary $\partial\Oc$  from that in Section
\ref{sec:conforming_mesh}. This is confirmed by the numerical
experiments of Subsection \ref{sec:defects-weak-anchoring}.

The next task is to modify the energies $E_1 [s, \vn]$ and $E_2 [s]$
to account for the colloid, or equivalently for the phase variable
$\phi$. One possible choice is
\begin{equation}\label{energy-with-phi}
E_1[s, \vn]:=\int_\Om \phi \,\big(\kappa|\nabla s|^2+ s^2|\nabla\vn|^2 \big) dx,
\qquad
E_2[s]:=\int_\Om \phi \, \psi(s) dx,
\end{equation}
which has the disadvantage that the system is near singular in
$\Oc$ where $\phi\approx0$ and still requires values for $\vn$ and
$s$ inside $\Oc$. We thus prefer to take a extreme approach and
think of the colloid $\Oc$ as a rigid membrane filled with liquid crystal
material and subjected to the same weak anchoring condition as the exterior.
In the limit $\epsilon\to0$, the two systems inside and outside of
$\Oc$ decouple and we may simply consider the latter.
This suggests keeping the original forms for $E_1$ and $E_2$
in \eqref{energy}. Therefore, we use the continuous total energy
in \eqref{energy_with_colloid}.

We now discuss the discrete counterpart of \eqref{energy_with_colloid},
starting with $\Ea[s,\vn]$. We first
introduce the following discrete inner products:
\begin{equation}\label{eqn:discrete_inner_prod_anchor_alt}
\begin{split}
  \ipanchor^{\vn}(\vn_h, \vv_h ; s_h, \nabla \phase) & := \iO I_h \left\{ s_h^2 \left[ (\vn_h \cdot \vv_h) |\nabla \phase|^2 - (\nabla \phase \cdot \vn_h) (\nabla \phase \cdot \vv_h) \right] \right\}, \\
  \ipanchor^{s}(s_h, z_h ; \vn_h, \nabla \phase) & := \iO I_h \left\{ s_h z_h \left[ |\vn_h|^2 |\nabla \phase|^2 - (\nabla \phase \cdot \vn_h)^2 \right] \right\},
\end{split}
\end{equation}
where $I_h$ is the Lagrange interpolant. These expressions correspond
to using so-called mass lumping quadrature which, for all $f\in C^0(\overline{\Om)}$,
reads
\begin{equation}\label{mass-lumping}
  \int_\Om I_h f = \sum_{T\in\Tk_h} \int_T I_h f
  = \sum_{T\in\Tk_h} \frac{|T|}{d+1} \sum_{i=1}^{d+1} f(x_T^i),
\end{equation}
where $\{x_T^i\}_{i=1}^{d+1}$ are the vertices of $T$. This quadrature
rule is exact for piecewise linear polynomials and has the advantage
that the finite element realization of \eqref{eqn:discrete_inner_prod_anchor_alt}
is a \emph{diagonal} matrix.  The following result elaborates on this.
\begin{lem}[monotone property for lumped mass matrix]\label{lem:monotone_lumped_mass}
Let $m_h : \Uh \times \Uh \rightarrow \R$ be a bilinear form defined by
\begin{equation*}
  m_h(\vn_h, \vv_h) := \iO I_h \left[ \vn_h \cdot H(x) \vv_h \right] dx,
\end{equation*}
where $H$ is a continuous $d \times d$ symmetric positive semi-definite matrix.  If $|\vn_h(x_i)| \geq 1$ at all nodes $x_i$ in $\Nk_h$, then
\begin{equation*}
  m_h(\vn_h, \vn_h) \geq m_h \left( \frac{\vn_h}{|\vn_h|}, \frac{\vn_h}{|\vn_h|} \right).
\end{equation*}
\end{lem}
\begin{proof}
In view of \eqref{mass-lumping}, we rewrite $m_h(\vn_h, \vv_h)$ as
\begin{equation*}
  m_h(\vn_h, \vv_h) = \sum_{T \in \Tk_h} \frac{|T|}{d+1} \sum_{i = 1}^{d+1} \left[ \vn_h(x_T^i) \cdot H(x_T^i) \vv_h(x_T^i) \right].
\end{equation*}
Then, clearly
\begin{equation*}
\begin{split}
  m_h(\vn_h, \vn_h) &= \sum_{T \in \Tk_h} \frac{|T|}{d+1} \sum_{i = 1}^{d+1} |\vn_h(x_T^i)|^2 \left[ \frac{\vn(x_T^i)}{|\vn_h(x_T^i)|} \cdot H(x_T^i) \frac{\vn_h(x_T^i)}{|\vn_h(x_T^i)|} \right] \\
  &\geq \sum_{T \in \Tk_h} \frac{|T|}{d+1} \sum_{i = 1}^{d+1} \left[ \frac{\vn_h(x_T^i)}{|\vn_h(x_T^i)|} \cdot H(x_T^i) \frac{\vn_h(x_T^i)}{|\vn_h(x_T^i)|} \right] = m_h \left( \frac{\vn_h}{|\vn_h|}, \frac{\vn_h}{|\vn_h|} \right),
\end{split}
\end{equation*}
which concludes the proof.
\end{proof}

To apply Lemma \ref{lem:monotone_lumped_mass} to the first bilinear
form in \eqref{eqn:discrete_inner_prod_anchor_alt} we observe that $H$ reads
\[
H = s_h^2 \Big(|\nabla\phi|^2 \vI - \nabla\phi\otimes\nabla\phi \Big),
\]
and is symmetric positive semi-definite, whence
\begin{equation}\label{monotone_anchoring}
  \ipanchor^{\vn}(\vn_h, \vn_h ; s_h, \nabla \phase) \geq \ipanchor^{\vn} \left( \frac{\vn_h}{|\vn_h|},  \frac{\vn_h}{|\vn_h|} ; s_h, \nabla \phase \right).
\end{equation}

Thus, we take the discrete weak anchoring energy to be
\begin{equation}\label{eqn:colloid_weak_anchoring_energy_discrete}
\begin{split}
  \Ea^h[s_h,\vn_h] := \frac{\anchorcoef}{2} C_0 \epsilon \left( \ipanchor^{\vn}(\vn_h, \vn_h ; s_h, \nabla \phase) + \iO |I_h \nabla \phase|^2 (s_h - s^*)^2 \right);
\end{split}
\end{equation}
note that $\ipanchor^{\vn}(\vn_h, \vn_h ; s_h, \nabla \phase) =
\ipanchor^{s}(s_h, s_h ; \vn_h, \nabla \phase)$.
The discrete total energy is then given by \eqref{discrete_energy_with_colloid}.

\subsubsection{Penalizing Dirichlet Conditions}\label{sec:penalize_dirichlet}

The weak anchoring energy \eqref{eqn:colloid_weak_anchoring_energy} is
insensitive to the orientation of the director field $\vn$, a drawback
of this approach for the Ericksen model. To impose a general
Dirichlet boundary condition $(g,\vr)$ on the colloid's surface,
in a more consistent manner, we consider a different penalization.

Consider the following penalization energy:
\[
J[s,\vu] = \frac{\anchorcoef}{2} \int_{\partial\Oc} |\vu-\vr|^2 + |s-g|^2,
\]
where $\anchorcoef$ is a large penalty parameter. We proceed as in
Section \ref{sec:weak-anchoring}, that is we first manipulate this formula to
get one with suitable monotonicity properties. Write $\vu = s \vn$ and
$\vr = g \vnu$, replace $g$ by $s$ because $s\approx g$, and expand the
first square using that $|\vn| = |\vnu| = 1$ to get
$|\vu-\vr|^2 = s^2 |\vn-\vnu|^2$.
We next express the line energy $J[s,\vu]$ as a bulk energy within the
immerse boundary method. Recall that $\vnu=\nabla\phi/|\nabla\phi|$
and make use of \eqref{eqn:phase_field_perimeter} to define
\begin{equation}\label{penalty}
\Ea[s,\vn] := \frac{\anchorcoef}{2} C_0\epsilon \int_\Omega
|\nabla\phi|^2 \left\{ s^2 \left| \vn - \frac{\nabla \phi}{|\nabla \phi|} \right|^2 + |s-g|^2 \right\}.
\end{equation}

The discrete form of this penalized Dirichlet energy is given by
\begin{equation}\label{discrete-dirichlet-penalty}
  \Ea^h[s_h,\vn_h] := \frac{\anchorcoef}{2} C_0 \epsilon \left( \widetilde{\ipanchor}^{s}(s_h, s_h ; \vn_h, \nabla \phase)
  + \int_{\Omega} | I_h \nabla \phase|^2 (s_h - g_h)^2 \right),
\end{equation}
where
\begin{equation}\label{dirichlet-penalty-discrete-forms}
\begin{split}
\widetilde{\ipanchor}^{\vn}(\vn_h, \vv_h ; s_h, \nabla \phase) &:=
   \int_\Omega I_h \left\{ s_h^2 |\nabla \phase|^2 \vn_h \cdot \vv_h \right\}, \\
\linanchor(\vv_h; s_h, \nabla \phase) &:=
    \int_\Omega I_h \left\{ s_h^2 |\nabla \phase| \nabla \phase \cdot \vv_h \right\}, \\
\widetilde{\ipanchor}^{s}(s_h, z_h ; \vn_h, \nabla \phase) &:=
   \int_\Omega I_h \left\{ s_h z_h \big| |\nabla \phase| \vn_h - \nabla \phase \big|^2 \right\}.
\end{split}
\end{equation}
Note that
\begin{equation}\label{dirichlet-penalty-discrete-form-equivalent}
\begin{split}
\widetilde{\ipanchor}^{s}(s_h, s_h ; \vn_h, \nabla \phase) = \widetilde{\ipanchor}^{\vn}(\vn_h, \vn_h ; s_h, \nabla \phase)
- 2 \linanchor(\vn_h; s_h, \nabla \phase) + \int_{\Omega} I_h \left\{ s_h^2 |\nabla \phase|^2 \right\},
\end{split}
\end{equation}
and that $\widetilde{\ipanchor}^{\vn}$ and $\linanchor(\vv_h; s_h, \nabla \phase)$
are also useful in computing variational derivatives of \eqref{discrete-dirichlet-penalty}.
Recall that the total discrete energy is given by \eqref{discrete_energy_with_colloid}.
The presence of the Lagrange interpolation operator $I_h$ in
\eqref{discrete-dirichlet-penalty}, \eqref{dirichlet-penalty-discrete-forms}
is needed to ensure that Step 2 (projection) of the Algorithm in
Subsection \ref{algorithm} decreases the energy
(see \eqref{monotone-dirichlet-penalty} below).

\begin{lem}[monotone property for penalized Dirichlet energy]\label{lem:monotone_penal_dirichlet}
Let $d \geq 2$ and let $\vn$, $\vm$ be arbitrary vectors in $\mathbb{S}^{d-1}$. If
$\vt \in \R^d$ such that $|\vn+\vt|\ge 1$, then
\begin{equation*}
  | (\vn + \vt) - \vm | \geq \left| \frac{\vn + \vt}{|\vn + \vt|} - \vm \right|.
\end{equation*}
\end{lem}
\begin{proof}
Let %
$\vn_0:= \frac{\vn+\vt}{|\vn+\vt|}$ and note that
\[
|(\vn+\vt) - \vm|^2 = |\vn+\vt|^2 + 1- 2 \, \vn_0\cdot\vm \, |\vn+\vt|,
\qquad
|\vn_0 - \vm|^2 = 2 - 2 \, \vn_0\cdot\vm.
\]
This implies
\begin{align*}
  |(\vn+\vt) - \vm|^2 - |\vn_0 - \vm|^2 &= |\vn+\vt|^2 - 1
  -2 \, \vn_0\cdot\vm \big(|\vn+\vt| - 1 \big)
  \\
  & = \big(|\vn+\vt| - 1\big) \big(|\vn+\vt| + 1 - 2 \vn_0\cdot\vm \big)
  \ge 0,
\end{align*}
because $|\vn+\vt|\ge 1$. This is the asserted estimate.
\end{proof}

We now apply Lemma \ref{lem:monotone_penal_dirichlet} to $\widetilde{\ipanchor}^{s}$
in \eqref{discrete-dirichlet-penalty}, by setting $\vm = \nabla \phase / |\nabla \phase|$,
and obtain the monotonicity property
\begin{equation}\label{monotone-dirichlet-penalty}
\begin{split}
    \widetilde{\ipanchor}^{s}(s_h, s_h ; \vn_h + \vt_h, \nabla \phase) &= \int_\Omega I_h \left\{ |\nabla \phase|^2 s_h^2 \left| (\vn_h + \vt_h) - \frac{\nabla \phase}{|\nabla \phase|} \right|^2 \right\} \\
    &\ge \widetilde{\ipanchor}^{s} \left( s_h, s_h ; \frac{\vn_h + \vt_h}{|\vn_h + \vt_h|}, \nabla \phase \right),
\end{split}
\end{equation}
where it is assumed that $\vn_h (x_i) \cdot \vt_h (x_i) = 0$ at all nodes $x_i$.

\subsubsection{Minimization scheme}

We apply the Algorithm in Section \ref{algorithm} to the total energy \eqref{discrete_energy_with_colloid}
in the case of either weak anchoring energy \eqref{eqn:colloid_weak_anchoring_energy_discrete}
or penalized Dirichlet energy \eqref{discrete-dirichlet-penalty}.
To this end, we need the following additional variational derivatives of $\Ea^h$.
The first order variation of $\Ea^h[s_h,\vn_h]$ in the direction $\vv_h \in \Yh (\vn_h^k) \cap \Hbdy{\bdyvn}$ at the director variable $\vn_h^k$ reads
\begin{equation}\label{first_order_discrete_variation_dvN_weak_anchor}
\begin{split}
 \delta_{\vn_h} \Ea^h [s_h^k, \vn_h^k ; \vv_h] &= \anchorcoef C_0\epsilon
 \, \ipanchor^{\vn}(\vn_h^k, \vv_h ; s_h^k, \nabla \phase),
\end{split}
\end{equation}
for the energy \eqref{eqn:colloid_weak_anchoring_energy_discrete}, whereas
the expression reads
\begin{equation}\label{first_order_discrete_variation_dvN_penalty}
\begin{split}
 \delta_{\vn_h} \Ea^h [s_h^k, \vn_h^k ; \vv_h] &= \anchorcoef C_0 \epsilon \left\{
 \widetilde{\ipanchor}^{\vn}(\vn_h^k, \vv_h ; s_h^k, \nabla \phase)
  - \linanchor(\vv_h; s_h^k, \nabla \phase) \right\},
\end{split}
\end{equation}
for the energy \eqref{discrete-dirichlet-penalty}.

The first order variation of $\Ea^h[s_h,\vn_h]$ in the direction $z_h \in \Sh \cap \Hbdy{\bdys}$ at the degree of orientation variable $s_h^k$ is
\begin{equation}\label{first_order_discrete_variation_dS_weak_anchor}
\begin{split}
 \delta_{s_h} \Ea^h [s_h^k ,\vn_h^{k} ; z_h] &= \anchorcoef C_0 \epsilon \left\{ \ipanchor^{s}(s_h^k, z_h ; \vn_h^k, \nabla \phase) + \iO |I_h \nabla \phase|^2 (s_h^k - s^*) z_h \right\}
\end{split}
\end{equation}
for the energy \eqref{eqn:colloid_weak_anchoring_energy_discrete}, whereas
the expression reads
\begin{equation}\label{first_order_discrete_variation_dS_penalty}
\begin{split}
 \delta_{s_h} \Ea^h [s_h^k ,\vn_h^{k} ; z_h] &= \anchorcoef C_0 \epsilon \left\{
 \widetilde{\ipanchor}^{s}(s_h^k, z_h ; \vn_h^k, \nabla \phase) + \int_{\Omega} | I_h \nabla \phase|^2 (s_h^k - g_h) z_h \right\}
\end{split}
\end{equation}
for the energy \eqref{discrete-dirichlet-penalty}.
Note that \eqref{monotone_anchoring} and \eqref{monotone-dirichlet-penalty} guarantee that the
projection step in our algorithm reduces the energy $\Ea^h[s_h,\vn_h]$,
whence Theorem \ref{energydecreasing} still holds in this context (see Theorem \ref{thm:energydecreasing_general}).

\subsubsection{Simulating defects with weak anchoring}\label{sec:defects-weak-anchoring}

The computational domain is a unit cube $\Om = [0, 1]^3$.  The colloid is represented by a sphere of radius $0.25$ centered at $(0.5,0.5,0.5)\tp$. We consider the boundary conditions shown in Figure \ref{fig:LC_Sphere_Inclusion_BCs_A}.  The strong anchoring condition on $\dOm$ is given by
\begin{equation}\label{eqn:saturn_washer_BCs}
  \vn = (0,0,1)\tp, ~ \text{on } \dOm, \quad
  s = s^*, ~ \text{on } \dOm,
\end{equation}
whereas $\Ea$ given in \eqref{eqn:colloid_weak_anchoring_energy_discrete}
models weak anchoring on the colloid's surface with parameters
\[
\anchorcoef = 300.0,  \qquad \epsilon = 0.06.
\]
We use the same double well potential as before.  The initial conditions in $\Om$ for the gradient flow are: $s = s^*$ and $\vn = (0,0,1)\tp$.

The equilibrium solution, with $\kappa = 1.0$,
is shown in Figure \ref{fig:Saturn_Washer_Phase_Field_all}.
The defect region is significantly different than that shown in Section \ref{sec:conform_disperse-point_defect}.
This is due to the fact that the weak anchoring energy \eqref{eqn:colloid_weak_anchoring_energy}
is invariant with respect to arbitrary changes in the sign of $\vn$.
\begin{figure}
\begin{center}
\subfloat{

\includegraphics[width=2.5in]{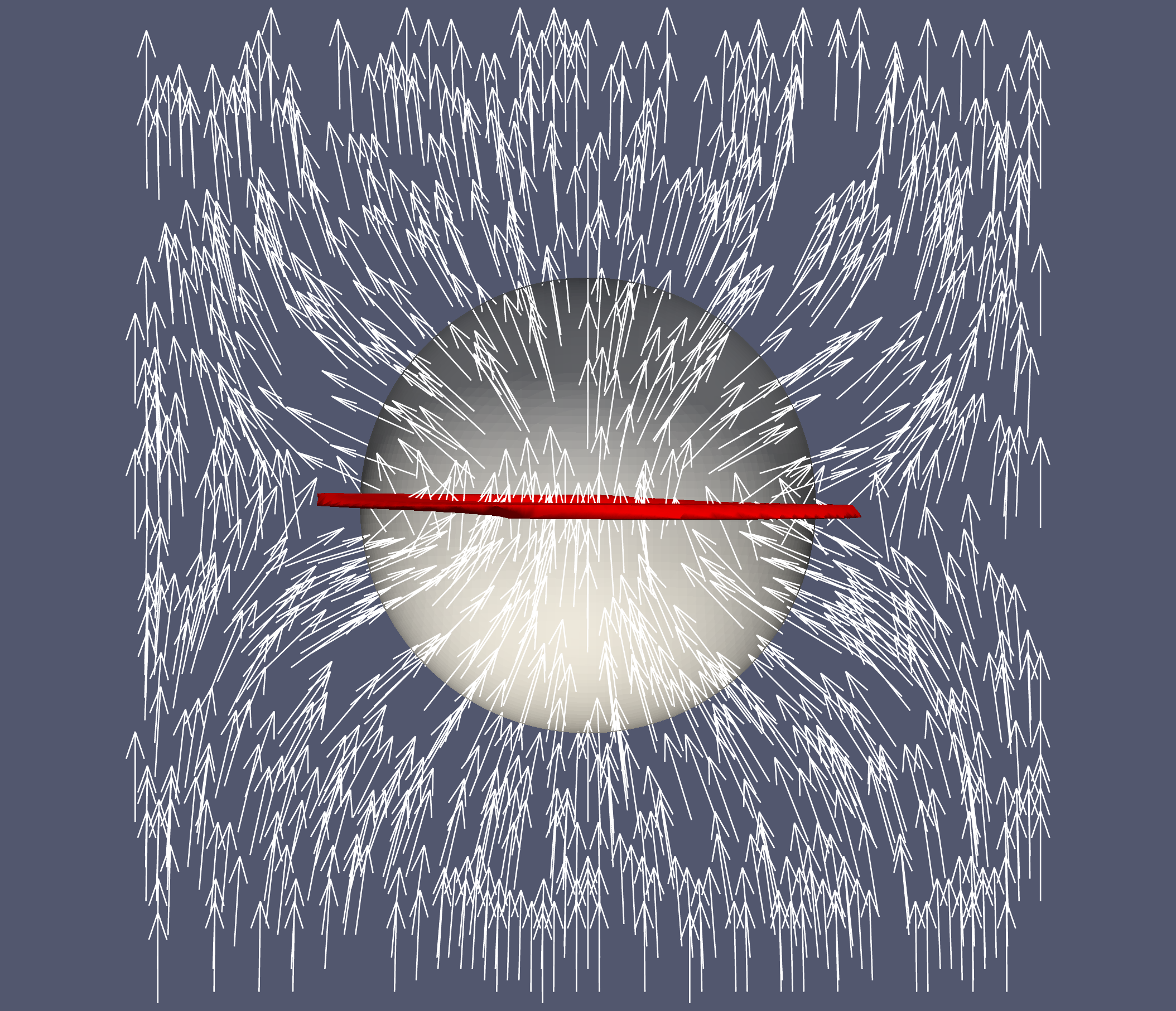}
}
\hspace{0.1in}
\subfloat{

\includegraphics[width=2.5in]{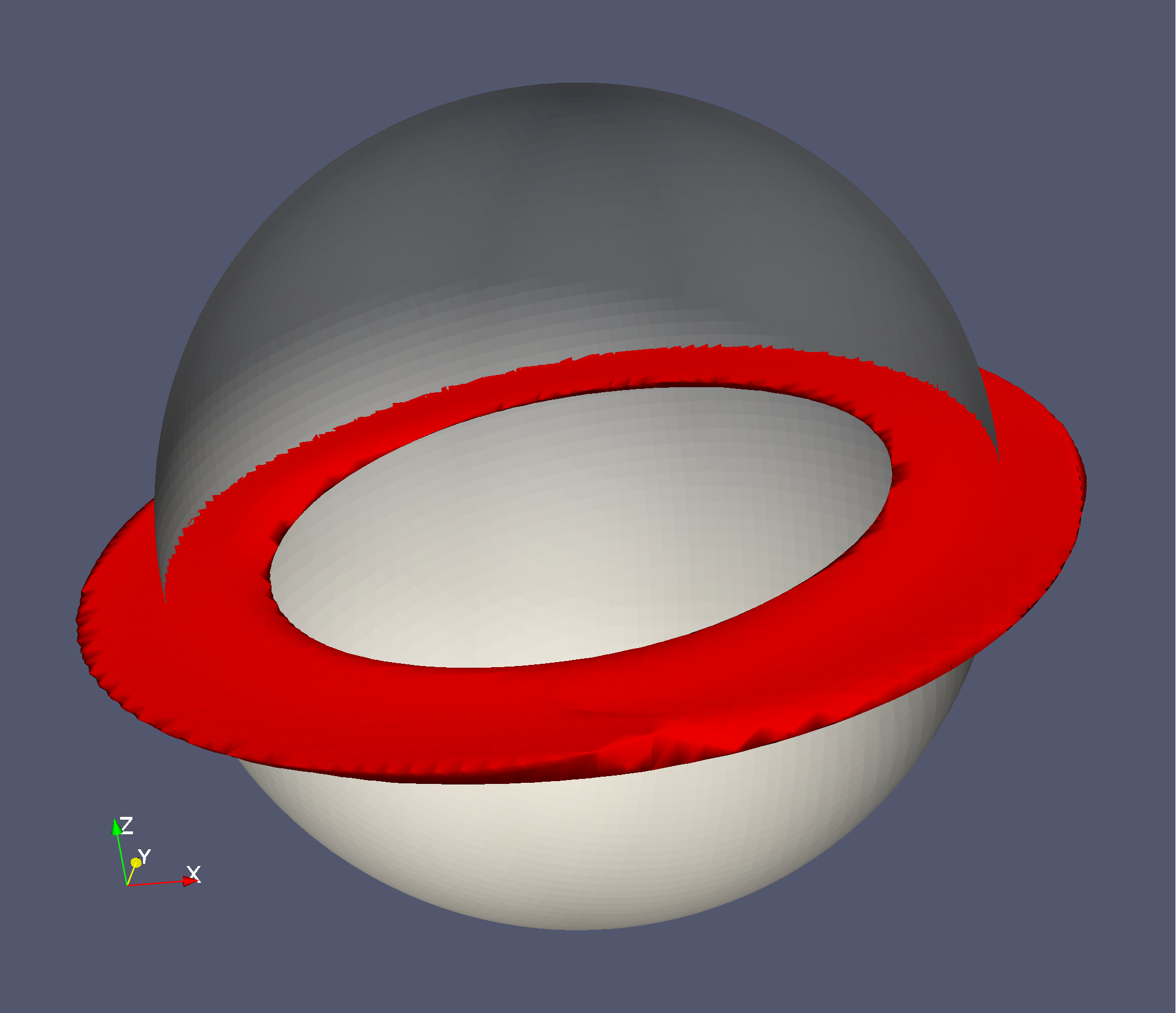}
}
\caption{Simulation results for boundary conditions in \eqref{eqn:saturn_washer_BCs} ($\kappa = 1.0$) using the immersed boundary approach and anchoring energy \eqref{eqn:colloid_weak_anchoring_energy_discrete}.  The surface mesh of the colloid boundary is shown and the $s=0.14$ iso-surface is plotted in red which indicates the defect region.  The director field is also shown (white arrows) on a plane parallel to the $y$-$z$ plane and passing through $x = 0.5$.  The defect region is spread out into a ``washer'' region (c.f. Section \ref{sec:conform_disperse-point_defect}).}
\label{fig:Saturn_Washer_Phase_Field_all}
\end{center}
\end{figure}

Next, we change the boundary conditions as we did in Section \ref{sec:conform_ring_defect}, i.e. the strong anchoring condition on $\dOm$ is given by
\begin{equation}\label{eqn:saturn_ring_immersed_bdy_BCs}
\begin{split}
  \vn \text{ smoothly interpolates between } (0,0,-1)\tp \text{ and }
  (0,0,1)\tp, ~ \text{and } s = s^*, ~ \text{on } \dOm,
\end{split}
\end{equation}
whereas $\Ea$ models weak anchoring on the colloid's surface.  The initial conditions in $\Om$ for the gradient flow are: $s = s^*$ and
\begin{equation*}
  \vn(x,y,z) =
  \begin{cases}
    (0,0,-1)\tp, ~ \text{ if } z < 0, \\
    (0,0,+1)\tp, ~ \text{ if } z \geq 0.
  \end{cases}
\end{equation*}

The equilibrium solution, with $\kappa = 1.0$,
is shown in Figure \ref{fig:Saturn_Ring_Phase_Field_all}.
Again, the choice of boundary conditions essentially induces the Saturn ring defect.
The radius of the Saturn ring is $\approx 0.38$.  Also, note that the structure of
the director field is not the same as would be obtained with the
Landau-deGennes model \cite{Alama_PRE2016}.
\begin{figure}
\begin{center}
\subfloat{

\includegraphics[width=2.5in]{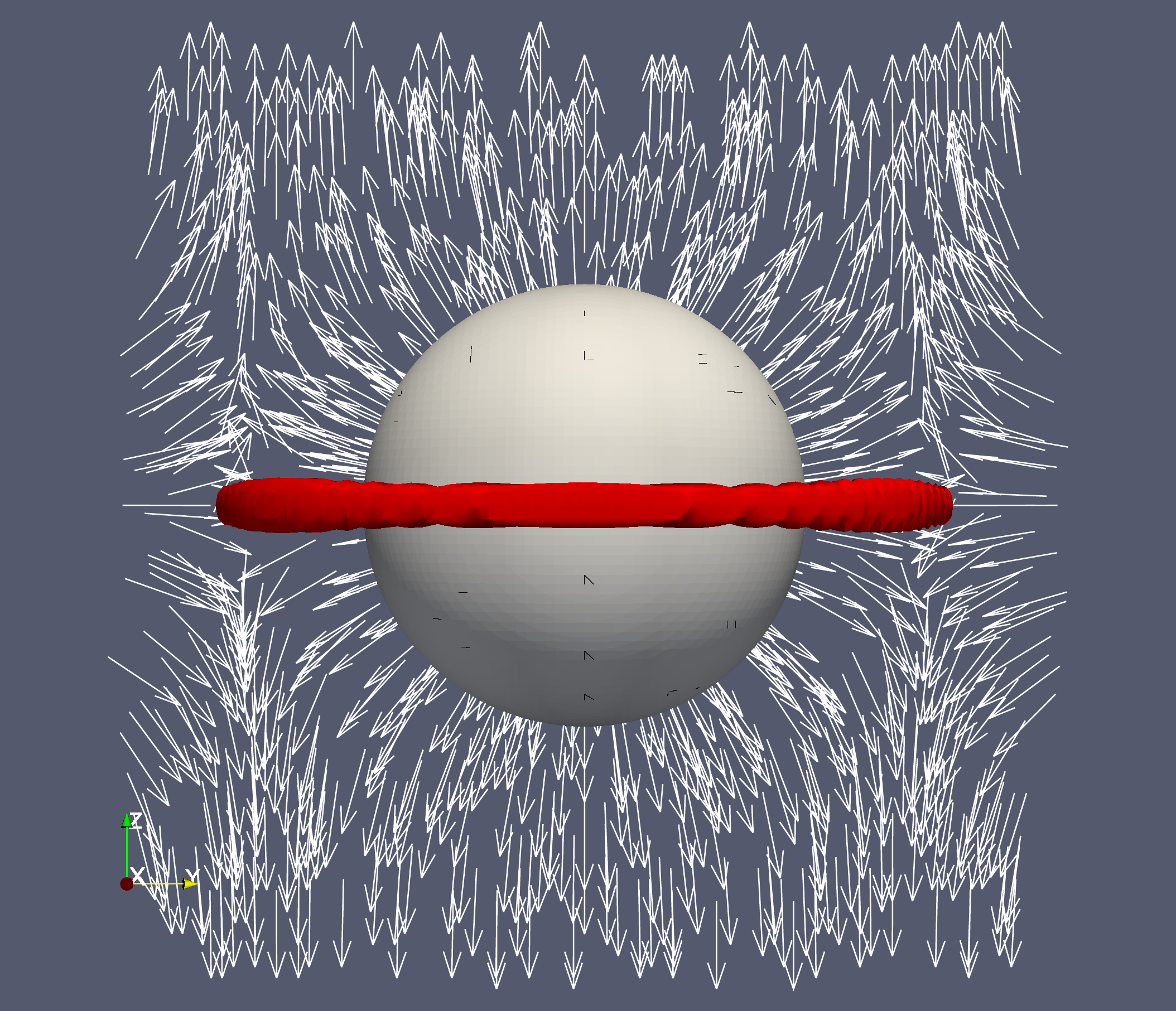}
}
\hspace{0.1in}
\subfloat{

\includegraphics[width=2.5in]{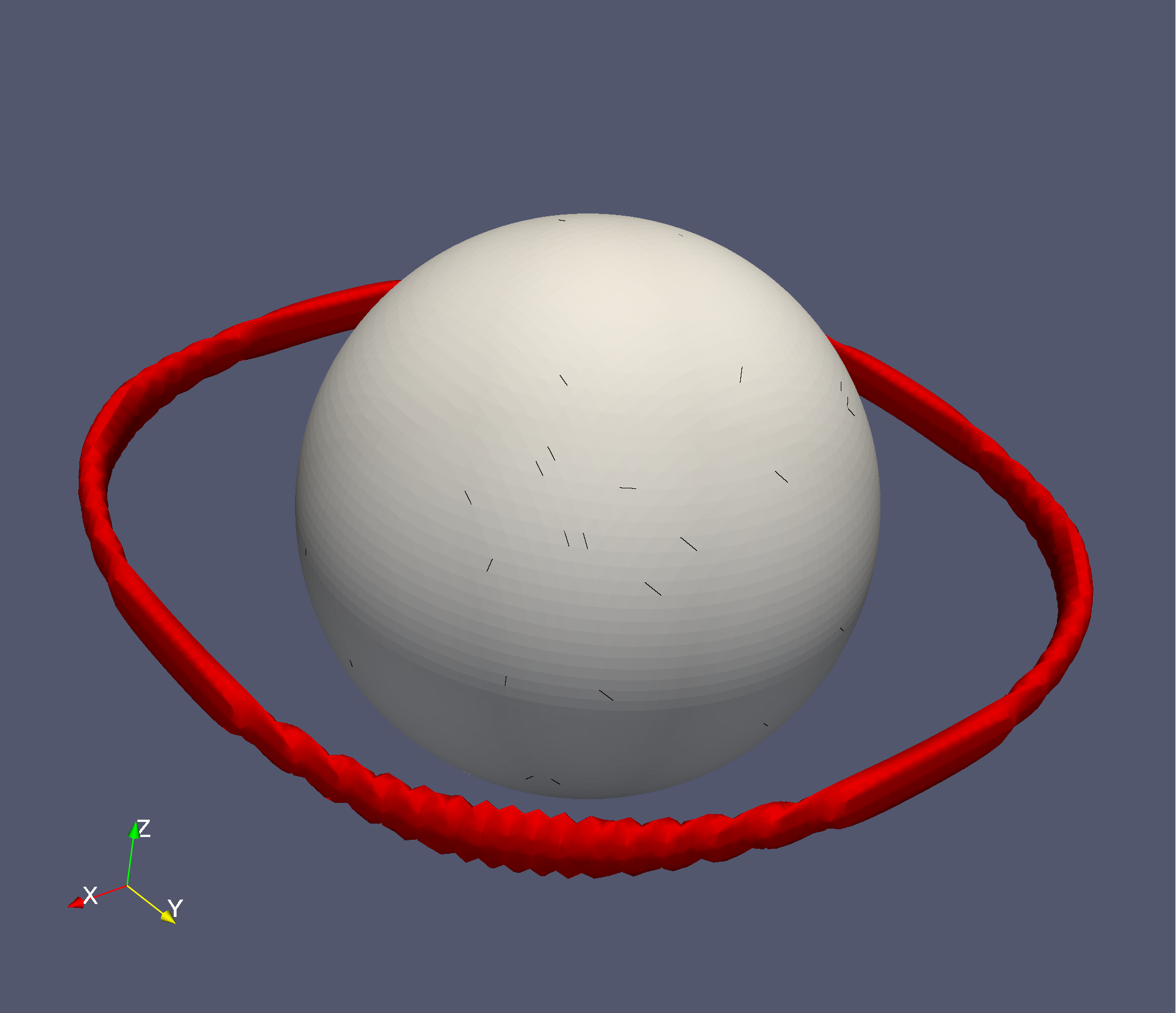}
}
\caption{Simulation results for boundary conditions in \eqref{eqn:saturn_ring_immersed_bdy_BCs} ($\kappa = 1.0$) using the immersed boundary approach and anchoring energy \eqref{eqn:colloid_weak_anchoring_energy_discrete}.  The surface mesh of the colloid boundary is shown and the $s=0.06$ iso-surface is plotted in red which indicates the defect region.  The director field is also shown (white arrows) on a plane parallel to the $y$-$z$ plane and passing through $x = 0.5$.  The defect region mimics the classic Saturn ring defect (c.f. Section \ref{sec:conform_ring_defect}).}
\label{fig:Saturn_Ring_Phase_Field_all}
\end{center}
\end{figure}

\subsubsection{Simulating defects with penalized Dirichlet conditions}\label{sec:defects-penalized-dirichlet}

We adopt the same computational conditions here, except that $\Ea$ is given by \eqref{discrete-dirichlet-penalty}
which models a Dirichlet condition (penalized) on the colloid's surface.
Everything else is the same as before, including parameter values.

Using the first set of boundary conditions \eqref{eqn:saturn_washer_BCs},
the equilibrium solution ($\kappa = 1.0$)
is shown in Figure \ref{fig:Hole_3D_Defect_Phase_Field_Dirichlet_all}.
The defect region is essentially the same as in Section \ref{sec:conform_disperse-point_defect} (see Figure \ref{fig:Prism_Sphere_Hole_kappa_1_all}).
In other words, the penalized Dirichlet condition is \emph{not} invariant with respect to arbitrary changes in the sign of $\vn$
\begin{figure}
\begin{center}
\subfloat{

\includegraphics[width=2.5in]{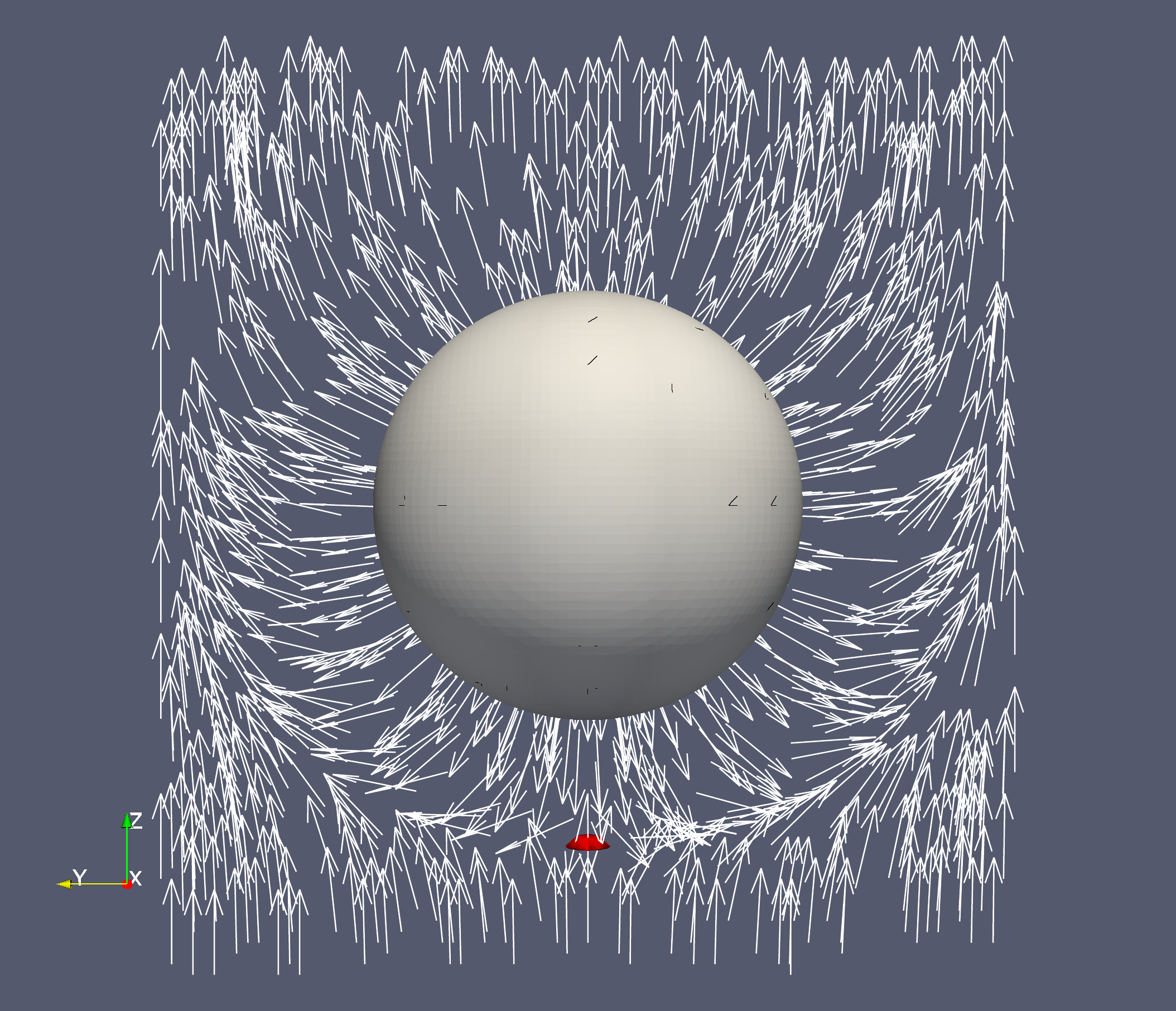}
}
\hspace{0.1in}
\subfloat{

\includegraphics[width=2.5in]{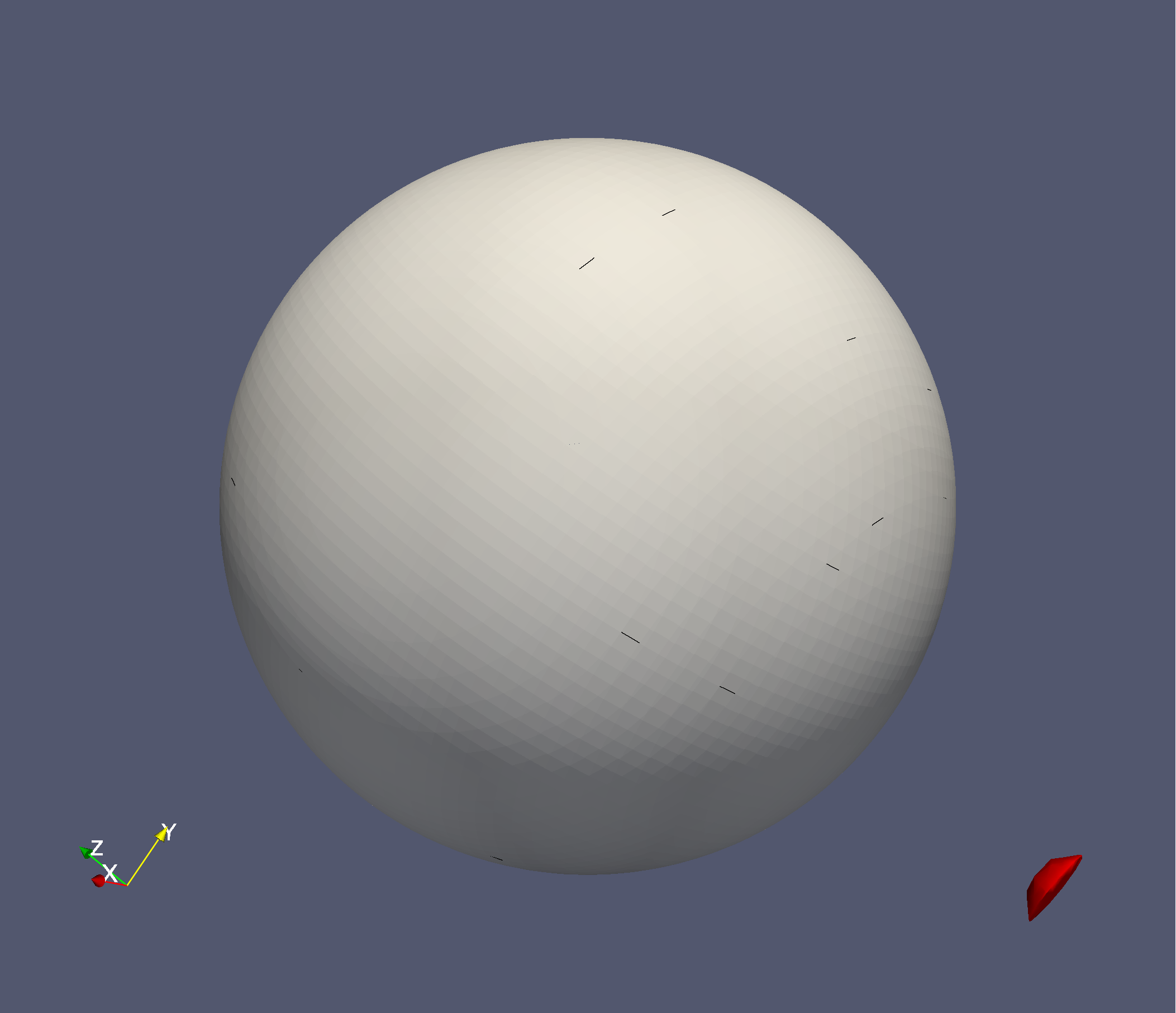}
}
\caption{Simulation results for boundary conditions in \eqref{eqn:saturn_washer_BCs} ($\kappa = 1.0$) using the immersed boundary approach and anchoring energy \eqref{discrete-dirichlet-penalty}.  The surface mesh of the colloid boundary is shown and the $s=0.06$ iso-surface is plotted in red which indicates the defect region.  The director field is also shown (white arrows) on a plane parallel to the $y$-$z$ plane and passing through $x = 0.5$.  The defect region is ``point-like'' at height $z=0.11$ (c.f. Section \ref{sec:conform_disperse-point_defect}).}
\label{fig:Hole_3D_Defect_Phase_Field_Dirichlet_all}
\end{center}
\end{figure}

Next, we change the outer boundary conditions as we did in Section \ref{sec:conform_ring_defect}.
Using the second set of boundary conditions \eqref{eqn:saturn_ring_immersed_bdy_BCs},
the equilibrium solution ($\kappa = 1.0$)
is shown in Figure \ref{fig:Saturn_Ring_Phase_Field_Dirichlet_all}.

The choice of boundary conditions induces the Saturn ring defect (similar to Figure \ref{fig:Prism_Sphere_Hole_BC2_kappa_1_all}).
The radius of the Saturn ring is $\approx 0.405$.  Also, note that the structure of
the director field is not the same as would be obtained with the
Landau-deGennes model \cite{Alama_PRE2016}.
\begin{figure}
\begin{center}
\subfloat{

\includegraphics[width=2.5in]{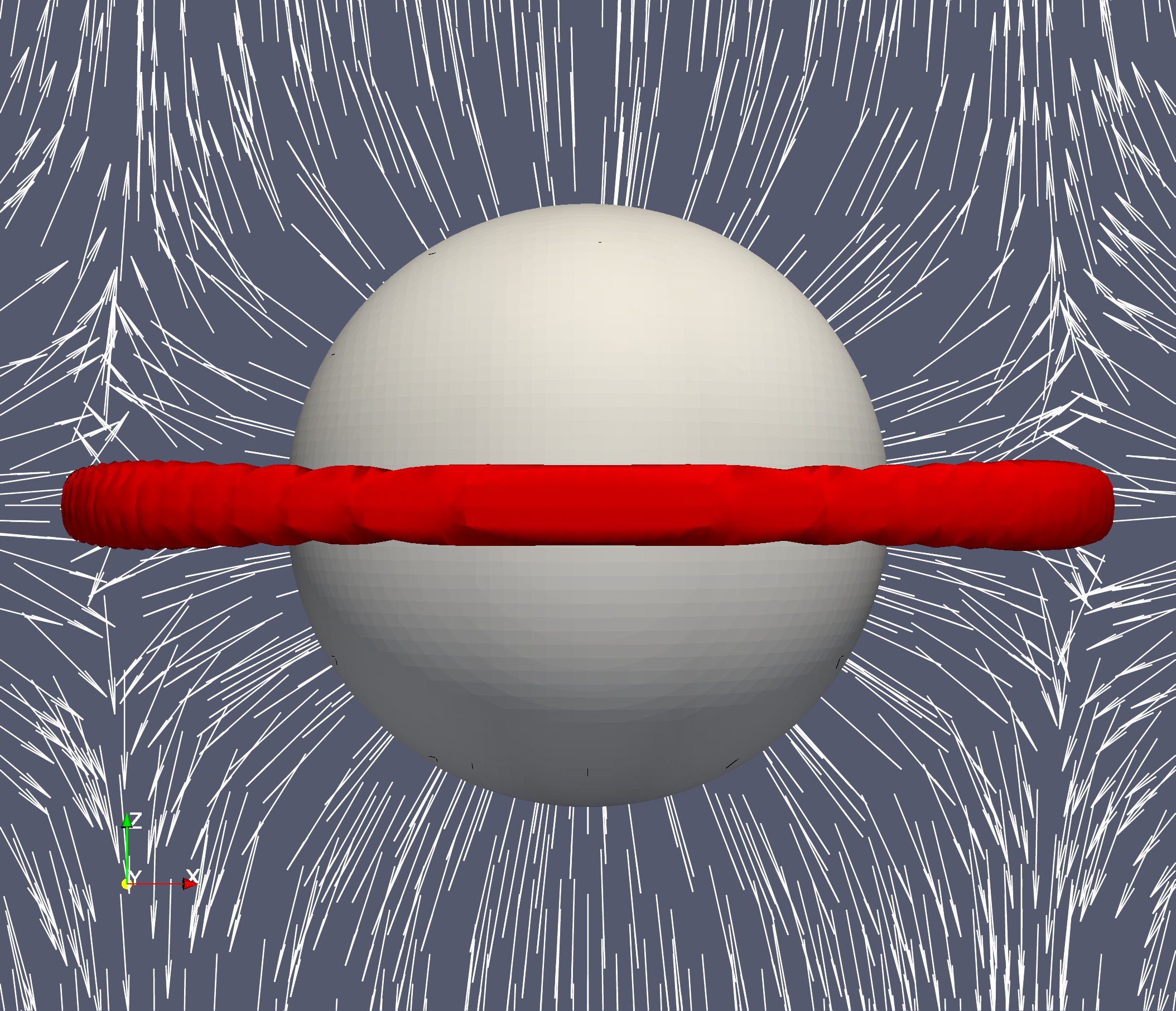}
}
\hspace{0.1in}
\subfloat{

\includegraphics[width=2.5in]{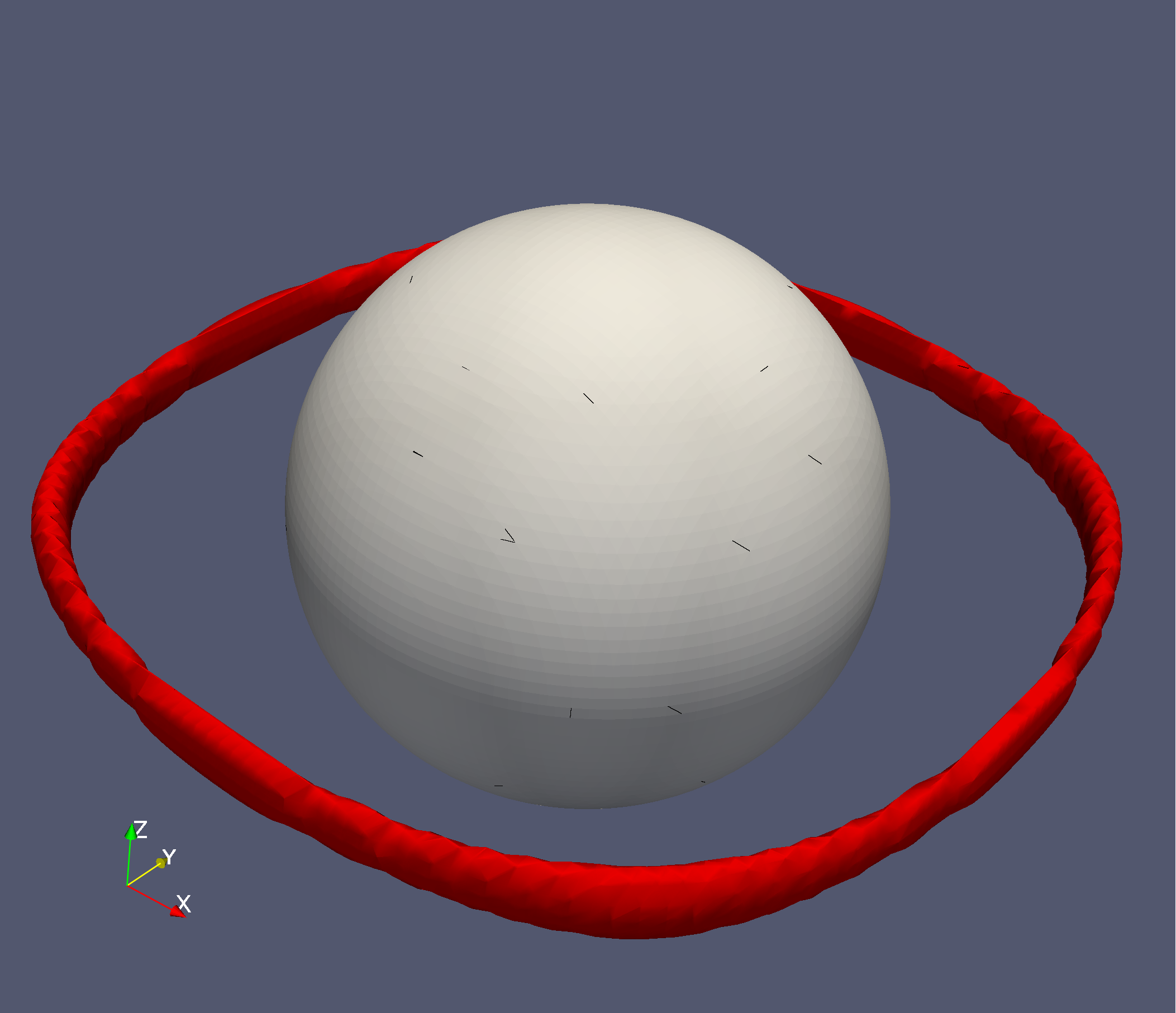}
}
\caption{Simulation results for boundary conditions in \eqref{eqn:saturn_ring_immersed_bdy_BCs} ($\kappa = 1.0$) using the immersed boundary approach and anchoring energy \eqref{discrete-dirichlet-penalty}.  The surface mesh of the colloid boundary is shown and the $s=0.08$ iso-surface is plotted in red which indicates the defect region.  The director field is also shown (white arrows) on a plane parallel to the $y$-$z$ plane and passing through $x = 0.5$.  The defect region mimics the classic Saturn ring defect (c.f. Section \ref{sec:conform_ring_defect}).}
\label{fig:Saturn_Ring_Phase_Field_Dirichlet_all}
\end{center}
\end{figure}

\section{Electric field}\label{sec:elec_field}

External field effects, such as an electric field, can be modeled by adding another term to the energy.  The following sections describe this as well as prove the monotone energy decreasing property of our algorithm applied to the modified energy.

\subsection{Modified energy}\label{sec:elec_field_energy}

The energy now takes the form
\begin{equation}\label{eqn:energy_with_external_field}
  E[s,\vn] = E_1[s,\vn] + E_2[s] + \Ea[s,\vn] + \Eext[s,\vn],
\end{equation}
where $\Eext[s,\vn]$ is the external field energy.  Following
\cite{Biscari_CMT2007, deGennes_book1995}, we
let $\Eext[s,\vn]$ be defined by
\begin{equation}\label{eqn:electric_energy}
  \Eext[s,\vn] := -\frac{\Ecoef}{2} \left( \ebar \iO (1 - s \ga) |\vE|^2 + \ea \iO s (\vE \cdot \vn)^2 \right),
\end{equation}
where $\vE$ is the \emph{given} (fixed) electric field.  The other
constants are related to the material properties of the liquid crystal
medium. Define $\varepsilon_{\parallel}$, $\varepsilon_{\perp}$ to be
the dielectric permittivities in the directions parallel and
orthogonal to a liquid crystal molecule.  Define $\ebar =
(\varepsilon_{\parallel} + 2 \varepsilon_{\perp})/3$ to be the average
dielectric permittivity (the $2$ is for the two directions orthogonal
to the director), $\ea = \varepsilon_{\parallel} -
\varepsilon_{\perp}$ the dielectric anisotropy, and $\ga = \ea/(3
\ebar)$ a dimensionless ratio.  We allow for $\ea$ to be positive or
negative and note that $0 \leq \ga \leq 1$ when $0 \leq \varepsilon_{\perp} \leq \varepsilon_{\parallel}$.

Note that the sign of the second integral in \eqref{eqn:electric_energy} can be negative (however it is bounded because $s$, $\vE$, and $\vn$ are bounded).  Thus, in order to preserve our energy decreasing minimization scheme (Section \ref{algorithm}), we first introduce a discrete quantity analogous to \eqref{eqn:discrete_inner_prod_anchor_alt}:
\begin{equation}\label{eqn:discrete_inner_prod_elec}
  \ipelec(s_h, \vn_h, \vv_h) = \iO I_h \left[ |\ea| |\vE|^2 (\vn_h \cdot \vv_h) - \ea s_h (\vE \cdot \vn_h) (\vE \cdot \vv_h) \right].
\end{equation}
To apply Lemma \ref{lem:monotone_lumped_mass} we see that
the matrix $H$ reads
\[
H = |\epsilon_a| |\vE|^2 \vI - \epsilon_a s_h \vE\otimes\vE,
\]
which is symmetric and positive semi-definite because $|s_h|\le
1$. Consequently
\begin{equation}\label{monotone_electric}
  \ipelec(s_h, \vn_h, \vn_h) \geq \ipelec \left( s_h, \frac{\vn_h}{|\vn_h|}, \frac{\vn_h}{|\vn_h|} \right).
\end{equation}

We now define the discrete counterpart of
\eqref{eqn:energy_with_external_field} to be
\begin{equation}\label{eqn:discrete energy_with_Eext}
  E^h [s_h, \vn_h] := E^h_1 [s_h, \vn_h ] + E^h_2 [s_h] + \Ea^h[s_h,\vn_h] + \Eext^h [s_h, \vn_h ],
\end{equation}
where the discrete electric energy is similar to \eqref{eqn:electric_energy}
and is given by
\begin{equation}\label{eqn:electric_energy_discrete}
  \Eext^h[s_h,\vn_h] = \frac{\Ecoef}{2} \left( -\ebar \iO (1 - s_h \ga) |\vE|^2 + \ipelec(s_h, \vn_h, \vn_h) - |\ea| \iO |\vE|^2 \right).
\end{equation}
Observe that \eqref{eqn:electric_energy_discrete} is an approximation of
\begin{equation}\label{eqn:electric_energy_discrete_approx}
  \Eext^h[s_h,\vn_h] = \frac{\Ecoef}{2} \left( -\ebar \iO (1 - s_h \ga) |\vE|^2 - \ea \iO s_h (\vE \cdot \vn_h)^2 + |\ea| \iO |\vE|^2 (|\vn_h|^2 - 1) \right),
\end{equation}
where the ``extra'' term is non-positive and \emph{consistent}
(i.e. it vanishes as $h \rightarrow 0$ provided the singular set
$\mathbb{S}$ has zero Lebesgue measure).  Moreover, $\iO |\vE|^2
|\vn|^2$ is constant at the continuous level, whence the extra term does not fundamentally change the energy.  However, it is needed to ensure the projection step in the algorithm decreases the (discrete) energy, which is guaranteed by \eqref{monotone_electric}.

\subsection{Minimization scheme}\label{sec:elec_field_min_scheme}

We apply the Algorithm of Section \ref{algorithm} to the energy \eqref{eqn:discrete energy_with_Eext}, except we need the following variational derivatives of $\Eext^h$.  The first order variation of $\Eext^h$ in the direction $\vv_h \in \Yh (\vn_h^k) \cap \Hbdy{\bdyvn}$ at the director variable $\vn_h^k$ reads
\begin{equation}\label{first_order_discrete_variation_dvN_electric}
\begin{split}
 \delta_{\vn_h} \Eext^h [s_h^k, \vn_h^k ; \vv_h] &= \Ecoef \ipelec(s_h^k, \vn_h^k, \vv_h).
\end{split}
\end{equation}
The first order variation of $\Eext^h$ in the direction $z_h \in \Sh \cap \Hbdy{\bdys}$ at the degree of orientation variable $s_h^k$ is
\begin{equation}\label{first_order_discrete_variation_dS_electric}
\begin{split}
 \delta_{s_h} \Eext^h [s_h^k , \vn_h^{k} ; z_h] &= \frac{\Ecoef}{2} \left( \ebar \ga \iO |\vE|^2 z_h - \ea \iO I_h \left[ (\vE \cdot \vn_h^{k})^2 z_h \right] \right).
\end{split}
\end{equation}

\subsection{Simulations}\label{sec:elec_field_sims}

We present simulations of the classic Freedericksz transition and the effect of an electric field on the shape of the Saturn ring defect.

\subsubsection{Freedericksz transition}\label{sec:Freedericksz}

We consider a two dimensional cell with no colloids present, i.e. $\anchorcoef = 0$.  The domain is defined to be $\Om = [0, 1]^2 \subset \R^2$, with $\Gm_{D} = \{ 0, 1 \} \times [0, 1]$, and $\Gm_{N} = \dOm \setminus \overline{\Gm_{D}}$.  The boundary conditions are given by
\begin{equation}\label{eqn:Freed_BCs}
\begin{split}
  \vn &= (0,1)\tp, \quad s = s^*, ~ \text{ on } \Gm_{D}, \\
  (\vnu \cdot \nabla) \vn &= \vzero, \quad  (\vnu \cdot \nabla) s = 0, ~ \text{ on } \Gm_{N},
\end{split}
\end{equation}
where $\vnu$ is the outer normal vector of $\dOm$.  Moreover, the
double well potential is defined in \eqref{eqn:double_well_defn}.  The
initial conditions in $\Om$ for the gradient flow are:
\[
s = s^*,
\qquad
\vn = (10^{-2},1)\tp / |(10^{-2},1)|.
\]
They are chosen to perturb the minimizing pair $\vn = (0,1)\tp$,
$s = s^*$ without the electric field.

The equilibrium solution, for $\kappa = 1.0$, is shown in Figure \ref{fig:Electric_Field_2D_all}, with electric field parameters as follows: $\Ecoef = 16.0$, $\vE = (1, 0)\tp$, $\ebar = 1.0$, $\ea = 2.0$, $\ga = 0.5$.  The director field deflects toward the right to better align with the imposed electric field $\vE$, which is the expected response known as the Freedericksz transition.  In this case, $0.6995 \leq s \leq 0.7757$ so the role of $s$ is not so critical because there is no defect region.
\begin{figure}
\begin{center}
\subfloat{

\includegraphics[width=2.5in]{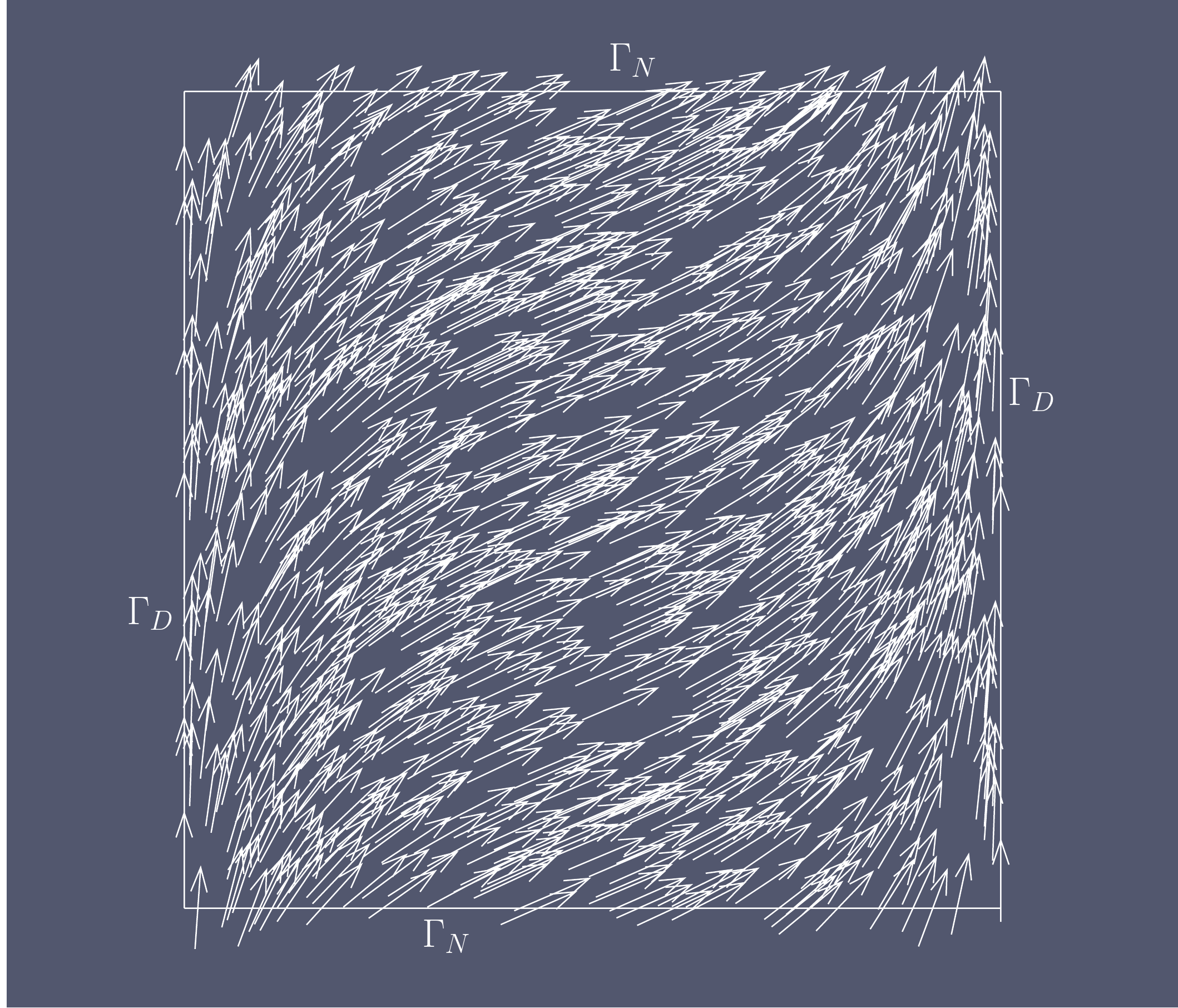}
}
\hspace{0.1in}
\subfloat{

\includegraphics[width=2.5in]{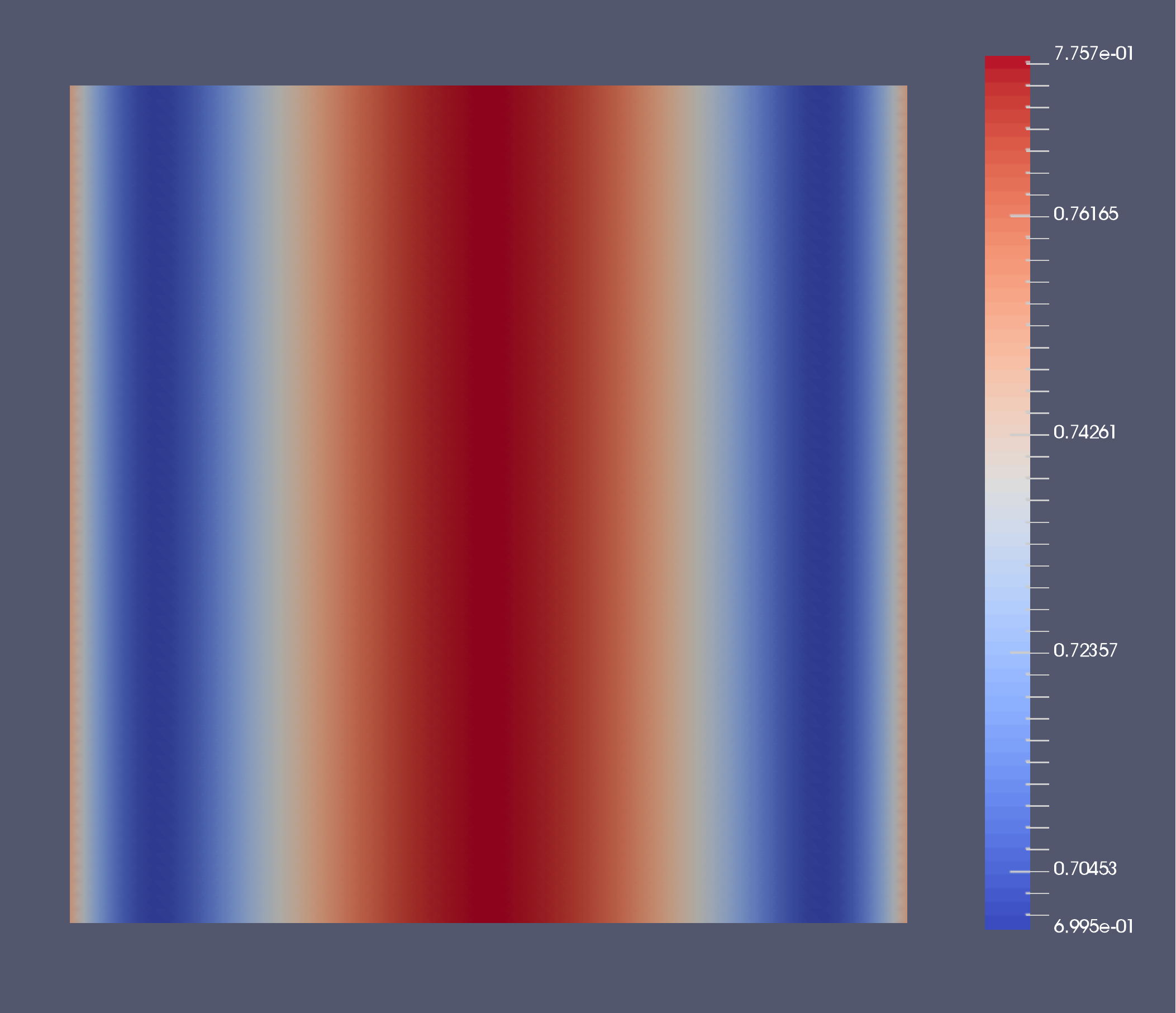}
}
\caption{Simulation results for the Freedericksz transition (Section \ref{sec:Freedericksz}) with Dirichlet boundary conditions on the left and right of $\Om$. The director field is shown (white arrows) at equilibrium with the given electric field $\vE = (1,0)\tp$.  The degree-of-orientation is shown on the right.}
\label{fig:Electric_Field_2D_all}
\end{center}
\end{figure}

\subsubsection{Saturn ring interaction with an electric field}\label{sec:saturn_ring_elec_field}

We consider the interaction of an electric field with a colloidal particle in three dimensions.  The domain is defined to be $\Om = [0, 1]^3 \subset \R^3$, with $\Gm_{D} = \dOm$.  The placement of the colloidal sphere and the boundary conditions are the same as in Section \ref{sec:defects-weak-anchoring} corresponding to Figure \ref{fig:Saturn_Ring_Phase_Field_all}, i.e. recall the description in Figure \ref{fig:LC_Sphere_Inclusion_BCs_B}. The double well potential is defined in \eqref{eqn:double_well_defn}.  The weak anchoring parameters are the same as in Section \ref{sec:defects-weak-anchoring}.  The electric field parameters are given as follows: $\Ecoef = 160.0$, $\vE = (0, 1, 0)\tp$, $\ebar = 1.0$, $\ea = 2.0$, $\ga = 0.5$.

The equilibrium solution, for $\kappa = 1.0$, is shown in Figure \ref{fig:Saturn_Ring_Electric_Field_all}.  The Saturn ring defect changes significantly (i.e. breaks into four pieces) because $\Ecoef$ is so large.  Note that the solution without the electric field is given in Figure \ref{fig:Saturn_Ring_Phase_Field_all}.
\begin{figure}
\begin{center}
\subfloat{

\includegraphics[width=2.5in]{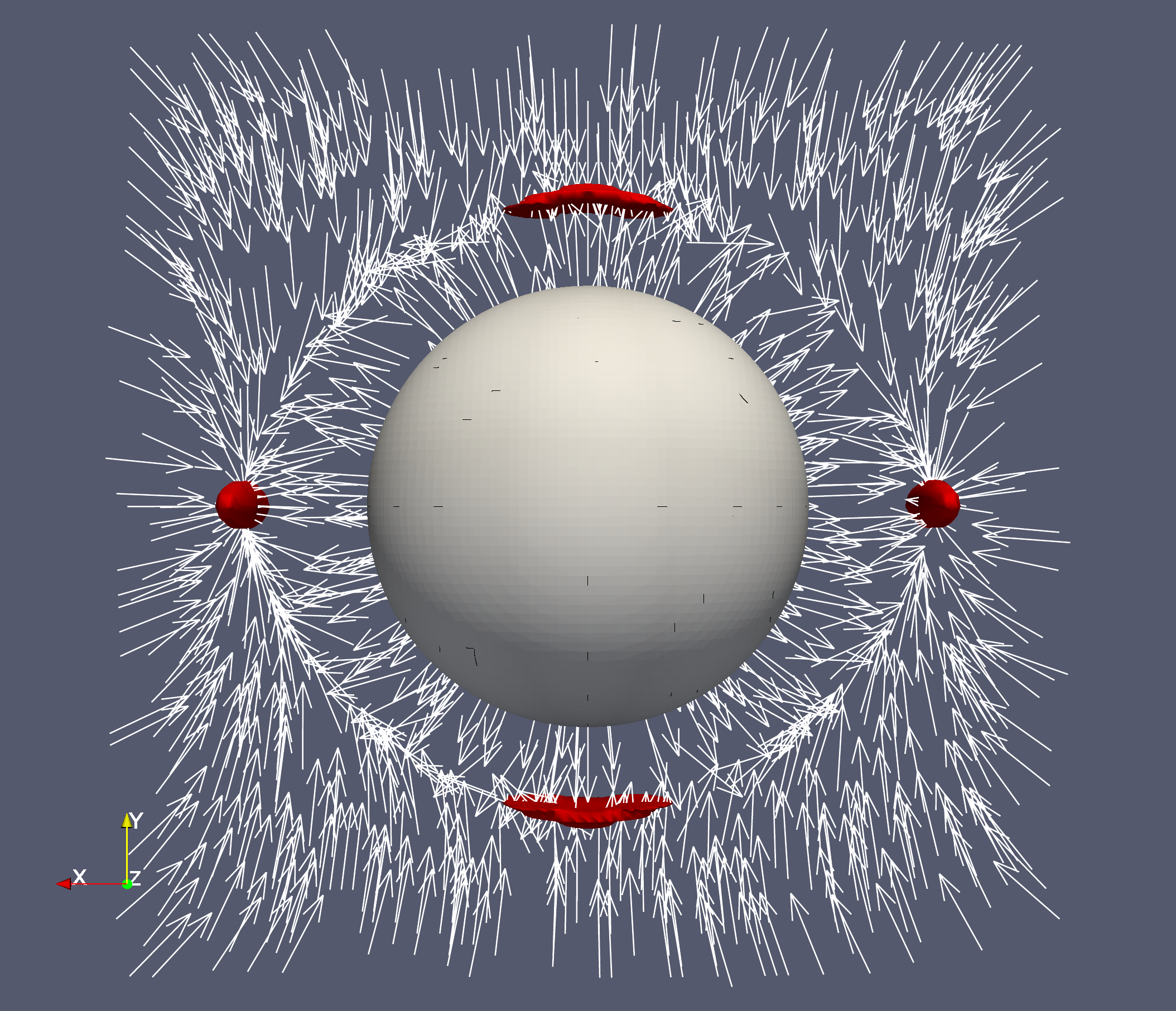}
}
\hspace{0.1in}
\subfloat{

\includegraphics[width=2.5in]{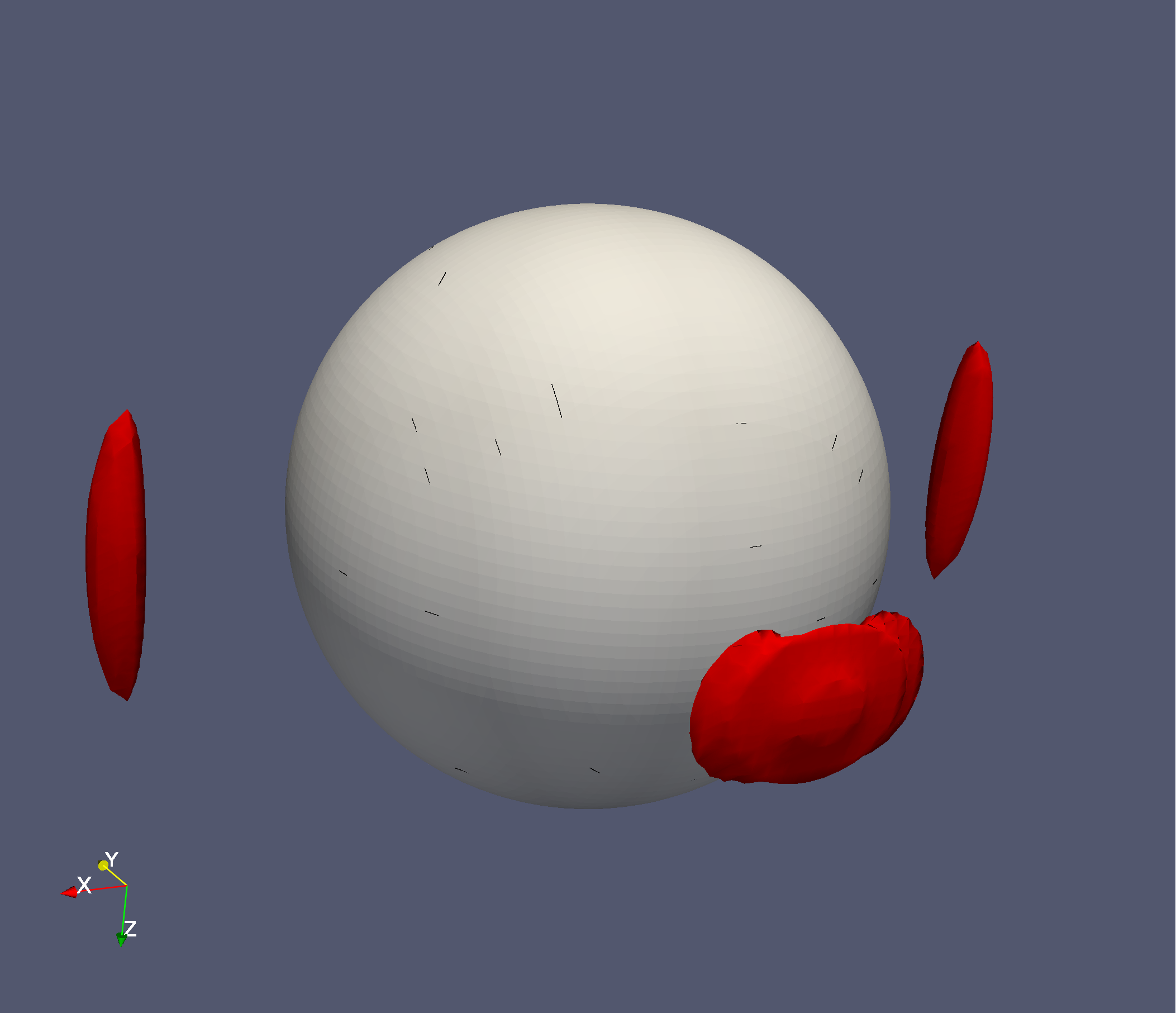}
}
\caption{Simulation results for deforming the Saturn ring with an electric field (Section \ref{sec:saturn_ring_elec_field}).  The director field is shown (white arrows) at equilibrium with the given electric field $\vE = (0,1,0)\tp$.  The $s=0.15$ iso-surface is shown in red.  The ring defect changes into two sliver-like defects and two disk-like defects.  The location of the sliver (disk) defect is at a radius of $\approx 0.38$ ($\approx 0.36$); the spherical colloid has a radius of $0.25$.}
\label{fig:Saturn_Ring_Electric_Field_all}
\end{center}
\end{figure}

\section{Discrete quasi-gradient flow algorithm}\label{sec:general_scheme}

We demonstrate that our minimization scheme in Section \ref{algorithm}
monotonically decreases the following discrete total energy
\begin{equation}\label{discrete_energy_total}
  E^h[s_h, \vn_h] := E_1^h [s_h, \vn_h] + E_2^h [s_h] + \Ea^h[s_h, \vn_h] + \Eext^h [s_h, \vn_h ],
\end{equation}
which includes the colloidal and electric field effects.

\subsection{Discrete minimizing movements}\label{sec:min_movements}

The method of minimizing movements \cite[pg. 32]{DeGiorgi_collection2006} is rather convenient for
ensuring energy decrease even for non-strictly convex energies.
Note that this is the case of $\vn_h\mapsto E^h[s_h,\vn_h]$ for fixed
$s_h$ when the latter vanishes at one or more nodes as well as the
energies \eqref{energy-with-phi} which are degenerate because of the
presence of the phase variable $\phi$.
This approach could be applied to more general non-convex energies
as well.

We present the idea for an abstract energy $E:\mathcal{H}\to\R$ where $\mathcal{H}$ is a
Hilbert space with norm $\|\cdot\|$ and inner product $\langle\cdot,\cdot\rangle.$
We construct a sequence of iterates $\{u^k\}_{n=0}^\infty \subset \mathcal{H}$
as follows: choose $u^0 \in \mathcal{H}$ arbitrarily and consider
minimizing the augmented functional
\[
F[u] := E[u] + \rho \|u-u^k\|^2
\]
for $\rho$ sufficiently large so that $F$ is {\it strictly convex}. Whether
this is possible depends on the specific structure of $E$, but note
that any $\rho>0$ would work for our two examples in this paper.
Let $u^{k+1}\in\mathcal{H}$ be the unique minimizer of $F$. Then
\begin{equation}\label{energy_decrease}
E[u^{k+1}] < F[u^{k+1}] = E[u^k] + \rho \|u^{k+1}-u^k\|^2
\le F[u^k] = E[u^k],
\end{equation}
provided $u^{k+1}\ne u^k$. This means that we achieve strict energy
decrease unless we reach a stationary point. Convergence of $u^k$ to a
local minimizer of $E$ is a delicate matter and within the context of
$\Gamma$-convergence. We elaborate on both energy decrease and
convergence for our concrete functionals below.

\subsection{Energy decreasing property}\label{sec:energy_property}

We now capitalize on the preceding calculations to show the following
key result, which extends Theorem \ref{energydecreasing} proved
in \cite{Nochetto_SJNA2017}.
Given $(s_h^k,\vn_h^k)\in\Sh\times\Nh$, we modify the
discrete total energy of \eqref{discrete_energy_total} as follows
for any $\rho \ge 0$:
\begin{equation}\label{modified_energy}
F^h[s_h,\vn_h] := E^h[s_h,\vn_h] + \rho \|\vn_h-\vn_h^k\|_{L^2(\Omega)}^2.
\end{equation}
\begin{thm}[energy decrease]\label{thm:energydecreasing_general}
Let $\Tk_h$ satisfy \eqref{weakly-acute}.
Given $(s_h^k,\vn_h^k)\in\Sh\times\Nh$,
the iterate $(s_h^{k+1}, \vn_h^{k+1})\in\Sh\times\Nh$
of the Algorithm of Section \ref{algorithm}
for \eqref{modified_energy} exists and satisfies
\begin{equation*}
  E^h [s_h^{k+1} ,\vn_h^{k+1}] \leq E^h [s_h^k ,\vn_h^k ] - \frac{1}{\delta t} \int_{\Omega} (s_h^{k+1} - s_h^k)^2 dx.
\end{equation*}
Equality holds if and only if $(s_h^{k+1}, \vn_h^{k+1}) = (s_h^k,\vn_h^k)$ (equilibrium state).
\end{thm}

\begin{proof}
Steps (a) and (b) show monotonicity, whereas Step (c) proves strict decrease of the energy.

{\it Step (a): Minimization.}
Since $F^h[s_h^k, \vn_h]$ is convex in $\vn_h$ for fixed
  $s_h^k$, there exists a tangential variation $\vt_h^k$ which
  minimizes $F^h[s_h^k, \vn_h^k + \vv_h]$ amongst all tangential
variations $\vv_h$. Invoking \eqref{energy_decrease} we deduce
\begin{equation*}
  E^h [s_h^k , \vn_h^k + \vt_h^k] \le E^h [s_h^k , \vn_h^k].
\end{equation*}

{\it Step (b): Projection.} Since the mesh $\Tk_h$ is weakly acute, we claim that
\begin{align*}
  \vn_h^{k+1}=\frac{\vn_h^k + \vt_h^k}{|\vn_h^k + \vt_h^k|}
  \quad\Rightarrow\quad
  E^h[s_h^k , \vn_h^{k+1}] \leq E^h[s_h^k , \vn_h^k + \vt_h^k].
\end{align*}
First we show that
\begin{align*}
  \vn_h^{k+1}=\frac{\vn_h^k + \vt_h^k}{|\vn_h^k + \vt_h^k|}
  \quad\Rightarrow\quad
  E_1^h \big[s_h^k , \vn_h^{k+1} \big] \leq E_1^h \big[s_h^k , \vn_h^k + \vt_h^k \big].
\end{align*}
Following \cite{Alouges_SJNA1997, Bartels_SJNA2006}, let $\vv_h = \vn_h^k + \vt_h^k$, $\vw_h = \frac{\vv_h}{|\vv_h|}$,
and observe that $|\vv_h|\ge1$ (at the nodes) and $\vw_h$ is well-defined.
By \eqref{discrete_energy_E1} (definition of the discrete energy), we only need to show that
\begin{align*}
  k_{ij} \frac{(s_i^k)^2 + (s_j^k)^2 }{2} |\vw_h(x_i) - \vw_h(x_j)|^2 \leq  k_{ij} \frac{(s_i^k)^2 + (s_j^k)^2 }{2} |\vv_h(x_i) - \vv_h(x_j)|^2 .
\end{align*}
for all $x_i,x_j\in \Nk_h$. Because $k_{ij}\ge0$ for $i\ne j$, this is equivalent to showing that
$ |\vw_h(x_i) - \vw_h(x_j)| \leq |\vv_h(x_i) - \vv_h(x_j)|$. This
follows from the fact that the mapping $\va \mapsto \va/|\va|$ defined
on $\{\va \in \mathbb R^d: |\va| \geq 1 \}$
is Lipschitz continuous with constant $1$.
Note that equality above holds if and only if $\vn_h^{k+1} = \vn_h^{k}$
or equivalently $\vt_h^k = \vzero$.

Next, from \eqref{monotone_anchoring} and \eqref{eqn:colloid_weak_anchoring_energy_discrete}, we get
\begin{equation*}
  \Ea^h[s_h^k , \vn_h^{k+1}] \le \Ea^h[s_h^k , \vn_h^k + \vt_h^k].
\end{equation*}
Moreover, using \eqref{monotone_electric} and \eqref{eqn:electric_energy_discrete}, we also get
\begin{equation*}
  \Eext^h[s_h^k , \vn_h^{k+1}] \le \Eext^h[s_h^k , \vn_h^k + \vt_h^k].
\end{equation*}
Therefore, we find that
\begin{equation*}
  E^h[s_h^k , \vn_h^{k+1}] \le E^h[s_h^k , \vn_h^k + \vt_h^k].
\end{equation*}

{\it Step (c): Gradient flow.}
Since $E_1^h$ is quadratic in terms of $s_h^k$, and
\[
2 s_h^{k+1} \big(s_h^{k+1} - s_h^k \big)
= \big(s_h^{k+1} - s_h^k \big)^2
+ \big|s_h^{k+1}\big|^2 - \big|s_h^k\big|^2,
\]
reordering terms gives
\begin{align*}
 E_1^h [ s_h^{k+1} ,\vn_h^{k+1} ] - E_1^h [ s_h^k ,\vn_h^{k+1} ]
= &\;
R_1
-  E^h_1 [ s_h^{k+1} - s_h^k ,\vn_h^{k+1}]
\leq R_1,
\end{align*}
where
\begin{equation*}
R_1 := \delta_{s_h} E^h_1 [ s_h^{k+1} ,\vn_h^{k+1} ; s_h^{k+1} - s_h^k].
\end{equation*}
Moreover, $\Ea^h$ is also quadratic in terms of $s_h^k$, so we get a similar inequality
\begin{equation*}
  \Ea^h [ s_h^{k+1} , \vn_h^{k+1} ] - \Ea^h [ s_h^k ,\vn_h^{k+1}  ] \leq R_{\mathrm{a}},
\end{equation*}
where
\begin{equation*}
R_{\mathrm{a}} := \delta_{s_h} \Ea^h [ s_h^{k+1} , \vn_h^{k+1} ; s_h^{k+1} - s_h^k].
\end{equation*}

Next, \eqref{convex_split_inequality} implies
\begin{align*}
E_2^h [ s_h^{k+1} ] - E_2^h [ s_h^k ]
=
\iO  \psi(s_h^{k+1}) dx  - \iO \psi(s_h^{k}) dx \leq R_2 := \delta_{s_h} E^h_2 [ s_h^{k+1} ; s_h^{k+1} - s_h^k],
\end{align*}
and accounting for the electric field gives
\begin{align*}
\Eext^h [ s_h^{k+1}, \vn_h^{k+1} ] - \Eext^h [ s_h^k, \vn_h^{k+1} ]
=
R_{\mathrm{ext}} := \delta_{s_h} \Eext^h [ s_h^{k+1} , \vn_h^{k+1}; s_h^{k+1} - s_h^k].
\end{align*}

Combining all estimates and invoking Step (c) of the Algorithm yields
\begin{align*}
  E^h [ s_h^{k+1} ,\vn_h^{k+1} ] - E^h [ s_h^k ,\vn_h^{k+1} ] \leq R_1 + R_2 + R_{\mathrm{a}} + R_{\mathrm{ext}}
  = - \frac {1}{\delta t} \int_{\Omega} (s_h^{k+1} - s_h^k)^2,
\end{align*}
which is the assertion.
Note finally that equality occurs if and only if $s_h^{k+1} = s_h^k$
and $\vn_h^{k+1} = \vn_h^k$, which corresponds to an equilibrium state.
This completes the proof.
\end{proof}

\begin{remark}\label{rem:convenient_use_min_move}

The choice $\rho>0$ in \eqref{modified_energy} ensures a positive
definite system to solve for $\vn_h^{k+1}$ no matter whether the system for $\rho=0$ is
singular or degenerate (see Remark \ref{rem:how_to_solve_degenerate_system}).

\end{remark}

\section{$\Gamma$-Convergence of the Discrete Energy}\label{sec:Gamma_conv_results}

We show that our discrete energy \eqref{discrete_energy_total} $\Gamma$-converges
to the continuous energy
\begin{equation}\label{energy_total}
  E[s, \vn] := E_1 [s, \vn] + E_2 [s] + \Ea[s, \vn] + \Eext [s, \vn ].
\end{equation}
This implies existence of global minimizers of
\eqref{energy_total}, and convergence of global minimizers of
\eqref{discrete_energy_total} to global minimizers of \eqref{energy_total}
along with convergence of discrete to continuous energies.

We recall the setting of our $\Gamma$-convergence result in
\cite{Nochetto_SJNA2017} and next extend it to the more general
energy \eqref{energy_total}. Let the continuous and discrete spaces be
\[
\X := L^2(\Omega) \times [L^2(\Omega)]^d,
\qquad
\X_h := \Sh \times \Uh.
\]
We define $E[s,\vn]$ as in \eqref{energy_total} for
$(s,\vu)\in\Admis(g,\vr)$ and $E[s,\vu]=\infty$ for
$(s,\vn)\in \X \setminus \Admis(g,\vr)$.
Likewise, we define $E^h [s_h ,\vn_h] $ as in
\eqref{discrete_energy_total} for $(s_h, \vu_h)\in \Admis_h (g_h, \vr_h)$
and $E^h [s, \vn] = \infty$ for all
$(s, \vu) \in \X \setminus \Admis_h (g_h, \vr_h)$.
We state the two properties of $\Gamma$-convergence for $E_1[s,\vn]$
\cite{Nochetto_SJNA2017}.

\begin{thm}[$\Gamma$-convergence]
  \label{Gamma_convergence}
 Let $\{ \Tk_h \}$ be a sequence of weakly acute meshes.
 Then, for every $(s, \vn) \in \X$ the following two properties hold:
 \begin{itemize}
  \item Lim-inf inequality: for every sequence $\{ (s_h, \vn_h) \}$ converging strongly to $(s, \vn)$ in $\X$, we have
    \begin{align}\label{liminf}
      E_1 [s, \vn] \leq \liminf_{h \rightarrow 0} E_1^h[s_h, \vn_h];
    \end{align}
  \item Lim-sup inequality: there exists a sequence $\{ (s_h, \vn_h) \}$ such that $(s_h, \vn_h)$ converges strongly to $(s, \vn)$ in $\X$ and
    \begin{align}
      \label{limsup}
      E_1 [s, \vn] \geq \limsup_{h\rightarrow 0} E_1^h [s_h, \vn_h].
    \end{align}
\end{itemize}
\end{thm}

We refer to \cite{Nochetto_SJNA2017} for a complete proof of this
rather technical theorem. We now give a brief outline.
The lim-sup inequality is a \emph{consistency} estimate
in the usual numerical analysis sense.
It reduces to showing that $\mathcal{E}_h,\widetilde{\mathcal{E}}_h
\rightarrow 0$ as $h \rightarrow 0$, in \eqref{eqn:energyequality} and
\eqref{eqn:residual}.  If $\vn\in[H^1(\Omega)]^d$, then the residual term
\eqref{eqn:residual} would be of order $h^2 \int_{\Om} | \nabla s_h |^2 dx$
which obviously converges to zero. The presence of defects entails
lack of $[H^1(\Omega)]^d$ regularity of $\vn$, whence
this heuristic argument fails. A rigorous proof involves a rather
delicate regularization argument of any pair $(s,\vu) \in \Admis(g,\vr)$
which preserves Dirichlet boundary values and the structure condition
$\vu=s\vn$ for some $\vn$ of unit norm away from the singular set $\Sing$.

Proving the lim-inf is more technical.
It follows from \eqref{abs_inequality}, which also reads
\[
E_1^h[s_h,\vn_h] \ge (\kappa-1) \int_\Omega|\nabla I_h|\widetilde{\vu}_h||^2 d\vx
+ \int_\Omega |\nabla \widetilde{\vu}_h|^2 d\vx
= \widetilde{E}^h_1[\widetilde{s}_h,\widetilde{\vu}_h],
\]
and the fact that $\widetilde{E}^h_1[\widetilde{s}_h,\widetilde{\vu}_h]$
is weakly lower semi-continuous
\cite[Lemma 3.4 (weak lower semicontinuity)]{Nochetto_SJNA2017}
This usually follows from convexity (with
respect to $\nabla \widetilde\vu_h$), but this is \textbf{not} obvious
when $0 < \kappa < 1$ and is a key contribution of \cite{Nochetto_SJNA2017}.

$\Gamma$-convergence combined with a coercivity property
yields that global minimizers of the discrete problem converge to
global minimizers of the continuous problem \cite{Braides_book2014, DalMaso_book1993}.
We explicitly show this property in Theorem \ref{thm:converge_numerical_soln}.
However, $\Gamma$-convergence does not yield rates of convergence.
In Section \ref{sec:eoc}, we provide some experimental rates of convergence.

\begin{thm}[convergence of global discrete minimizers]\label{thm:converge_numerical_soln}
Let $\{ \Tk_h \}$ satisfy \eqref{weakly-acute} and assume $\Ea$, $\Ea^h$ are given by
\eqref{eqn:colloid_weak_anchoring_energy}, \eqref{eqn:colloid_weak_anchoring_energy_discrete}.
If $(s_h,\vu_h) \in \Admis_h (g_h,\vr_h)$
is a sequence of global minimizers of $E^h[s_h,\vn_h]$ in \eqref{discrete_energy_total}, then
every cluster point is a global minimizer of the continuous energy
$E[s, \vn]$ in \eqref{energy_total}.
\end{thm}
\begin{proof}
We proceed in several steps.

1. {\it Coercivity}.
In view of \eqref{discrete_energy_total},
assume there is a constant $\Lambda>0$ such that
\[
\liminf_{h\to0} E^h[s_h,\vn_h] =
\liminf_{h\to0} \left( E_1^h[s_h,\vn_h] + E_2^h[s_h] + \Ea^h[s_h, \vn_h] + \Eext^h [s_h, \vn_h ] \right) \le \Lambda,
\]
for otherwise there is nothing to prove.
We apply \cite[Lemma 3.5 (coercivity)]{Nochetto_SJNA2017}
\begin{align*}
  E_1^h [s_h, \vn_h] \geq &\; \min\{\kappa, 1\} \iO |\nabla \widetilde{\vu}_h|^2 d\vx \geq \min\{\kappa, 1\} \iO |\nabla I_h |s_h||^2 d\vx,
\end{align*}
to extract subsequences (not relabeled)
$(\widetilde{s}_h,\widetilde{\vu}_h) \to (\widetilde{s},\widetilde{\vu})$
and $(s_h,\vu_h) \to (s,\vu)$ converging weakly in $[H^1(\Omega)]^{d+1}$,
strongly in $[L^2(\Omega)]^{d+1}$ and a.e. in $\Omega$.
We next invoke \cite[Lemma 3.6 (characterizing limits)]{Nochetto_SJNA2017}
to show that the limits satisfy the structure properties
\begin{equation}\label{structure}
\vu = s \vn,
\qquad
\widetilde{\vu} = \widetilde{s} \vn
\qquad\text{ in }
\Omega \setminus\Sing,
\end{equation}
for a suitable vector field $\vn$, with $|\vn|=1$, and such that
$\vn_h \to \vn$ strongly in $L^2(\Om \setminus \Sing)$ and a.e. in $\Om \setminus \Sing$.

2. {\it Lim-inf inequality.}
Using \cite[Lemma 3.4 (weak lower semicontinuity)]{Nochetto_SJNA2017} we deduce
\[
\widetilde{E}_1[\widetilde{s},\widetilde{\vu}]
= \iO (\kappa -1) | \nabla \widetilde{s} |^2 + | \nabla\widetilde{\vu} |^2 d\vx
\le \liminf_{h\to0}
\widetilde{E}_1^h[\widetilde{s}_h,\widetilde{\vu}_h]
\le \liminf_{h\to0} E_1^h[s_h,\vn_h],
\]
where the last inequality is a consequence of \eqref{abs_inequality}.
Since $s_h$ converges a.e. in $\Omega$ to $s$, so does
$\psi(s_h)$ to $\psi(s)$. Apply now Fatou's
lemma to write
\[
E_2[s]=\iO \psi(s) = \iO \lim_{h\to0} \psi(s_h)
\le \liminf_{h\to0} \iO \psi(s_h) = \liminf_{h\to0} E_2^h[s_h].
\]

We now consider the weak anchoring energy $\Ea^h[s_h,\vn_h]$ of
\eqref{eqn:colloid_weak_anchoring_energy_discrete}, i.e.
we show that
\begin{equation}\label{eqn:weak_anch_liminf_pt_1}
\iO s^2 \left[ |\vn|^2 |\nabla \phase|^2 - (\nabla \phase \cdot \vn)^2 \right] d\vx
 = \lim_{h\to0} \iO I_h \left\{ s_h^2 \left[ |\vn_h|^2 |\nabla \phase|^2 - (\nabla \phase \cdot \vn_h)^2 \right] \right\} d\vx.
\end{equation}
In view of \eqref{uh}, properties of the Lagrange interpolant yield
\begin{equation*}
  I_h \left\{ s_h^2 \left[ |\vn_h|^2 |\nabla \phase|^2 - (\nabla \phase \cdot \vn_h)^2 \right] \right\} = I_h \left[ |\vu_h|^2 |\nabla \phase|^2 - (\nabla \phase \cdot \vu_h)^2 \right].
\end{equation*}
Next, classic interpolation theory yields
\begin{equation*}
\begin{aligned}
  \big\| |\vu_h|^2 |\nabla \phase|^2 &- I_h[|\vu_h|^2 |\nabla
    \phase|^2] \big\|_{L^2(\Omega)}
\\
&\lesssim h \left\| |\vu_h| |\nabla \phase|^2 \nabla \vu_h \right\|_{L^2(\Omega)} +
h \left\| |\vu_h|^2 |\nabla \phase| \nabla^2 \phase \|_{L^2(\Omega)}
\lesssim h \| \vu_h \right\|_{H^1(\Omega)}.
\end{aligned}
\end{equation*}
Since $\| \vu_h \|_{H^1(\Omega)}$ is uniformly bounded,
$|\vu_h|^2 |\nabla \phase|^2 - I_h[|\vu_h|^2 |\nabla \phase|^2] \to 0$
in $L^2(\Omega)$ and a.e. in $\Omega$.  Similarly,
$(\nabla \phase \cdot \vu_h)^2 - I_h[(\nabla \phase \cdot \vu_h)^2] \to 0$
in $L^2(\Omega)$ and a.e. in $\Omega$.
Hence, using that $\vu_h\to\vu$ a.e. in
$\Omega$ and $|\vu_h|$ is uniformly bounded, combining the Lebesgue
dominated convergence theorem with \eqref{structure} implies the
following equivalent form of \eqref{eqn:weak_anch_liminf_pt_1}
\begin{equation}\label{eqn:convergence_pf_1}
\iO \left[ |\vu|^2 |\nabla \phase|^2 - (\nabla \phase \cdot \vu)^2\right] d\vx
= \lim_{h\to0} \iO \left[ |\vu_h|^2 |\nabla \phase|^2 - (\nabla \phase \cdot \vu_h)^2 \right] d\vx.
\end{equation}
Furthermore, since $\phi$ is smooth, the Lebesgue
dominated convergence theorem also gives
\begin{equation*}
\begin{split}
\iO |\nabla \phase|^2 (s(\vx) - s^*)^2 d\vx
= \lim_{h\to0} \iO |I_h \nabla \phase|^2 (s_h(\vx) - s^*)^2 d\vx.
\end{split}
\end{equation*}
Therefore, we obtain
\begin{equation}\label{eqn:weak_anch_liminf}
   \Ea[s,\vn] = \lim_{h\to0} \Ea^h[s_h, \vn_h].
\end{equation}

We proceed similarly for the electric energy \eqref{eqn:electric_energy}.
In view of \eqref{eqn:discrete_inner_prod_elec}, \eqref{eqn:electric_energy_discrete},
we show that
\begin{equation}\label{eqn:electric_energy_liminf_pt_1}
   \iO |\ea| |\vE|^2 - \ea s (\vE \cdot \vn)^2 d\vx =
   \lim_{h\to0} \iO I_h \left[ |\ea| |\vE|^2 |\vn_h|^2 - \ea s_h (\vE \cdot \vn_h)^2 \right] d\vx.
\end{equation}
First, we exploit that $|\vn_h|=1$ at the nodes to infer that
\begin{equation*}
\iO I_h \left[ |\vE|^2 |\vn_h|^2 \right] = \iO I_h \left[ |\vE|^2 \right]
\to \iO |\vE|^2, \quad\text{ as } h \to 0,
\end{equation*}
because $\vE$ is assumed to be smooth.  For the other term in
\eqref{eqn:discrete_inner_prod_elec}, nodal interpolation implies
\begin{equation*}
I_h \left[ s_h (\vE \cdot \vn_h)^2 \right] = I_h \left[ (\vE \cdot
  \vu_h) (\vE \cdot \vn_h) \right],
\end{equation*}
because $\vu_h=I_h[s_h\vn_h]$.
Standard interpolation theory on each element $T$ of $\Tk_h$ gives
\begin{equation*}
\begin{split}
  \| (\vE \cdot \vu_h) (\vE \cdot \vn_h) &- I_h[(\vE \cdot \vu_h) (\vE \cdot \vn_h)] \|_{L^1(T)} \lesssim h^2 \big( \| |\vu_h| |\vn_h| \|_{L^1(T)} \\
  & + \| |\vu_h| |\nabla \vn_h| \|_{L^1(T)} + \| |\nabla \vu_h| |\vn_h| \|_{L^1(T)} + \| |\nabla \vu_h| |\nabla \vn_h| \|_{L^1(T)} \big).
\end{split}
\end{equation*}
Summing over all $T \in \Tk_h$, and using Cauchy-Schwarz, we get
\begin{equation*}
  \| (\vE \cdot \vu_h) (\vE \cdot \vn_h) - I_h[(\vE \cdot \vu_h) (\vE \cdot \vn_h)] \|_{L^1(\Omega)}
  \lesssim h^2 \| \vu_h \|_{H^1(\Omega)} \| \vn_h \|_{H^1(\Omega)}.
\end{equation*}
Since $|\vn_h|\le 1$, an inverse estimate gives
$\| \nabla \vn_h\|_{L^2(\Omega)}\lesssim h^{-1}$, and so
\begin{equation*}
  \| (\vE \cdot \vu_h) (\vE \cdot \vn_h) - I_h[(\vE \cdot \vu_h) (\vE \cdot \vn_h)] \|_{L^1(\Omega)}
  \lesssim h \| \vu_h \|_{H^1(\Omega)} \to 0, ~ \text{ as } h \to 0.
\end{equation*}
Thus, we just need to show
\begin{equation}\label{eqn:electric_energy_liminf_pt_2}
   \iO (\vE \cdot \vu) (\vE \cdot \vn) d\vx =
   \lim_{h\to0} \iO (\vE \cdot \vu_h) (\vE \cdot \vn_h) d\vx.
\end{equation}
We decompose the integral into the singular set
$\Sing = \{s=0\} = \{ \vu = \vzero \}$ and the complement and use the
Lebesgue dominated convergence theorem upon realizing that
$(\vE \cdot \vu_h) (\vE \cdot \vn_h)$ is uniformly bounded.
Since $\vu_h\to\vu$ a.e. in $\Omega$, we obtain
\[
\lim_{h\to0} \int_{\Sing} (\vE \cdot \vu_h) (\vE \cdot \vn_h) d\vx
= \int_{\Omega} \chi_\Sing \lim_{h\to0} (\vE \cdot \vu_h) (\vE \cdot \vn_h) d\vx = 0
\]
In addition, we utilize \cite[Lemma 3.6 (characterizing limits)]{Nochetto_SJNA2017}
to deduce that $\vn_h\to\vn$ a.e. in $\Omega \setminus \Sing$, whence
\begin{equation*}
  \lim_{h\to0} \int_{\Omega \setminus \Sing} (\vE \cdot \vu_h) (\vE \cdot \vn_h) d\vx
  = \int_{\Omega \setminus \Sing} (\vE \cdot \vu) (\vE \cdot \vn) d\vx.
\end{equation*}
Collecting the above results,
and recalling that $\vu = s \vn$, we obtain \eqref{eqn:electric_energy_liminf_pt_1}.
Finally, the Lebesgue dominated convergence theorem implies
\begin{equation*}
  \iO (1 - s \ga) |\vE|^2 d\vx = \lim_{h\to0} \iO (1 - s_h \ga) |\vE|^2 d\vx,
\end{equation*}
whence
\begin{equation}\label{eqn:electric_energy_liminf}
   \Eext[s,\vn] = \lim_{h\to0} \Eext^h[s_h,\vn_h].
\end{equation}

Consequently, we arrived at
\begin{equation}\label{final-lim-inf}
\widetilde{E}_1[\widetilde{s},\widetilde{\vu}] + E_2[s] + \Ea[s,\vn] + \Eext[s,\vn]
\le \liminf_{h\to0} E^h[s_h,\vn_h].
\end{equation}

3 {\it Lim-sup inequality.}
This is a consistency inequality. Since we have to use Lagrange
interpolation, and so point values, we first need to invoke a regularization procedure.
Given $\epsilon>0$ arbitrary, we resort to
\cite[Proposition 3.2 (regularization of functions in
$\Admis_h(g_h,\vr_h)$)]{Nochetto_SJNA2017} to find a pair
$(t_\epsilon,\vv_\epsilon)\in \Admis(g,\vr)\cap[W^1_\infty(\Omega)]^{d+1}$ such that
\begin{equation}\label{min-epsilon}
E[t_\epsilon,\vm_\epsilon] \le
\inf_{(t,\vm)\in\Admis(g,\vr)} E[t,\vm] + \epsilon \leq E[s,\vn] + \epsilon,
\end{equation}
where $\vm_\epsilon := t_\epsilon^{-1} \vv_\epsilon$ if
$t_\epsilon \neq 0$ or otherwise $\vm_\epsilon$ is an arbitrary unit vector.
Let $(t_{\epsilon,h},\vv_{\epsilon,h})\in \Admis_h(g_h,\vr_h)$ and
$\vm_{\epsilon,h} \in \Nh$ be the Lagrange interpolants of
$(t_\epsilon,\vv_\epsilon,\vm_\epsilon)$ and apply
\cite[Lemma 3.3 (lim-sup inequality)]{Nochetto_SJNA2017} to
$(t_\epsilon,\vv_\epsilon)$ and $\vm_\epsilon$ to write
\[
E_1[t_\epsilon,\vm_\epsilon] = \lim_{h\to0} E_1^h [t_{\epsilon,h},\vm_{\epsilon,h}].
\]
Moreover, \cite[Theorem 3.7 (convergence of global discrete minimizers)]{Nochetto_SJNA2017}
shows that
\[
E_2[t_\epsilon] = \int_\Omega \lim_{h\to0}\psi(t_{\epsilon,h})
= \lim_{h\to0} \int_\Omega \psi(t_{\epsilon,h}) = \lim_{h\to0} E_2^h[t_{\epsilon,h}].
\]

We now consider the weak anchoring energy \eqref{eqn:colloid_weak_anchoring_energy_discrete}, and observe that
\begin{equation}\label{eqn:weak_anch_limsup_pt_1}
  I_h \left\{ t_{\epsilon,h}^2 \left[ |\vm_{\epsilon,h}|^2 |\nabla \phase|^2 - (\nabla \phase \cdot \vm_{\epsilon,h})^2 \right] \right\} =
  I_h \left\{ \left[ |\vv_{\epsilon, h}|^2 |\nabla \phase|^2 - (\nabla \phase \cdot \vv_{\epsilon, h})^2 \right] \right\}
\end{equation}
because of the definition of Lagrange interpolant and $\vv_\epsilon=t_\epsilon\vm_\epsilon$.
Hence, following a similar argument as in \eqref{eqn:convergence_pf_1},
we find that \eqref{eqn:weak_anch_limsup_pt_1} converges in $L^2(\Omega)$ as $h \to 0$
because $\vv_\epsilon\in [W^1_\infty(\Omega)]^d$.
Therefore, the convergence
\begin{equation*}
\begin{split}
  \Ea[t_{\epsilon}, \vm_{\epsilon}] = \lim_{h\to0} \Ea^h[t_{\epsilon,h},\vm_{\epsilon,h}]
\end{split}
\end{equation*}
follows in a similar fashion as the convergence of \eqref{eqn:weak_anch_liminf}.  Also, since
\begin{equation}\label{eqn:weak_anch_limsup_pt_2}
\begin{split}
  \iO |I_h \nabla \phase|^2 (t_{\epsilon,h} - s^*)^2 \to \iO |\nabla \phase|^2 (t_{\epsilon} - s^*)^2, ~ \text{ as } h \to 0,
\end{split}
\end{equation}
by standard interpolation theory,
taking \eqref{eqn:weak_anch_limsup_pt_1} and \eqref{eqn:weak_anch_limsup_pt_2} together,
we get
\begin{equation*}
\begin{split}
  \Ea[t_{\epsilon}, \vm_{\epsilon}] = \lim_{h\to0} \Ea^h[t_{\epsilon,h},\vm_{\epsilon,h}].
\end{split}
\end{equation*}

For the electric energy \eqref{eqn:electric_energy_discrete},
the definition of the Lagrange interpolant again implies
\begin{equation}\label{eqn:electric_energy_limsup_pt_1}
\begin{split}
   I_h \left[ |\ea| |\vE|^2 |\vm_{\epsilon,h}|^2 - \ea t_{\epsilon,h} (\vE \cdot \vm_{\epsilon,h})^2 \right] =
   I_h \left[ |\ea| |\vE|^2 - \ea t_{\epsilon} (\vE \cdot \vm_{\epsilon})^2 \right].
\end{split}
\end{equation}
The first term in \eqref{eqn:electric_energy_limsup_pt_1} clearly converges
to $|\epsilon_a| |\vE|^2$ in $L^1(\Omega)$.
For the second term $\xi_\epsilon:=t_\epsilon(\vE\cdot\vm_\epsilon)^2$,
take $\delta > 0$ arbitrary, define
$\Sing_\delta := \{|t_\epsilon| \le \delta\}$, and note that
\[
x_i\in\Nk_h: \quad |x-x_i| \le Ch \quad\Rightarrow\quad
|t_\epsilon(x)-t_\epsilon(x_i)|\le C_\epsilon h.
\]
Let $h$ be small, depending on $\epsilon$ and $\delta$, so that
$C_\epsilon h\le\frac{\delta}{2}$.
If $x\in\Sing_\delta$, then $t_\epsilon(x_i) \le \frac32\delta$ and
\[
\int_{\Sing_\delta} \big| \xi_\epsilon - I_h \xi_\epsilon \big|
\le C_\epsilon \delta.
\]
On the other hand, if $x \notin \Sing_\delta$, then $t_\epsilon(x_i)\ge \frac12\delta$
and $\xi_\epsilon$ is Lipschitz in $\Omega\setminus\Sing_{\frac{\delta}{2}}$ with constant
$C_{\epsilon,\delta}$. Therefore
\[
\int_{\Omega\setminus\Sing_\delta} \big| \xi_\epsilon - I_h \xi_\epsilon \big|
\le C_{\epsilon,\delta} h.
\]
Taking the limits, first as $h\to0$ and next as $\delta\to0$, we
infer that
\[
\lim_{h\to0} \int_\Omega I_h\xi_\epsilon dx = \int_\Omega \xi_\epsilon dx
\]
which implies convergence of the second term in \eqref{eqn:electric_energy_limsup_pt_1}.
Moreover, since
$
  \iO t_{\epsilon,h} |\vE|^2 d\vx \to \iO t_{\epsilon} |\vE|^2 d\vx,
$
as $h\to0$, we obtain
\begin{equation*}
  \Eext[t_{\epsilon}, \vm_{\epsilon}] = \lim_{h\to0} \Eext^h [t_{\epsilon,h},\vm_{\epsilon,h}].
\end{equation*}
Collecting the preceding estimates we end up with the lim-sup equality
\begin{equation}\label{final-lim-sup}
E[t_\epsilon,\vv_\epsilon] = \lim_{h\to0} E^h[t_{\epsilon,h},\vv_{\epsilon,h}].
\end{equation}

4 {\it Convergence of energy.}
We observe that
$\nabla\widetilde{\vu}=\nabla\widetilde{s}\otimes\vn+\widetilde{s}\nabla\vn$
a.e. in $\Omega\setminus\Sing$, whence
\[
\widetilde{E}_1[\widetilde{s},\widetilde{\vu}]=
\int_{\Omega\setminus\Sing} \kappa|\nabla\widetilde{s}|^2 +
|\widetilde{s}|^2 |\nabla\vn|^2
= \int_{\Omega\setminus\Sing} \kappa|\nabla s|^2 + |s|^2 |\nabla\vn|^2
= E_1[s,\vn],
\]
and note that $\Om \setminus \Sing$ can be replaced by $\Om$ because
$\nabla s = \vzero$ on $\Sing$.
Therefore, in view of \eqref{final-lim-inf}, \eqref{min-epsilon} and
\eqref{final-lim-sup}, we arrive at
\begin{equation*}
E[s,\vn] \le \liminf_{h\to0} E^h[s_h,\vn_h] \le \limsup_{h\to0} E^h[s_h,\vn_h]
\le \lim_{h\to0} E^h[t_{\epsilon,h},\vn_{\epsilon,h}] = E[t_\epsilon,\vv_\epsilon]
\le E[s,\vn] + \epsilon.
\end{equation*}
Finally, letting $\epsilon \to 0$, we see that the pair $(s,\vn)$ is a
global minimizer of $E$ and $E[s,\vn]=\lim_{h\to0} E^h[s_h,\vn_h]$, as asserted.
This concludes the proof.
\end{proof}

It remains to show the $\Gamma$-convergence when the discrete weak anchoring energy \eqref{eqn:colloid_weak_anchoring_energy_discrete} is replaced by the penalized Dirichlet
energy \eqref{discrete-dirichlet-penalty}.

\begin{cor}[convergence of global discrete minimizers]\label{cor:converge_numerical_soln_penalized_Dirichlet}
Let $\{ \Tk_h \}$ satisfy \eqref{weakly-acute} and assume $\Ea$, $\Ea^h$ are given by
\eqref{penalty}, \eqref{discrete-dirichlet-penalty}.
If $(s_h,\vu_h) \in \Admis_h (g_h,\vr_h)$
is a sequence of global minimizers of $E^h[s_h,\vn_h]$ in \eqref{discrete_energy_total}, then
every cluster point is a global minimizer of the continuous energy
$E[s, \vn]$ in \eqref{energy_total}.
\end{cor}
\begin{proof}
Following the proof of Theorem \ref{thm:converge_numerical_soln}, we
only need to show that the lim-inf and lim-sup
inequalities hold for the anchoring energy \eqref{discrete-dirichlet-penalty}.

{\bf Step 1. Lim-inf inequality:} Thanks to Step 1 in Theorem \ref{thm:converge_numerical_soln}, any $(s_h, \vu_h) \to (s, \vu)$ converging in $\mathbb{X}$, there exists a subsequence $(s_h, \vu_h) \to (s, \vu)$ converging weakly in $[H^1(\Omega)]^{d+1}$, strongly in $[L^2(\Omega)]^{d+1}$ and a.e. in $\Omega$.
To prove the limit-inf equality, we note that
\begin{align*}
 \Ea^h[s_h,\vn_h] &= \frac{\anchorcoef}{2} C_0 \epsilon \left( \int_\Omega I_h \left\{ s_h^2 \big| |\nabla \phase| \vn_h - \nabla \phase \big|^2  \right\}
  + \int_{\Omega} | I_h \nabla \phase|^2 (s_h - g_h)^2 \right)
\\
 &=
  \frac{\anchorcoef}{2} C_0 \epsilon \left( \int_\Omega I_h
\left\{   \big| T_1^h \big|^2  \right\}
  + \int_{\Omega} (T_2^h)^2
\right),
\end{align*}
where
\begin{align*}
T_1^h := |\nabla \phase| \vu_h - s_h \nabla \phase
\quad \mbox{and} \quad
T_2^h := | I_h \nabla \phase| (s_h - g_h).
\end{align*}
For the first term $T_1^h$, since $\vu_h$ and $s_h$ converge to $\vu$ and $s$ in $L^2(\Omega)$, we note that
\[
T_1^h \to T_1 :=  |\nabla \phase| \vu - s \nabla \phase
\quad
\mbox{ in $L^2(\Omega)$ as $h \to 0$.}
\]
Therefore, $\int_{\Omega} (T_1^h)^2 \to \int_{\Omega} (T_1)^2$ which implies that (similar to what is done in \eqref{eqn:convergence_pf_1})
\[
\int_{\Omega} I_h \{(T_1^h)^2 - (T_1)^2 \} \leq C \int_{\Omega} (T_1^h)^2 - (T_1)^2 \to 0
\quad
\mbox{ as $h \to 0$.}
\]
Moreover, since $\int_{\Omega} I_h \{ (T_1)^2 \} \to \int_{\Omega} (T_1)^2 $ as $h \to 0$, we obtain that
\begin{align}\label{eqn:T1}
\int_{\Omega} I_h \left\{ (T_1^h)^2 \right\} \to \int_{\Omega} (T_1)^2 .
\end{align}

For the second term $T_2^h$, we have
\[
T_2^h \to T_2 := | \nabla \phase | (s - g)
\quad
\mbox{ a.e. as $h \to 0$ and}
\quad
| T_2^h | \leq 2 \max_{\Omega} |\nabla \phase|.
\]
By the Lebesgue dominated convergence theorem, we have
\[
\int_{\Omega} (T_2^h)^2 \to  \int_{\Omega} (T_2)^2.
\]
Combining this with \eqref{eqn:T1}, we infer that
\begin{align*}
\lim_{h \to 0} E_a^h[s_h, \vu_h] = E_a[s, \vu].
\end{align*}

{\bf Step 2. Lim-sup inequality:}
We follow step 3 in the proof of Theorem \ref{thm:converge_numerical_soln} and set $(t_{\epsilon}, \vv_{\epsilon}) \in \mathbb{A}(g, \vr) \cap [W^1_{\infty}(\Omega)]^{d+1}$ such that $(t_{\epsilon}, \vv_{\epsilon}) \to (s, \vu)$  weakly in $[H^1(\Omega)]^{d+1}$, strongly in $[L^2(\Omega)]^{d+1}$ and a.e. in $\Omega$.
Let $(t_{\epsilon,h}, \vv_{\epsilon,h})$ be the Lagrange interpolants of $(t_{\epsilon}, \vv_{\epsilon})$, then $(t_{\epsilon,h}, \vv_{\epsilon,h}) \to (t_{\epsilon}, \vv_{\epsilon})$  strongly in $[L^2(\Omega)]^{d+1}$ and a.e. in $\Omega$. By a similar procedure as before, we are able to show that
\[
\lim_{h \to 0} E_a^h[t_{\epsilon,h}, \vv_{\epsilon,h}] = E_a[t_{\epsilon}, \vv_{\epsilon}].
\]
This concludes the proof.
\end{proof}

\section{Conclusions}\label{sec:conclusion}

We present a robust finite element method for the Ericksen energy
that models nematic liquid crystals with variable degree of
orientation.  This is augmented by additional energy terms to model
colloidal effects and electric fields.  We present several
simulations to illustrate the diverse range of phenomena that can be
captured by our method, e.g. interesting defect structures (such as
the Saturn ring) as well as the ability to modulate the defect
structures with external fields.  We prove a monotone energy
decreasing property for our quasi-gradient flow method (applied to
\eqref{discrete_energy_total}) which hinges on a mass-lumping strategy for the auxiliary energy terms $\Ea^h$ and $\Eext^h$.  Furthermore, we provide a full $\Gamma$-convergence proof of our discrete energy \eqref{discrete_energy_total} to the original continuous energy \eqref{energy_total}.

The following are possible extensions of this work: modeling of liquid
crystal droplets, i.e. by coupling the Ericksen energy to
Cahn-Hilliard; coupling with full electro-statics with or without
charge transport, as well as including electro-dynamics to model
liquid crystal laser devices;
and also optimizing colloidal particle distributions by actuating the
liquid crystal medium.
Furthermore, we plan on extending our method to handle the full $\vQ$-tensor model.

\bigskip

\noindent
\textbf{Acknowledgements:}
R. H. Nochetto and W. Zhang acknowledge
financial support by the NSF via
DMS-1411808. S. W. Walker acknowledges financial support by the NSF
via DMS-1418994 and DMS-1555222 (CAREER).
Moreover, R.H. Nochetto acknowledges support by the Institut Henri
Poincar\'e (Paris) and the Hausdorff Institute (Bonn), whereas
W. Zhang acknowledges support by the Brin post-doctoral fellowship at the University of Maryland.



\bigskip\noindent
\textbf{References}

\bibliographystyle{elsarticle-num}


\end{document}